\documentclass[12pt]{amsart}
\usepackage{amsmath,amsthm,amsfonts}
\usepackage[utf8]{inputenc}
\usepackage[usenames, dvipsnames]{color}
\usepackage{hyperref}
\usepackage{multicol}
\usepackage{amscd}
\usepackage{faktor}
\usepackage{bbm}
\usepackage{bm}
\usepackage{pb-diagram}
\usepackage{amssymb}
\usepackage{tikz-cd}
\usepackage{braket}
\usepackage{amsfonts}
\usepackage{amsthm}
\usepackage{enumitem}
\usepackage{cancel}
\usepackage[english]{babel}
\usepackage[margin=1in]{geometry}
\usepackage{soul}
\input xy
\usepackage[all]{xy}

\newcommand{\cC}{{\mathcal C}}

\newcommand{\G}{{\mathcal G}}
\newcommand{\cH}{\mathcal{H}}

\newcommand{\cL}{{\mathcal L}}

\newcommand{\Id}{\operatorname{Id}}

\newcommand{\Co}{\operatorname{Col}}

\newcommand{\Z}{{\mathbb Z}}

\newcommand{\C}{{\mathbb{C} }}

\usepackage{comment}
\newcommand{\commentt}[1]

\theoremstyle{plain}
\numberwithin{equation}{section}

\newtheorem{theorem}{Theorem}[section]

\newtheorem{corollary}[theorem]{Corollary}
\newtheorem{proposition}[theorem]{Proposition}

\theoremstyle{definition}
\newtheorem{definition}[theorem]{Definition}
\newtheorem{example}[theorem]{Example}
\theoremstyle{remark}
\newtheorem{remark}[theorem]{Remark}

\title[Twisted quantum double as LTO]{Twisted Kitaev Quantum Double model as local topological order}

\author[S. Cui]{Shawn X. Cui}
\address{ Department of Mathematics, Department of Physics and Astronomy, Purdue University, West Lafayette, IN 47906, United States}
\email{cui177@purdue.edu}

\author[C. Galindo]{C\'esar Galindo}
\address{ Departamento de Matem\'aticas, Universidad de los Andes, Bogot\'a, Colombia}
\email{cn.galindo1116@uniandes.edu.co}

\author[D. Romero]{Diego Romero}
\address{ Departamento de Matem\'aticas, Universidad de los Andes, Bogot\'a, Colombia}
\email{da.romero12@uniandes.edu.co}

\begin{document}

\begin{abstract} 
We study the Twisted Kitaev Quantum Double model within the framework of Local Topological Order (LTO). We extend its definition to arbitrary 2D lattices, enabling an explicit characterization of the ground state space through the invariant spaces of monomial representations. We reformulate the LTO conditions to include general lattices and prove that the twisted model satisfies all four LTO axioms on any 2D lattice. As a corollary, we show that its ground state space is a quantum error-correcting code.
\end{abstract}

\maketitle

\section{Introduction}\label{intro}

Topological phases of matter are gapped quantum systems distinguished by long-range entanglement, ground state degeneracy dependent on the topology of the underlying manifold, protected gapless edge modes, and fractionalized quasi-particle excitations known as anyons \cite{wen2004quantum, nayak2008non}. These characteristics make topological phases promising candidates for robust quantum information storage and processing, as they inherently resist local perturbations—an essential quality for building large-scale fault-tolerant quantum computers \cite{kitaev2003fault, freedman2002modular}.

A significant class of two-dimensional topological phases can be effectively modeled using exactly solvable spin lattice models, where local Hilbert spaces are associated with vertices or edges of a lattice, and the Hamiltonian is constructed as a sum of commuting local projectors. Notable among these models are the Levin-Wen string-net models, which utilize a unitary fusion category as input \cite{levin2005string}, and the Kitaev quantum double models, based on finite groups, including the toric code as a special case for the group $\mathbb{Z}_2$ \cite{kitaev2003fault}. Extensions of Kitaev's model, such as the Twisted Kitaev Quantum Double model incorporating a cohomological twist \cite{twisted}, and generalizations employing finite-dimensional $C^*$-Hopf algebras \cite{buerschaper2013hierarchy, yan2022ribbon}, further enrich the landscape of exactly solvable models. Both the Levin-Wen and Kitaev families realize anomaly-free topological phases characterized by the Drinfeld center of fusion categories \cite{buerschaper2009mapping}.

An axiomatic formulation of \emph{topological order in spin lattice models} was provided by Bravyi and Hastings \cite{Bravyi_2010, bravyi2011short}, capturing the essential features of topological phases. They demonstrated that quantum systems satisfying these axioms exhibit gap stability under local perturbations. Although it is widely believed that both the Levin-Wen and Kitaev quantum double models satisfy these axioms, rigorous mathematical proofs have only been established for specific instances \cite{shaw2020, qiu2020ground, naaijkens2012anyons}.

Recently, Jones, Naaijkens, Penney, and Wallick \cite{LTO} extended the framework of topological order by proposing a set of axioms for \emph{local topological order} (LTO), formulated through nets of local ground state projections. Beyond the standard conditions of topological order, the LTO axioms introduce additional requirements via operators that act along the boundary of lattice regions. This approach leads to a local net of boundary $C^*$-algebras, situated in a lattice dimension one lower than that of the bulk. These boundary nets are instrumental in the bulk-boundary correspondence: the bulk topological order is recovered as the category of DHR bimodules of the boundary algebra. Notably, \cite{LTO} demonstrates that both the Levin-Wen model and the toric code satisfy the LTO axioms, wherein the boundary nets correspond to fusion categorical structures closely associated with subfactor theory.  Subsequently, in \cite{chuah2024boundary}, it was shown that the Kitaev quantum double model, for any finite group, also fulfills the LTO axioms, generalizing results from the toric code. For a finite group \( G \), the boundary algebra is identified as the fusion categorical net of \(\text{Hilb}(G)\) with generator \(\mathbb{C}[G]\) or \(\text{Rep}(G)\) with generator \(\mathbb{C}^G\), depending on the boundary’s rough or smooth type. It is essential to note that LTO conditions are stated in terms of the microscopic Hamiltonians, meaning that although the Levin-Wen and Kitaev's models realize equivalent topological orders, the fulfillment of LTO conditions for one model does not imply that for the other.

In this paper, we study the Twisted Kitaev Quantum Double model within the framework of local topological order. This model, which takes as input a finite group $G$ and a third cohomology class $\alpha \in H^3(G, U(1))$, realizes the topological order associated with the Drinfeld center of $\mathrm{Hilb}^\alpha(G)$—the category of $G$-graded finite-dimensional Hilbert spaces with an associativity constraint twisted by $\alpha$.
Unlike the original Kitaev model, which can be defined on an arbitrary 2D lattice, the twisted model has thus far been formulated only on triangulated lattices \cite{twisted}—a limitation from the perspective of topological order theory. Our first contribution is to extend the definition of the twisted model to arbitrary 2D lattices (Section~\ref{sec:general lattice}), providing greater flexibility for studying the model. Among other advantages, this allows for an explicit characterization of the ground state space in terms of monomial spaces and their monomial representations (Section~\ref{GS}). We then reformulate the LTO conditions from \cite{LTO} for arbitrary lattices (Section~\ref{sec:local quantum smooth}), extending their applicability beyond square lattices. Our main result is the following theorem.

\begin{theorem} The Twisted Kitaev Quantum Double model, defined on an arbitrary 2D lattice, satisfies the set of four LTO axioms \cite{LTO}. \end{theorem}
 
We provide a detailed proof of this theorem for the case of smooth boundaries in Section~\ref{sec:local quantum smooth} and make the necessary adjustments to accommodate rough boundaries in Section~\ref{LTOROUGH}.

As a corollary, we establish that the ground state space of the Twisted Kitaev Quantum Double model defines a quantum error-correcting code. An interesting direction for future research is to explore encoding and decoding protocols, as well as error thresholds within this model, and to compare its performance with other topological codes, such as the toric code and the Abelian Kitaev model \cite{cui2024abelian}. We leave this investigation for future work.
\smallbreak
The paper is organized as follows. In Section~\ref{sec:preliminaries}, we provide the necessary background on combinatorial surfaces and revisit the Twisted Quantum Double model developed in \cite{twisted}. In Section~\ref{sec:general lattice}, we extend the definition of the Twisted Quantum Double model to general lattices, detailing the modifications required for non-triangular configurations and the associated operators. Section~\ref{GS} presents an algorithm to construct a basis for the ground state space, leveraging the fact that the vertex operators are monomial operators.

In Section~\ref{sec:local quantum smooth}, we study the axiomatization of Local Topological Order (LTO), originally proposed in \cite{LTO} for lattices where all faces are squares, and adapt this axiomatization to arbitrary lattices. In Section~\ref{SECCIONLTO}, we develop and prove the conditions under which the Twisted Quantum Double model for finite groups satisfies the LTO axioms for regions with smooth boundaries. By focusing on smooth boundaries, we facilitate a streamlined proof that bypasses the added complexities introduced by rough boundary considerations. In Section~\ref{LTOROUGH}, we extend these results to address the rough boundary case, which involves additional constraints due to the behavior of the twisted model near these boundaries.

\section{Preliminaries}\label{sec:preliminaries}

\subsection{Combinatorial Surfaces}

In this article, a \emph{surface} refers to a compact, oriented, piecewise-linear (PL) 2-dimensional manifold, assumed to have no boundary unless specified otherwise. In dimension 2, PL manifolds and topological manifolds are equivalent categories. For an oriented surface $\Sigma$, we denote its boundary, with the induced orientation, by $\partial \Sigma$.

A \emph{lattice} is a topological graph $\G$ formed by taking a disjoint union of intervals $[0,1]$ and identifying their endpoints. The intervals map to \emph{edges} of $\G$, denoted by $E(\G)$, and the endpoints map to \emph{vertices}, denoted by $V(\G)$.

A \emph{skeleton} of a surface $\Sigma$ is a lattice $\mathcal{L}$ embedded in $\Sigma$, where each vertex has at least two incident edges and every connected component of $\Sigma \setminus \mathcal{L}$ is an open disk. For instance, the vertex-edge structure of any triangulation of $\Sigma$ forms a skeleton.

\subsection{Twisted Quantum Double model}\label{TM}

We review the construction of a discrete model for 2D topological phases, known as the the Twisted Kitaev Quantum Double model (or Twisted Quantum Double model), introduced in \cite{twisted}. This model is defined by a finite group $G$, a 3-cocycle $\alpha \in Z^3(G, U(1))$, and an oriented surface $\Sigma$ equipped with a $\Delta$-complex structure.

To define the model, we consider a \emph{triangular lattice} $\mathcal{L}$, which is the 1-skeleton of a $\Delta$-complex structure on $\Sigma$. Specifically, $\mathcal{L}$ is a lattice where every face is bounded by exactly three edges. The vertices $V(\mathcal{L})$ correspond to the 0-simplices of the $\Delta$-complex, and the edges $E(\mathcal{L})$ correspond to its 1-simplices. The 2-simplices themselves are called the faces $\mathcal{F}(\mathcal{L})$. Each face is adjacent to exactly three edges (see \cite{hatcher} for the formal definition of a $\Delta$-complex).

Given a $\Delta$-complex of an oriented closed surface, we lift it to its universal cover. We then take the $\Delta$-subcomplex corresponding to a fundamental domain within the universal cover. Since the universal cover of an oriented closed surface is $\mathbb{R}^2$, this $\Delta$-subcomplex lies in the plane. We label the vertices of the $\Delta$-subcomplex with a total order $1 < 2 < \ldots < n$, and orient each edge to point from the vertex with the smaller label to the one with the larger label (see Figure~\ref{graph1}).

\begin{figure}[h!]
    \centering
\tikzset{every picture/.style={line width=0.75pt}} 

\begin{tikzpicture}[x=0.6pt,y=0.6pt,yscale=-1,xscale=1]

\draw    (140,74) -- (100,174) ;
\draw    (100,174) -- (211,109) ;
\draw    (100,174) -- (200,166) ;
\draw    (100,174) -- (152,259) ;
\draw    (140,74) -- (211,109) ;
\draw    (211,109) -- (200,166) ;
\draw    (152,259) -- (200,166) ;
\draw    (200,166) -- (314,175) ;
\draw    (152,259) -- (314,175) ;
\draw    (211,109) -- (314,175) ;
\draw    (314,175) -- (375,244) ;
\draw    (211,109) -- (308,71) ;
\draw    (308,71) -- (314,175) ;
\draw    (308,71) -- (379,88) ;
\draw    (375,244) -- (308,71) ;
\draw    (375,244) -- (379,88) ;
\draw    (379,88) -- (424,169) ;
\draw    (375,244) -- (424,169) ;
\draw    (152,259) -- (375,244) ;
\draw    (140,74) -- (308,71) ;
\draw  [fill={rgb, 255:red, 0; green, 0; blue, 0 }  ,fill opacity=1 ] (118.42,127.68) -- (117.44,118.55) -- (125.71,122.1) -- cycle ;
\draw  [fill={rgb, 255:red, 0; green, 0; blue, 0 }  ,fill opacity=1 ] (152.1,143.61) -- (156.53,135.57) -- (161.27,143.22) -- cycle ;
\draw  [fill={rgb, 255:red, 0; green, 0; blue, 0 }  ,fill opacity=1 ] (152.23,169.94) -- (159.81,164.77) -- (160.58,173.74) -- cycle ;
\draw  [fill={rgb, 255:red, 0; green, 0; blue, 0 }  ,fill opacity=1 ] (228,72.44) -- (220.07,77.06) -- (219.93,68.06) -- cycle ;
\draw  [fill={rgb, 255:red, 0; green, 0; blue, 0 }  ,fill opacity=1 ] (258.84,90.03) -- (253.2,97.27) -- (249.73,88.96) -- cycle ;
\draw  [fill={rgb, 255:red, 0; green, 0; blue, 0 }  ,fill opacity=1 ] (205.55,136.23) -- (208.74,144.83) -- (199.85,143.42) -- cycle ;
\draw  [fill={rgb, 255:red, 0; green, 0; blue, 0 }  ,fill opacity=1 ] (265.82,144.23) -- (256.67,143.5) -- (261.69,136.03) -- cycle ;
\draw  [fill={rgb, 255:red, 0; green, 0; blue, 0 }  ,fill opacity=1 ] (174.1,216.02) -- (173.94,206.84) -- (181.86,211.12) -- cycle ;
\draw  [fill={rgb, 255:red, 0; green, 0; blue, 0 }  ,fill opacity=1 ] (130.11,223.36) -- (122.14,218.81) -- (129.86,214.19) -- cycle ;
\draw  [fill={rgb, 255:red, 0; green, 0; blue, 0 }  ,fill opacity=1 ] (229.54,219.01) -- (234.2,211.1) -- (238.72,218.89) -- cycle ;
\draw  [fill={rgb, 255:red, 0; green, 0; blue, 0 }  ,fill opacity=1 ] (253.09,170.43) -- (244.7,174.15) -- (245.56,165.19) -- cycle ;
\draw  [fill={rgb, 255:red, 0; green, 0; blue, 0 }  ,fill opacity=1 ] (311.27,126.99) -- (306.24,119.32) -- (315.22,118.7) -- cycle ;
\draw  [fill={rgb, 255:red, 0; green, 0; blue, 0 }  ,fill opacity=1 ] (340.6,154.72) -- (333.47,148.94) -- (341.84,145.63) -- cycle ;
\draw  [fill={rgb, 255:red, 0; green, 0; blue, 0 }  ,fill opacity=1 ] (347.09,212.55) -- (338.48,209.36) -- (345.34,203.54) -- cycle ;
\draw  [fill={rgb, 255:red, 0; green, 0; blue, 0 }  ,fill opacity=1 ] (376.84,170) -- (372.66,161.82) -- (381.66,162.18) -- cycle ;
\draw  [fill={rgb, 255:red, 0; green, 0; blue, 0 }  ,fill opacity=1 ] (279.65,250.29) -- (272.11,255.52) -- (271.26,246.56) -- cycle ;
\draw  [fill={rgb, 255:red, 0; green, 0; blue, 0 }  ,fill opacity=1 ] (406.81,138.57) -- (399.18,133.47) -- (407.2,129.4) -- cycle ;
\draw  [fill={rgb, 255:red, 0; green, 0; blue, 0 }  ,fill opacity=1 ] (399.49,206.96) -- (399.61,197.78) -- (407.4,202.3) -- cycle ;
\draw  [fill={rgb, 255:red, 0; green, 0; blue, 0 }  ,fill opacity=1 ] (179.11,93.23) -- (169.95,93.83) -- (173.84,85.71) -- cycle ;
\draw  [fill={rgb, 255:red, 0; green, 0; blue, 0 }  ,fill opacity=1 ] (339.6,78.61) -- (348.4,76) -- (346.4,84.77) -- cycle ;
\draw (130,57.4) node [anchor=north west][inner sep=0.75pt]    {$1$};
\draw (377,247.4) node [anchor=north west][inner sep=0.75pt]    {$10$};
\draw (147,259.4) node [anchor=north west][inner sep=0.75pt]    {$9$};
\draw (424,161.4) node [anchor=north west][inner sep=0.75pt]    {$7$};
\draw (308,180.4) node [anchor=north west][inner sep=0.75pt]    {$8$};
\draw (196,169.4) node [anchor=north west][inner sep=0.75pt]    {$2$};
\draw (87,164.4) node [anchor=north west][inner sep=0.75pt]    {$5$};
\draw (204,88.4) node [anchor=north west][inner sep=0.75pt]    {$4$};
\draw (377,70.4) node [anchor=north west][inner sep=0.75pt]    {$3$};
\draw (299,52.4) node [anchor=north west][inner sep=0.75pt]    {$6$};
\end{tikzpicture}
    \caption{A portion of a triangular lattice (or $\Delta$-complex), where each edge is oriented according to the ordering of the vertices.}
    \label{graph1}
\end{figure}
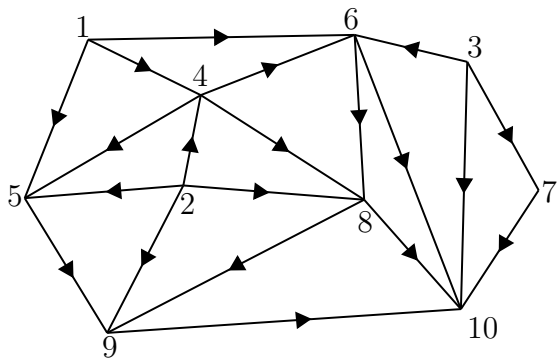

The ordering must be consistent across boundary edges of the fundamental domain, ensuring proper identification when reconstructing the $\Delta$-complex on the surface. For instance, in Figure~\ref{Dmaintorus}~(left), we show a $\Delta$-complex for the fundamental domain of a torus. When this $\Delta$-complex is mapped onto the torus, pairs of opposite edges are identified, so the orientations must align with these identifications. In Figure~\ref{Dmaintorus}~(right), we present one possible ordering of vertices that fulfills this consistency requirement.
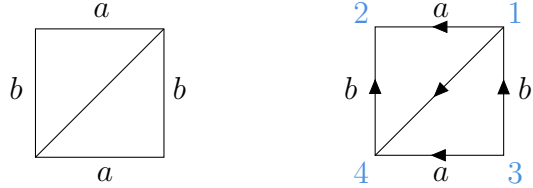
\begin{figure}[h!]
    \centering
\begin{tikzpicture}[x=0.6pt,y=0.6pt,yscale=-1,xscale=1]

\draw   (27,27) -- (108,27) -- (108,108) -- (27,108) -- cycle ;
\draw    (27,108) -- (108,27) ;
\draw    (241,107) -- (322,26) ;
\draw   (241.15,25.75) -- (322.15,25.75) -- (322.15,106.75) -- (241.15,106.75) -- cycle ;
\draw  [fill={rgb, 255:red, 0; green, 0; blue, 0 }  ,fill opacity=1 ] (277.9,25.73) -- (286.39,22.25) -- (286.39,29.22) -- cycle ;
\draw  [fill={rgb, 255:red, 0; green, 0; blue, 0 }  ,fill opacity=1 ] (276.9,106.73) -- (285.39,103.25) -- (285.39,110.22) -- cycle ;
\draw  [fill={rgb, 255:red, 0; green, 0; blue, 0 }  ,fill opacity=1 ] (241.14,71.73) -- (237.66,63.24) -- (244.63,63.24) -- cycle ;
\draw  [fill={rgb, 255:red, 0; green, 0; blue, 0 }  ,fill opacity=1 ] (322.14,71.98) -- (318.66,63.49) -- (325.63,63.49) -- cycle ;
\draw  [fill={rgb, 255:red, 0; green, 0; blue, 0 }  ,fill opacity=1 ] (278.5,69.5) -- (282.04,61.03) -- (286.97,65.96) -- cycle ;

\draw (62,9) node [anchor=north west][inner sep=0.75pt]   [align=left] {$a$};
\draw (64,112.5) node [anchor=north west][inner sep=0.75pt]   [align=left] {$a$};
\draw (9,56) node [anchor=north west][inner sep=0.75pt]   [align=left] {$b$};
\draw (112,56) node [anchor=north west][inner sep=0.75pt]   [align=left] {$b$};
\draw (226,9) node [anchor=north west][inner sep=0.75pt]  [color={rgb, 255:red, 74; green, 144; blue, 226 }  ,opacity=1 ] [align=left] {$2$};
\draw (323,9) node [anchor=north west][inner sep=0.75pt]  [color={rgb, 255:red, 74; green, 144; blue, 226 }  ,opacity=1 ] [align=left] {$1$};
\draw (226,110) node [anchor=north west][inner sep=0.75pt]  [color={rgb, 255:red, 74; green, 144; blue, 226 }  ,opacity=1 ] [align=left] {$4$};
\draw (323,110) node [anchor=north west][inner sep=0.75pt]  [color={rgb, 255:red, 74; green, 144; blue, 226 }  ,opacity=1 ] [align=left] {$3$};
\draw (276,9) node [anchor=north west][inner sep=0.75pt]   [align=left] {$a$};
\draw (276,112.5) node [anchor=north west][inner sep=0.75pt]   [align=left] {$a$};
\draw (220,56) node [anchor=north west][inner sep=0.75pt]   [align=left] {$b$};
\draw (329.75,56) node [anchor=north west][inner sep=0.75pt]   [align=left] {$b$};

\end{tikzpicture}
    \caption{Ordering of the vertices in a lattice on the torus, consistent with the identification of edges.}
    \label{Dmaintorus}
\end{figure}

We fix a finite group $G$ and define a $G$-\emph{coloring} of $\mathcal{L}$ as a map $\varphi: E(\mathcal{L}) \to G$ assigning an element of $G$ to each edge of $\mathcal{L}$. The element $\varphi(l) \in G$ assigned to an edge $l$ is called its $\varphi$-color. We denote by $\operatorname{Col}(\mathcal{L})$ the set of all $G$-colorings of $\mathcal{L}$.

The Hilbert space associated with $(\Sigma, \mathcal{L}, G)$ is defined as
\[
\mathcal{H}_{\text{tot}} := \operatorname{Span}\{\, \ket{\varphi} \mid \varphi \in \operatorname{Col}(\mathcal{L}) \,\},
\]
which is a finite-dimensional Hilbert space with an orthonormal basis given by all possible $G$-colorings of $\mathcal{L}$.

For a coloring $\varphi$ and an edge $l$ with endpoints $v$ and $u$ where $v < u$, we denote the $\varphi$-color of $l$ as $[vu]_\varphi$, or simply $[vu]$ when $\varphi$ is clear from context. Additionally, we define $[uv] := [vu]^{-1}$ for $v < u$.

\begin{definition}[Face Operator] \label{defface}
Let $f$ be a (triangular) face of $\mathcal{L}$, and let $\varphi$ be a $G$-coloring of $\mathcal{L}$. The face operator $B_f$ depends only on the $G$-colors of the edges forming the boundary $\partial f$. An oriented surface $\Sigma$  induces a counterclockwise cyclic order $<_c$ on the vertices of each face $f$ when drawn on a plane and denoted as $(v_1 <_c v_2 <_c v_3)$. The operator $B_f$ acts as the identity on $\ket{\varphi}$ if $[v_1v_2]_\varphi [v_2v_3]_\varphi [v_3v_1]_\varphi = e$, where $e$ is the identity element of $G$, and as zero otherwise (see Figure~\ref{facetwit}).
\end{definition}

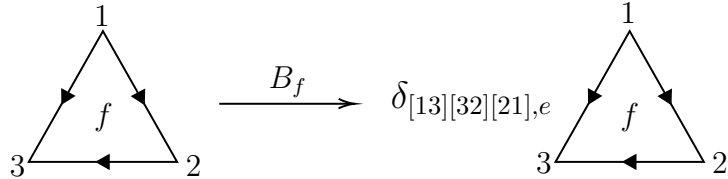
\begin{figure}[h!]
    \centering

\tikzset{every picture/.style={line width=0.75pt}} 

\begin{tikzpicture}[x=0.55pt,y=0.55pt,yscale=-1,xscale=1]

\draw   (100,31) -- (151,121) -- (49,121) -- cycle ;
\draw  [fill={rgb, 255:red, 0; green, 0; blue, 0 }  ,fill opacity=1 ] (72.09,80.51) -- (71.96,71.33) -- (79.87,75.64) -- cycle ;
\draw  [fill={rgb, 255:red, 0; green, 0; blue, 0 }  ,fill opacity=1 ] (128.12,80.79) -- (120.4,75.81) -- (128.37,71.61) -- cycle ;
\draw  [fill={rgb, 255:red, 0; green, 0; blue, 0 }  ,fill opacity=1 ] (96,121) -- (104,116.5) -- (104,125.5) -- cycle ;
\draw    (180,81) -- (271,81) ;
\draw [shift={(273,81)}, rotate = 180.62] [color={rgb, 255:red, 0; green, 0; blue, 0 }  ][line width=0.75]    (10.93,-3.29) .. controls (6.95,-1.4) and (3.31,-0.3) .. (0,0) .. controls (3.31,0.3) and (6.95,1.4) .. (10.93,3.29)   ;
\draw   (462.25,31) -- (513.25,121) -- (411.25,121) -- cycle ;
\draw  [fill={rgb, 255:red, 0; green, 0; blue, 0 }  ,fill opacity=1 ] (434.34,80.51) -- (434.21,71.33) -- (442.12,75.64) -- cycle ;
\draw  [fill={rgb, 255:red, 0; green, 0; blue, 0 }  ,fill opacity=1 ] (490.37,80.79) -- (482.65,75.81) -- (490.62,71.61) -- cycle ;
\draw  [fill={rgb, 255:red, 0; green, 0; blue, 0 }  ,fill opacity=1 ] (458.25,121) -- (466.25,116.5) -- (466.25,125.5) -- cycle ;

\draw (92,11.4) node [anchor=north west][inner sep=0.75pt]    {$1$};
\draw (34,114.4) node [anchor=north west][inner sep=0.75pt]    {$3$};
\draw (155,114.4) node [anchor=north west][inner sep=0.75pt]    {$2$};
\draw (212,53.4) node [anchor=north west][inner sep=0.75pt]    {$B_f$};
\draw (92,78.4) node [anchor=north west][inner sep=0.75pt]    {$f$};
\draw (454,78.4) node [anchor=north west][inner sep=0.75pt]    {$f$};
\draw (454,9.4) node [anchor=north west][inner sep=0.75pt]    {$1$};
\draw (396,112.4) node [anchor=north west][inner sep=0.75pt]    {$3$};
\draw (517,112.4) node [anchor=north west][inner sep=0.75pt]    {$2$};
\draw (296,63.4) node [anchor=north west][inner sep=0.75pt]    {\large{$\delta_{[13][ 32][ 21],e}$}};
\end{tikzpicture}
\caption{The action of operator $B_f$.}
\label{facetwit}
\end{figure}

The second type of operators are the vertex operators, which depend on a normalized 3-cocycle \(\alpha \in Z^3(G, U(1))\), a function \(\alpha: G^3 \to U(1)\) satisfying the following conditions for all \(g_1, g_2, g_3, g_4 \in G\):

\begin{enumerate}
    \item \(\alpha(g_2, g_3, g_4)\, \alpha(g_1, g_2 g_3, g_4)\, \alpha(g_1, g_2, g_3) = \alpha(g_1 g_2, g_3, g_4)\, \alpha(g_1, g_2, g_3 g_4)\),
    \item \(\alpha(e, g_1, g_2) = \alpha(g_1, e, g_2) = \alpha(g_1, g_2, e) = 1\),
\end{enumerate}
where \( e \) is the identity element of \( G \).

The Hamiltonian of the model is
\[
H = \sum_{v \in V(\mathcal{L})} \left(1 - A_v\right) + \sum_{f \in F(\mathcal{L})} \left(1 - B_f\right).
\]

\begin{definition}
A matrix \( M \) is called \emph{monomial} if each row and each column contains exactly one nonzero entry. This means \( M \) can be expressed as the product of a permutation matrix \( P \) and an invertible diagonal matrix \( D \). An operator \( A \) is \emph{monomial} if there exists a basis \( B \) such that its matrix representation \( [A]_B \) is monomial.
\end{definition}

\begin{definition}[Vertex Operator] \label{defvertex}
Let \( v \) be a vertex, and let \( s(v) \) denote its star, the set of edges connected to \( v \). The operator \( A_v \) is defined as
\[
A_v := \frac{1}{|G|} \sum_{g \in G} A_v^g,
\]
where each \( A_v^g = D_v^{g,\alpha} P_v^g \) is a monomial operator with respect to the basis of \( G \)-colorings. Here, \( P_v^g \) is a permutation matrix, and \( D_v^{g,\alpha} \) is an invertible diagonal matrix depending on \( g \) and the 3-cocycle \( \alpha \).

The permutation \( P_v^g \) acts on a coloring \( \varphi \in \operatorname{Col}(\mathcal{L}) \) by modifying the colors of edges in \( s(v) \) as follows:

\begin{itemize}
    \item If an edge points away from \( v \), multiply its color on the left by \( g \).
    \item If an edge points towards \( v \), multiply its color on the right by \( g^{-1} \).
\end{itemize}

The diagonal matrix \( D_v^{g,\alpha} \) has diagonal entries given by the function \( \alpha_v^g: \operatorname{Col}(\mathcal{L}) \to U(1) \), which depends on \( \alpha \).

To define \( \alpha_v^g \), we introduce a new vertex \( v' \) such that \( v' \) satisfies \( v' < v \) but is greater than all vertices less than \( v \). We assign the edge \( [v'v] = g \). For any vertex \( u \) connected to \( v \), we define
\[
[uv'] := [uv][vv'].
\]

Let \( \nabla_v \) be the set of faces \( f \) in \( \mathcal{L} \) such that \( v \in \partial f \). For each \( f \in \nabla_v \), let \( u_1 < u_2 < u_3 < u_4 \) be the vertices of \( \partial f \cup \{ v' \} \) in ascending order. We define
\[
\zeta(f) := \alpha\big([u_1 u_2], [u_2 u_3], [u_3 u_4]\big).
\]

The sign function \( \epsilon: \nabla_v \to \{ \pm 1 \} \) is defined as follows:

\begin{itemize}
    \item For each \( f \in \nabla_v \), consider the sequence \( (v', v_1, v_2, v_3) \), where \( v_1 \) is the smallest vertex in \( \partial f \) with respect to the total order, and \( (v_1, v_2, v_3) \) follows the cyclic order induced by the orientation of \( \Sigma \).
    \item The value \( \epsilon(f) \) is the sign of the permutation rearranging \( (v', v_1, v_2, v_3) \) into ascending order.
\end{itemize}

Finally, we define
\begin{equation} \label{starfase}
\alpha_v^g(\varphi) := \prod_{f \in \nabla_v} \zeta(f)^{\epsilon(f)}.
\end{equation}
\end{definition}

\begin{example}
Consider the vertex \( v_3 \) in Figure~\ref{phase}, where \( \nabla_{v_3} = \{ f_1, f_2, f_3, f_4 \} \). For face \( f_2 \), the vertices are ordered as \( v_1 < v_3' < v_3 < v_4 \). Walking along \( \partial f_2 \) counterclockwise gives the sequence \( (v_3', v_1, v_3, v_4) \), which leads to the scalar
\[
\alpha\big([v_1 v_3'], [v_3' v_3], [v_3 v_4]\big)^{-1} = \alpha\big([v_1 v_3][v_3 v_3']^{-1}, [v_3' v_3], [v_3 v_4]\big)^{-1} = \alpha\big([v_3 v_1] g^{-1}, g, [v_3 v_4]\big)^{-1}.
\]
Repeating this process for each face in \( \nabla_{v_3} \) yields
\begin{align*}
\alpha_{v_3}^g(\varphi) = &\ \alpha\big([v_1 v_2]_\varphi, [v_2 v_3]_\varphi g^{-1}, g\big) \alpha\big([v_1 v_3]_\varphi g^{-1}, g, [v_3 v_4]_\varphi\big)^{-1} \\
&\times \alpha\big([v_2 v_3]_\varphi g^{-1}, g, [v_3 v_5]_\varphi\big) \alpha\big(g, [v_3 v_4]_\varphi, [v_4 v_5]_\varphi\big)^{-1}.
\end{align*}
\end{example}

\begin{figure}[h!]
\centering
\begin{tikzpicture}[x=0.60pt,y=0.60pt,yscale=-1,xscale=1]

\draw  [fill={rgb, 255:red, 0; green, 0; blue, 0 }  ,fill opacity=1 ] (158,113) -- (162.5,105) -- (153.5,105) -- cycle ;
\draw  [fill={rgb, 255:red, 0; green, 0; blue, 0 }  ,fill opacity=1 ] (202.5,145.99) -- (194.52,141.51) -- (194.48,150.51) -- cycle ;
\draw  [fill={rgb, 255:red, 0; green, 0; blue, 0 }  ,fill opacity=1 ] (125.5,146) -- (118.5,141.5) -- (118.5,150.5) -- cycle ;
\draw  [fill={rgb, 255:red, 0; green, 0; blue, 0 }  ,fill opacity=1 ] (158,191) -- (162.5,183) -- (153.5,183) -- cycle ;
\draw    (158,65) -- (158,227) ;
\draw    (77,146) -- (200,146) -- (239,146) ;
\draw    (77,146) -- (158,227) ;
\draw    (158,65) -- (239,146) ;
\draw    (158,227) -- (239,146) ;
\draw    (77,146) -- (158,65) ;
\draw  [fill={rgb, 255:red, 0; green, 0; blue, 0 }  ,fill opacity=1 ] (201.33,108.83) -- (192.49,106.35) -- (199.85,99.99) -- cycle ;
\draw  [fill={rgb, 255:red, 0; green, 0; blue, 0 }  ,fill opacity=1 ] (121.83,189.83) -- (112.99,187.35) -- (119.35,180.99) -- cycle ;
\draw  [fill={rgb, 255:red, 0; green, 0; blue, 0 }  ,fill opacity=1 ] (194.17,189.83) -- (196.65,180.99) -- (203.01,187.35) -- cycle ;
\draw  [fill={rgb, 255:red, 0; green, 0; blue, 0 }  ,fill opacity=1 ] (113.67,108.83) -- (116.15,99.99) -- (122.51,106.35) -- cycle ;

\draw (150,50) node [anchor=north west][inner sep=0.75pt]   [align=left] {$v_1$};
\draw (125,107) node [anchor=north west][inner sep=0.75pt]   [align=left] {$f_1$};
\draw (169,107) node [anchor=north west][inner sep=0.75pt]   [align=left] {$f_2$};
\draw (125,170) node [anchor=north west][inner sep=0.75pt]   [align=left] {$f_3$};
\draw (169,170) node [anchor=north west][inner sep=0.75pt]   [align=left] {$f_4$};
\draw (158,147) node [anchor=north west][inner sep=0.75pt]   [align=left] {$v_3$};
\draw (57,139) node [anchor=north west][inner sep=0.75pt]   [align=left] {$v_2$};
\draw (242,139) node [anchor=north west][inner sep=0.75pt]   [align=left] {$v_4$};
\draw (150,228) node [anchor=north west][inner sep=0.75pt]   [align=left] {$v_5$};
\end{tikzpicture}
\caption{Example for the calculation of $\alpha_{v_3}^g$.}
\label{phase}
\end{figure}
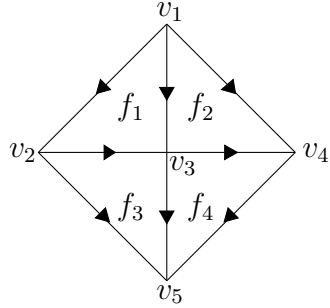

Figure~\ref{vertextiwst} illustrates the action of $A_v^g$ on a basis element $\ket{\varphi}$,  where $\varphi \in \operatorname{Col}(\mathcal{L})$. Note that vertices in $\mathcal{L}$ do not need to have the same number of incident edges. Figure~\ref{vertextiwst} illustrates an example where a vertex is incident to four edges, though this is not always the case.
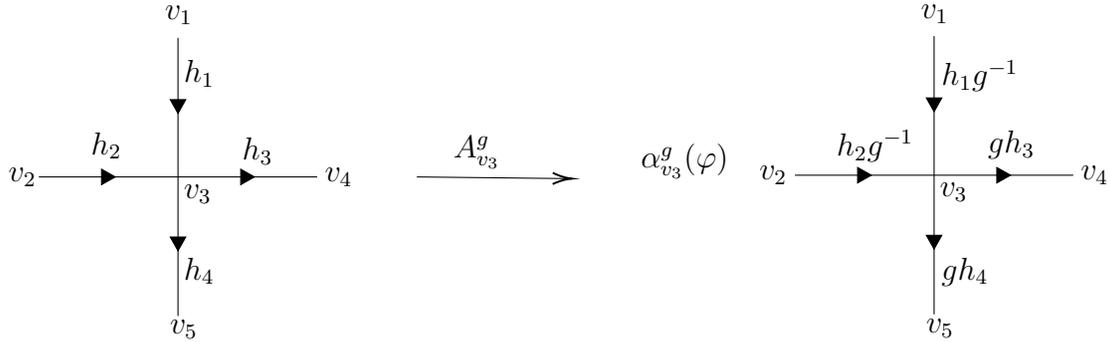
\begin{figure}[h!]
\centering
\begin{tikzpicture}[x=0.65pt,y=0.65pt,yscale=-1,xscale=1]

\draw  [fill={rgb, 255:red, 0; green, 0; blue, 0 }  ,fill opacity=1 ] (100,99) -- (104.5,91) -- (95.5,91) -- cycle ;
\draw  [fill={rgb, 255:red, 0; green, 0; blue, 0 }  ,fill opacity=1 ] (144.5,135.99) -- (136.5,131.51) -- (136.5,140.51) -- cycle ;
\draw  [fill={rgb, 255:red, 0; green, 0; blue, 0 }  ,fill opacity=1 ] (63.5,136) -- (55.5,131.5) -- (55.5,140.5) -- cycle ;
\draw  [fill={rgb, 255:red, 0; green, 0; blue, 0 }  ,fill opacity=1 ] (100,179) -- (104.5,171) -- (95.5,171) -- cycle ;
\draw    (100,55) -- (100,217) ;
\draw    (19,136) -- (181,136) ;
\draw    (239,136) -- (327,136.98) ;
\draw [shift={(329,137)}, rotate = 180.62] [color={rgb, 255:red, 0; green, 0; blue, 0 }  ][line width=0.75]    (10.93,-3.29) .. controls (6.95,-1.4) and (3.31,-0.3) .. (0,0) .. controls (3.31,0.3) and (6.95,1.4) .. (10.93,3.29)   ;
\draw  [fill={rgb, 255:red, 0; green, 0; blue, 0 }  ,fill opacity=1 ] (540,98) -- (544.5,90) -- (535.5,90) -- cycle ;
\draw  [fill={rgb, 255:red, 0; green, 0; blue, 0 }  ,fill opacity=1 ] (584.5,134.99) -- (576.5,130.51) -- (576.5,139.51) -- cycle ;
\draw  [fill={rgb, 255:red, 0; green, 0; blue, 0 }  ,fill opacity=1 ] (503.5,135) -- (495.5,130.5) -- (495.5,139.5) -- cycle ;
\draw  [fill={rgb, 255:red, 0; green, 0; blue, 0 }  ,fill opacity=1 ] (540,178) -- (544.5,170) -- (535.5,170) -- cycle ;
\draw    (540,54) -- (540,216) ;
\draw    (459,135) -- (621,135) ;

\draw (91,35) node [anchor=north west][inner sep=0.75pt]   [align=left] {$v_1$};
\draw (102,139) node [anchor=north west][inner sep=0.75pt]   [align=left] {$v_3$};
\draw (0,129) node [anchor=north west][inner sep=0.75pt]   [align=left] {$v_2$};
\draw (184,129) node [anchor=north west][inner sep=0.75pt]   [align=left] {$v_4$};
\draw (94,218) node [anchor=north west][inner sep=0.75pt]   [align=left] {$v_5$};
\draw (259,110.4) node [anchor=north west][inner sep=0.75pt]    {$A_{v_3}^g$};
\draw (368,115.4) node [anchor=north west][inner sep=0.75pt]    {$\alpha _{v_3}^{g}(\varphi)$};
\draw (102,66.4) node [anchor=north west][inner sep=0.75pt]    {$h_{1}$};
\draw (48,108.4) node [anchor=north west][inner sep=0.75pt]    {$h_{2}$};
\draw (136,111.4) node [anchor=north west][inner sep=0.75pt]    {$h_{3}$};
\draw (102,181.4) node [anchor=north west][inner sep=0.75pt]    {$h_{4}$};
\draw (531,34) node [anchor=north west][inner sep=0.75pt]   [align=left] {$v_1$};
\draw (542,138) node [anchor=north west][inner sep=0.75pt]   [align=left] {$v_3$};
\draw (437,128) node [anchor=north west][inner sep=0.75pt]   [align=left] {$v_2$};
\draw (624,128) node [anchor=north west][inner sep=0.75pt]   [align=left] {$v_4$};
\draw (534,217) node [anchor=north west][inner sep=0.75pt]   [align=left] {$v_5$};
\draw (543,66.4) node [anchor=north west][inner sep=0.75pt]    {$h_{1} g^{-1}$};
\draw (482,107.4) node [anchor=north west][inner sep=0.75pt]    {$h_{2} g^{-1}$};
\draw (570,107.4) node [anchor=north west][inner sep=0.75pt]    {$gh_{3}$};
\draw (543,181.4) node [anchor=north west][inner sep=0.75pt]    {$gh_{4}$};
\end{tikzpicture}
\caption{Action of operator $A_{v_3}^g$, over a coloring $\varphi$.}
\label{vertextiwst}
\end{figure}

\section{Twisted Quantum Double model on a General Lattice}
\label{sec:general lattice}

In this section, we extend the definition of the Twisted Quantum Double model to a lattice \(\mathcal{L}\) that does not need to be triangular. Recall that a lattice is a graph; hence, a cycle \(\mathcal{C}\) in \(\mathcal{L}\) corresponds to a sequence of edges \((l_1, l_2, \ldots, l_n)\) associated with vertices \((v_1, v_2, \ldots, v_n)\), where \(\partial(l_i) = \{v_i, v_{i+1}\}\) for \(i = 1, \ldots, n - 1\), and \(\partial(l_n) = \{v_n, v_1\}\), ensuring that the cycle is closed.

As in the triangular case, we assign a total order to the vertices of the planar graph constructed when we lift the lattice to its universal cover \(\mathcal{L}\), inducing an orientation on the edges: each edge is directed from the vertex with the smaller label to the one with the larger label.

For a finite group \(G\), let \(\operatorname{Col}(\mathcal{L})\) denote the set of all \(G\)-colorings of \(\mathcal{L}\). The total Hilbert space of the model is defined by
\[
\mathcal{H}_{\text{tot}} := \operatorname{Span} \{ \ket{\varphi} \mid \varphi \in \operatorname{Col}(\mathcal{L}) \}.
\]

To define the face and vertex operators, we embed \(\mathcal{L}\) into a triangular lattice \(\mathcal{L}'\) that corresponds to a $\Delta$-complex of \(\Sigma\) canonically related to \(\mathcal{L}\) and the total ordering of its vertices. This embedding allows us to define an isometry \(T: \mathcal{H}_{\text{tot}}(\mathcal{L}) \to \mathcal{H}_{\text{tot}}(\mathcal{L}')\). We then define the vertex operator \(A_v\) on \(\mathcal{H}_{\text{tot}}(\mathcal{L})\) by \(A_v := T^* A_v' T\), where \(A_v'\) is the vertex operator previously defined on \(\mathcal{H}_{\text{tot}}(\mathcal{L}')\) in Section~\ref{TM}. Similarly, we define the face operators in the same manner

To construct the triangular lattice \(\mathcal{L}'\) we proceed as follows: for each face \(f\) of \(\mathcal{L}\) with smallest vertex \(v \in \partial(f)\), we introduce edges \([v,u]\) for every \(u \in \partial(f) \setminus \{v\}\). These edges are called \emph{ghost edges} if \([v,u] \notin \mathcal{L}\) (see Figure~\ref{tringulando un lattice} for an example, where ghost edges are drawn as dotted lines).

\begin{figure}[h!]
    \centering
\begin{tikzpicture}[x=0.6pt,y=0.6pt,yscale=-1,xscale=1]

\draw  [dash pattern={on 0.84pt off 2.51pt}]  (430.78,18.63) -- (476.46,101.6) ;
\draw  [dash pattern={on 0.84pt off 2.51pt}]  (440.45,186.43) -- (476.99,102.55) ;
\draw  [dash pattern={on 0.84pt off 2.51pt}]  (477.99,101.55) -- (560.21,15.88) ;
\draw  [dash pattern={on 0.84pt off 2.51pt}]  (564.78,109.67) -- (614.26,31.21) ;
\draw  [dash pattern={on 0.84pt off 2.51pt}]  (611.22,171.9) -- (614.26,31.21) ;
\draw  [fill={rgb, 255:red, 255; green, 255; blue, 255 }  ,fill opacity=1 ] (456.82,65.83) -- (450.62,61.7) -- (457.69,58.43) -- cycle ;
\draw  [fill={rgb, 255:red, 255; green, 255; blue, 255 }  ,fill opacity=1 ] (517.21,60.28) -- (515.83,67.62) -- (510.06,62.45) -- cycle ;
\draw  [fill={rgb, 255:red, 255; green, 255; blue, 255 }  ,fill opacity=1 ] (456.03,149.93) -- (455.71,142.14) -- (462.18,145.14) -- cycle ;
\draw  [fill={rgb, 255:red, 255; green, 255; blue, 255 }  ,fill opacity=1 ] (586.61,74.69) -- (587.08,66.99) -- (593.27,70.8) -- cycle ;
\draw  [fill={rgb, 255:red, 255; green, 255; blue, 255 }  ,fill opacity=1 ] (612.62,105.16) -- (609.44,97.78) -- (616.29,98.12) -- cycle ;
\draw    (45.83,18.63) -- (15.38,108.81) ;
\draw    (15.38,108.81) -- (91.5,101.6) ;
\draw    (15.38,108.81) -- (54.96,185.47) ;
\draw    (45.83,18.63) -- (99.88,50.19) ;
\draw    (99.88,50.19) -- (91.5,101.6) ;
\draw    (91.5,101.6) -- (178.29,109.72) ;
\draw    (54.96,185.47) -- (178.29,109.72) ;
\draw    (178.29,109.72) -- (224.73,171.95) ;
\draw    (99.88,50.19) -- (173.72,15.92) ;
\draw    (173.72,15.92) -- (178.29,109.72) ;
\draw    (173.72,15.92) -- (227.77,31.25) ;
\draw    (227.77,31.25) -- (262.03,104.3) ;
\draw    (224.73,171.95) -- (262.03,104.3) ;
\draw    (54.96,185.47) -- (224.73,171.95) ;
\draw    (45.83,18.63) -- (173.72,15.92) ;
\draw  [fill={rgb, 255:red, 0; green, 0; blue, 0 }  ,fill opacity=1 ] (29.41,67.01) -- (28.47,59.23) -- (35.11,61.63) -- cycle ;
\draw  [fill={rgb, 255:red, 0; green, 0; blue, 0 }  ,fill opacity=1 ] (59.32,104.64) -- (65.07,100.1) -- (65.72,108.19) -- cycle ;
\draw  [fill={rgb, 255:red, 0; green, 0; blue, 0 }  ,fill opacity=1 ] (110.54,17.69) -- (103.2,21.55) -- (103.33,13.6) -- cycle ;
\draw  [fill={rgb, 255:red, 0; green, 0; blue, 0 }  ,fill opacity=1 ] (135.12,34.01) -- (130.93,40.09) -- (127.84,32.78) -- cycle ;
\draw  [fill={rgb, 255:red, 0; green, 0; blue, 0 }  ,fill opacity=1 ] (95.91,74.78) -- (97.98,82.47) -- (91.21,81.2) -- cycle ;
\draw  [fill={rgb, 255:red, 0; green, 0; blue, 0 }  ,fill opacity=1 ] (38.22,153.33) -- (32,148.71) -- (38.49,145.59) -- cycle ;
\draw  [fill={rgb, 255:red, 0; green, 0; blue, 0 }  ,fill opacity=1 ] (113.86,149.13) -- (117.5,142.64) -- (121.29,149.48) -- cycle ;
\draw  [fill={rgb, 255:red, 0; green, 0; blue, 0 }  ,fill opacity=1 ] (133.44,105.27) -- (127.12,108.98) -- (127.58,100.88) -- cycle ;
\draw  [fill={rgb, 255:red, 0; green, 0; blue, 0 }  ,fill opacity=1 ] (176.2,66.42) -- (172.39,59.41) -- (179.24,59.03) -- cycle ;
\draw  [fill={rgb, 255:red, 0; green, 0; blue, 0 }  ,fill opacity=1 ] (203.53,143.49) -- (196.54,140.42) -- (202.43,135.93) -- cycle ;
\draw  [fill={rgb, 255:red, 0; green, 0; blue, 0 }  ,fill opacity=1 ] (152.39,177.54) -- (146.61,182.02) -- (146.01,173.94) -- cycle ;
\draw  [fill={rgb, 255:red, 0; green, 0; blue, 0 }  ,fill opacity=1 ] (249.34,76.75) -- (243.09,72.1) -- (249.49,68.96) -- cycle ;
\draw  [fill={rgb, 255:red, 0; green, 0; blue, 0 }  ,fill opacity=1 ] (243.21,138.41) -- (243.47,130.65) -- (249.73,134.21) -- cycle ;
\draw  [fill={rgb, 255:red, 0; green, 0; blue, 0 }  ,fill opacity=1 ] (76.02,36.24) -- (67.91,35.65) -- (71.45,29.51) -- cycle ;
\draw  [fill={rgb, 255:red, 0; green, 0; blue, 0 }  ,fill opacity=1 ] (197.82,22.68) -- (204.87,20.64) -- (202.49,28.34) -- cycle ;
\draw    (431.78,18.63) -- (401.33,108.81) ;
\draw    (401.33,108.81) -- (477.46,101.6) ;
\draw    (401.33,108.81) -- (440.92,185.47) ;
\draw    (431.78,18.63) -- (485.83,50.19) ;
\draw    (485.83,50.19) -- (477.46,101.6) ;
\draw    (477.46,101.6) -- (564.24,109.72) ;
\draw    (440.92,185.47) -- (564.24,109.72) ;
\draw    (564.24,109.72) -- (610.68,171.95) ;
\draw    (485.83,50.19) -- (559.68,15.92) ;
\draw    (559.68,15.92) -- (564.24,109.72) ;
\draw    (559.68,15.92) -- (613.73,31.25) ;
\draw    (613.73,31.25) -- (647.99,104.3) ;
\draw    (610.68,171.95) -- (647.99,104.3) ;
\draw    (440.92,185.47) -- (610.68,171.95) ;
\draw    (431.78,18.63) -- (559.68,15.92) ;
\draw  [fill={rgb, 255:red, 0; green, 0; blue, 0 }  ,fill opacity=1 ] (415.36,67.01) -- (414.43,59.23) -- (421.07,61.63) -- cycle ;
\draw  [fill={rgb, 255:red, 0; green, 0; blue, 0 }  ,fill opacity=1 ] (445.27,104.64) -- (451.02,100.1) -- (451.67,108.19) -- cycle ;
\draw  [fill={rgb, 255:red, 0; green, 0; blue, 0 }  ,fill opacity=1 ] (496.49,17.69) -- (489.16,21.55) -- (489.28,13.6) -- cycle ;
\draw  [fill={rgb, 255:red, 0; green, 0; blue, 0 }  ,fill opacity=1 ] (521.07,34.01) -- (516.88,40.09) -- (513.79,32.78) -- cycle ;
\draw  [fill={rgb, 255:red, 0; green, 0; blue, 0 }  ,fill opacity=1 ] (481.87,74.78) -- (483.93,82.47) -- (477.16,81.2) -- cycle ;
\draw  [fill={rgb, 255:red, 0; green, 0; blue, 0 }  ,fill opacity=1 ] (424.17,153.33) -- (417.95,148.71) -- (424.45,145.59) -- cycle ;
\draw  [fill={rgb, 255:red, 0; green, 0; blue, 0 }  ,fill opacity=1 ] (499.81,149.13) -- (503.45,142.64) -- (507.24,149.48) -- cycle ;
\draw  [fill={rgb, 255:red, 0; green, 0; blue, 0 }  ,fill opacity=1 ] (519.39,105.27) -- (513.08,108.98) -- (513.53,100.88) -- cycle ;
\draw  [fill={rgb, 255:red, 0; green, 0; blue, 0 }  ,fill opacity=1 ] (562.16,66.42) -- (558.34,59.41) -- (565.19,59.03) -- cycle ;
\draw  [fill={rgb, 255:red, 0; green, 0; blue, 0 }  ,fill opacity=1 ] (589.49,143.49) -- (582.5,140.42) -- (588.39,135.93) -- cycle ;
\draw  [fill={rgb, 255:red, 0; green, 0; blue, 0 }  ,fill opacity=1 ] (538.34,177.54) -- (532.56,182.02) -- (531.97,173.94) -- cycle ;
\draw  [fill={rgb, 255:red, 0; green, 0; blue, 0 }  ,fill opacity=1 ] (635.3,76.75) -- (629.04,72.1) -- (635.45,68.96) -- cycle ;
\draw  [fill={rgb, 255:red, 0; green, 0; blue, 0 }  ,fill opacity=1 ] (629.16,138.41) -- (629.43,130.65) -- (635.68,134.21) -- cycle ;
\draw  [fill={rgb, 255:red, 0; green, 0; blue, 0 }  ,fill opacity=1 ] (461.98,36.24) -- (453.87,35.65) -- (457.41,29.51) -- cycle ;
\draw  [fill={rgb, 255:red, 0; green, 0; blue, 0 }  ,fill opacity=1 ] (583.77,22.68) -- (590.82,20.64) -- (588.45,28.34) -- cycle ;
\draw (425.27,-3) node [anchor=north west][inner sep=0.75pt]    {$1$};
\draw (612.47,174.22) node [anchor=north west][inner sep=0.75pt]    {$10$};
\draw (438.21,188.04) node [anchor=north west][inner sep=0.75pt]    {$9$};
\draw (652,96.66) node [anchor=north west][inner sep=0.75pt]    {$7$};
\draw (557.78,115.8) node [anchor=north west][inner sep=0.75pt]    {$8$};
\draw (481.24,103.81) node [anchor=north west][inner sep=0.75pt]    {$2$};
\draw (387,96.66) node [anchor=north west][inner sep=0.75pt]    {$5$};
\draw (481.61,25.82) node [anchor=north west][inner sep=0.75pt]    {$4$};
\draw (613.31,10.59) node [anchor=north west][inner sep=0.75pt]    {$3$};
\draw (553.93,-6) node [anchor=north west][inner sep=0.75pt]    {$6$};
\draw (36.78,-3) node [anchor=north west][inner sep=0.75pt]    {$1$};
\draw (223.98,174.27) node [anchor=north west][inner sep=0.75pt]    {$10$};
\draw (49.72,188.09) node [anchor=north west][inner sep=0.75pt]    {$9$};
\draw (263,96.7) node [anchor=north west][inner sep=0.75pt]    {$7$};
\draw (170.29,115.84) node [anchor=north west][inner sep=0.75pt]    {$8$};
\draw (92.75,103.85) node [anchor=north west][inner sep=0.75pt]    {$2$};
\draw (0,96.7) node [anchor=north west][inner sep=0.75pt]    {$5$};
\draw (93.12,25.87) node [anchor=north west][inner sep=0.75pt]    {$4$};
\draw (224.82,10.63) node [anchor=north west][inner sep=0.75pt]    {$3$};
\draw (165.44,-6) node [anchor=north west][inner sep=0.75pt]    {$6$};
\end{tikzpicture}
    \caption{On the left, a generic lattice \(\mathcal{L}\) with ordered vertices. On the right, the corresponding canonical triangular lattice \(\mathcal{L}'\) obtained by adding ghost edges (dotted lines).}
    \label{tringulando un lattice}
\end{figure}
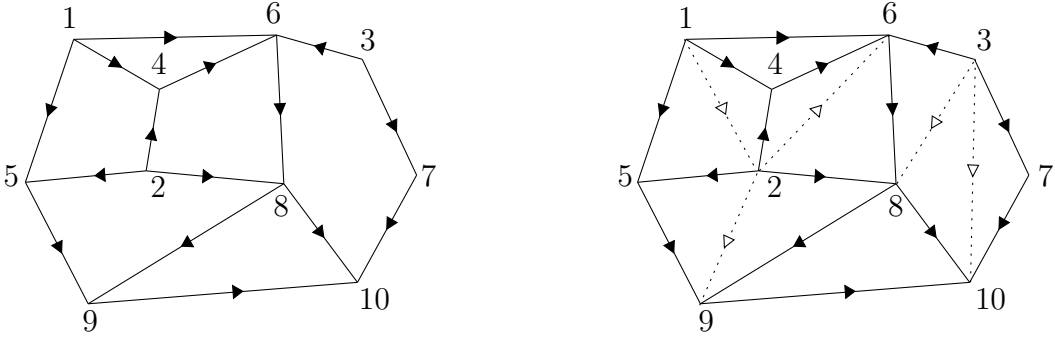

We define an injective map \(\iota: \operatorname{Col}(\mathcal{L}) \to \operatorname{Col}(\mathcal{L}')\) by extending a given \(G\)-coloring \(\varphi\) from \(\mathcal{L}\) to the triangular lattice \(\mathcal{L}'\). For each face \(f\) with smallest vertex \(v\), and for each \(u \in \partial(f) \setminus \{v\}\), we move counterclockwise along the boundary \(\partial f\) through the vertices \(v = v^1, v^2, \dots, v^n = u\).
The extended coloring is then defined as:
\[
[v u]_\varphi := [v^1 v^2]_\varphi [v^2 v^3]_\varphi \cdots [v^{n-1} v^n]_\varphi.
\]
Additionally, we define a surjective map \(\pi: \operatorname{Col}(\mathcal{L}') \to \operatorname{Col}(\mathcal{L})\), given by \(\varphi \mapsto \varphi|_{\mathcal{L}}\), which is the restriction of \(\varphi\) to the edges of \(\mathcal{L}\). Note that \(\pi \circ \iota = \operatorname{id}_{\operatorname{Col}(\mathcal{L})}\).

The map \(\iota: \operatorname{Col}(\mathcal{L}) \to \operatorname{Col}(\mathcal{L}')\) naturally induces an isometry \(T: \mathcal{H}_{\text{tot}}(\mathcal{L}) \to \mathcal{H}_{\text{tot}}(\mathcal{L}')\). The map \(\pi: \operatorname{Col}(\mathcal{L}') \to \operatorname{Col}(\mathcal{L})\) defines the adjoint, \(T^*: \mathcal{H}_{\text{tot}}(\mathcal{L}') \to \mathcal{H}_{\text{tot}}(\mathcal{L})\).

The vertex operator \(A_v^g\) of the model associated to \(\mathcal{L}\) is defined as:
\[
A_v^g := T^* \circ {A'}_v^{g} \circ T,
\]
where \({A'}_v^{g}\) are the vertex operators over \(\mathcal{L}'\) defined in Section~\ref{TM}. Similarly, we define \(A_v := \frac{1}{|G|} \sum_{g \in G} A_v^g\).

In order to define the face operator, let us introduce the notion of holonomy with respect to a \(G\)-coloring around a cycle in \(\mathcal{L}\). Given a coloring \(\varphi\) and a cycle determined by the sequence of vertices \((v_1, v_2, \ldots, v_n)\), the \emph{holonomy} is defined as
\[
\varphi_\mathcal{C} = [v_1 v_2]_\varphi [v_2 v_3]_\varphi \cdots [v_{n-1} v_n]_\varphi [v_n v_1]_\varphi.
\]

Given a cycle \(\mathcal{C} \subset \mathcal{L}\), we define the \emph{holonomy operator} \( B_\mathcal{C} \) as follows: \( B_\mathcal{C} \ket{\varphi} = \ket{\varphi} \) if \( \varphi_\mathcal{C} = e \) (the identity element of \( G \)), and \( B_\mathcal{C} \ket{\varphi} = 0 \) otherwise. The face operators \(B_f\) are defined as the holonomy operators for the cycle determined by \(\partial f\), with the orientation induced by \(\Sigma\) (counterclockwise when drawn on a plane). That is, \(B_f = B_{\partial f}\) (see Figure~\ref{faceoperatorsquare} for an example of the face operator for a square face).

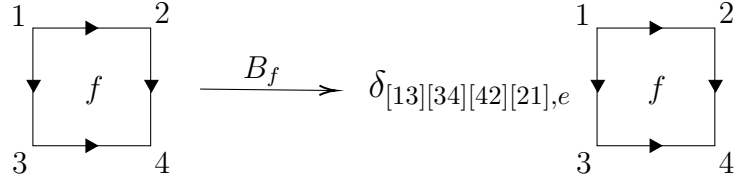
\begin{figure}[h!]
    \centering
\begin{tikzpicture}[x=0.55pt,y=0.55pt,yscale=-1,xscale=1]

\draw  [fill={rgb, 255:red, 0; green, 0; blue, 0 }  ,fill opacity=1 ] (97.5,66) -- (102,58) -- (93,58) -- cycle ;
\draw   (16,22) -- (97,22) -- (97,103) -- (16,103) -- cycle ;
\draw  [fill={rgb, 255:red, 0; green, 0; blue, 0 }  ,fill opacity=1 ] (16.5,66) -- (21,58) -- (12,58) -- cycle ;
\draw  [fill={rgb, 255:red, 0; green, 0; blue, 0 }  ,fill opacity=1 ] (60.5,22.02) -- (52.5,17.48) -- (52.5,26.48) -- cycle ;
\draw  [fill={rgb, 255:red, 0; green, 0; blue, 0 }  ,fill opacity=1 ] (60.5,103) -- (52.5,98.5) -- (52.5,107.5) -- cycle ;
\draw    (131,64) -- (222,64.98) ;
\draw [shift={(224,65)}, rotate = 180.62] [color={rgb, 255:red, 0; green, 0; blue, 0 }  ][line width=0.75]    (10.93,-3.29) .. controls (6.95,-1.4) and (3.31,-0.3) .. (0,0) .. controls (3.31,0.3) and (6.95,1.4) .. (10.93,3.29)   ;
\draw  [fill={rgb, 255:red, 0; green, 0; blue, 0 }  ,fill opacity=1 ] (485.5,66) -- (490,58) -- (481,58) -- cycle ;
\draw   (404,22) -- (485,22) -- (485,103) -- (404,103) -- cycle ;
\draw  [fill={rgb, 255:red, 0; green, 0; blue, 0 }  ,fill opacity=1 ] (404.5,66) -- (409,58) -- (400,58) -- cycle ;
\draw  [fill={rgb, 255:red, 0; green, 0; blue, 0 }  ,fill opacity=1 ] (448.5,22.02) -- (440.5,17.48) -- (440.5,26.48) -- cycle ;
\draw  [fill={rgb, 255:red, 0; green, 0; blue, 0 }  ,fill opacity=1 ] (448.5,103) -- (440.5,98.5) -- (440.5,107.5) -- cycle ;

\draw (-1,5) node [anchor=north west][inner sep=0.75pt]   [align=left] {$1$};
\draw (98,3) node [anchor=north west][inner sep=0.75pt]   [align=left] {$2$};
\draw (0,106) node [anchor=north west][inner sep=0.75pt]   [align=left] {$3$};
\draw (98,106) node [anchor=north west][inner sep=0.75pt]   [align=left] {$4$};
\draw (159,39.4) node [anchor=north west][inner sep=0.75pt]    {$B_f$};
\draw (245,49.4) node [anchor=north west][inner sep=0.75pt]  [font=\large]  {$\delta_{[ 13][ 34][42][21],e}$};
\draw (387,4) node [anchor=north west][inner sep=0.75pt]   [align=left] {$1$};
\draw (486,3) node [anchor=north west][inner sep=0.75pt]   [align=left] {$2$};
\draw (388,106) node [anchor=north west][inner sep=0.75pt]   [align=left] {$3$};
\draw (486,106) node [anchor=north west][inner sep=0.75pt]   [align=left] {$4$};
\draw (50,53.4) node [anchor=north west][inner sep=0.75pt]    {$f$};
\draw (437,53.4) node [anchor=north west][inner sep=0.75pt]    {$f$};
\end{tikzpicture}
    \caption{The action of operator $B$ on a square face $f$.}
    \label{faceoperatorsquare}
\end{figure}

Thus, the Hamiltonian of the model on the general lattice \(\mathcal{L}\) is defined as
\[
H = \sum_{v \in V(\mathcal{L})} \left(1 - A_v\right) + \sum_{f \in F(\mathcal{L})} \left(1 - B_f\right).
\]

\begin{example}
Let us illustrate this with an example. Consider the vertex \(v_5\) in Figure~\ref{phase3}. Since \(A_{v_5}^g := T^* \circ {A'}_{v_5}^{g} \circ T\), we first examine the operator \({A'}_{v_5}^{g}\) in the triangulated lattice \(\mathcal{L}'\). Note that when we consider the vertex \(v_5\) in \(\mathcal{L}\), it belongs to four faces, but in \(\mathcal{L}'\), this vertex belongs to five faces due to the addition of ghost edges. Thus, \(\nabla_{v_5} = \{ f_{1,1}, f_{1,2}, f_{2,2}, f_3, f_4\}\). The function \(\alpha_{v_5}^g(\varphi)\) will then be the product of the following five scalars, one for each face in \(\mathcal{L}'\):

\begin{align*}
   \zeta(f_{1,1})^{\epsilon(f_{1,1})} & = \alpha\left([v_1 v_2]_\varphi, [v_2 v_5]_\varphi g^{-1}, g\right),\\
   \zeta(f_{1,2})^{\epsilon(f_{1,2})} & = \alpha\left([v_1 v_4]_\varphi, [v_4 v_5]_\varphi g^{-1}, g\right)^{-1},\\
   \zeta(f_{2,2})^{\epsilon(f_{2,2})} & = \alpha\left([v_2 v_5]_\varphi g^{-1}, g, [v_5 v_6]_\varphi\right)^{-1},\\
   \zeta(f_3)^{\epsilon(f_3)} & = \alpha\left([v_4 v_5]_\varphi g^{-1}, g, [v_5 v_7]_\varphi\right),\\
   \zeta(f_4)^{\epsilon(f_4)} & = \alpha\left(g, [v_5 v_6]_\varphi, [v_6 v_7]_\varphi\right).
\end{align*}
Therefore, \(\alpha_{v_5}^g(\varphi) = \prod_{f \in \nabla_{v_5}} \zeta(f)^{\epsilon(f)}\).

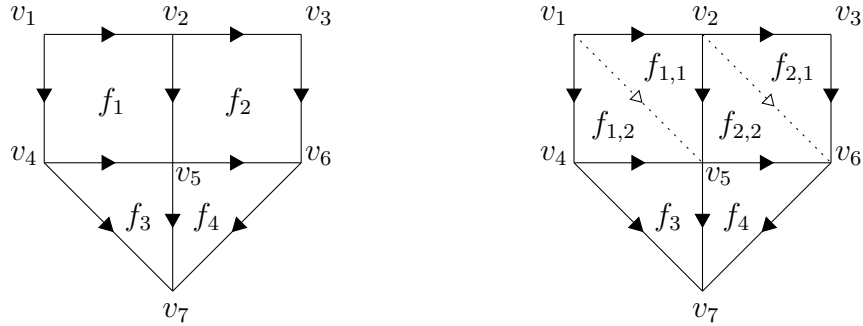
\begin{figure}[h!]
\centering
\begin{tikzpicture}[x=0.6pt,y=0.6pt,yscale=-1,xscale=1]

\draw    (100,25) -- (100,66.53) -- (100,187) ;
\draw    (19,106) -- (142,106) -- (181,106) ;
\draw  [fill={rgb, 255:red, 0; green, 0; blue, 0 }  ,fill opacity=1 ] (100,67.53) -- (95.5,59.53) -- (104.5,59.53) -- cycle ;
\draw  [fill={rgb, 255:red, 0; green, 0; blue, 0 }  ,fill opacity=1 ] (144.5,106) -- (136.5,110.5) -- (136.5,101.5) -- cycle ;
\draw  [fill={rgb, 255:red, 0; green, 0; blue, 0 }  ,fill opacity=1 ] (63.5,106) -- (55.5,110.5) -- (55.5,101.5) -- cycle ;
\draw    (19,106) -- (100,187) ;
\draw    (100,187) -- (181,106) ;
\draw  [fill={rgb, 255:red, 0; green, 0; blue, 0 }  ,fill opacity=1 ] (62.33,149.33) -- (53.49,146.85) -- (59.85,140.49) -- cycle ;
\draw  [fill={rgb, 255:red, 0; green, 0; blue, 0 }  ,fill opacity=1 ] (137.67,149.33) -- (140.15,140.49) -- (146.51,146.85) -- cycle ;
\draw    (181,25) -- (181,106) ;
\draw    (19,25) -- (19,106) ;
\draw    (19,25) -- (142,25) -- (181,25) ;
\draw  [fill={rgb, 255:red, 0; green, 0; blue, 0 }  ,fill opacity=1 ] (144.5,25.01) -- (136.5,29.51) -- (136.5,20.51) -- cycle ;
\draw  [fill={rgb, 255:red, 0; green, 0; blue, 0 }  ,fill opacity=1 ] (63.5,25) -- (55.5,29.5) -- (55.5,20.5) -- cycle ;
\draw  [fill={rgb, 255:red, 0; green, 0; blue, 0 }  ,fill opacity=1 ] (19,67.53) -- (14.5,59.53) -- (23.5,59.53) -- cycle ;
\draw  [fill={rgb, 255:red, 0; green, 0; blue, 0 }  ,fill opacity=1 ] (181,67.53) -- (176.5,59.53) -- (185.5,59.53) -- cycle ;
\draw  [fill={rgb, 255:red, 0; green, 0; blue, 0 }  ,fill opacity=1 ] (100,146.53) -- (95.5,138.53) -- (104.5,138.53) -- cycle ;
\draw    (434,25) -- (434,66.53) -- (434,187) ;
\draw    (353,106) -- (476,106) -- (515,106) ;
\draw  [fill={rgb, 255:red, 0; green, 0; blue, 0 }  ,fill opacity=1 ] (434,67.53) -- (429.5,59.53) -- (438.5,59.53) -- cycle ;
\draw  [fill={rgb, 255:red, 0; green, 0; blue, 0 }  ,fill opacity=1 ] (478.5,106) -- (470.5,110.5) -- (470.5,101.5) -- cycle ;
\draw  [fill={rgb, 255:red, 0; green, 0; blue, 0 }  ,fill opacity=1 ] (397.5,106) -- (389.5,110.5) -- (389.5,101.5) -- cycle ;
\draw    (353,106) -- (434,187) ;
\draw  [dash pattern={on 0.84pt off 2.51pt}]  (434,25) -- (515,106) ;
\draw    (434,187) -- (515,106) ;
\draw  [fill={rgb, 255:red, 0; green, 0; blue, 0 }  ,fill opacity=1 ] (396.33,149.33) -- (387.49,146.85) -- (393.85,140.49) -- cycle ;
\draw  [fill={rgb, 255:red, 0; green, 0; blue, 0 }  ,fill opacity=1 ] (471.67,149.33) -- (474.15,140.49) -- (480.51,146.85) -- cycle ;
\draw    (515,25) -- (515,106) ;
\draw    (353,25) -- (353,106) ;
\draw    (353,25) -- (476,25) -- (515,25) ;
\draw  [dash pattern={on 0.84pt off 2.51pt}]  (353,25) -- (434,106) ;
\draw  [fill={rgb, 255:red, 255; green, 255; blue, 255 }  ,fill opacity=1 ] (395.74,67.74) -- (388.5,66.01) -- (394.01,60.5) -- cycle ;
\draw  [fill={rgb, 255:red, 255; green, 255; blue, 255 }  ,fill opacity=1 ] (478.91,70.06) -- (471.67,68.33) -- (477.18,62.82) -- cycle ;
\draw  [fill={rgb, 255:red, 0; green, 0; blue, 0 }  ,fill opacity=1 ] (478.5,25.01) -- (470.5,29.51) -- (470.5,20.51) -- cycle ;
\draw  [fill={rgb, 255:red, 0; green, 0; blue, 0 }  ,fill opacity=1 ] (397.5,25) -- (389.5,29.5) -- (389.5,20.5) -- cycle ;
\draw  [fill={rgb, 255:red, 0; green, 0; blue, 0 }  ,fill opacity=1 ] (353,67.53) -- (348.5,59.53) -- (357.5,59.53) -- cycle ;
\draw  [fill={rgb, 255:red, 0; green, 0; blue, 0 }  ,fill opacity=1 ] (515,67.53) -- (510.5,59.53) -- (519.5,59.53) -- cycle ;
\draw  [fill={rgb, 255:red, 0; green, 0; blue, 0 }  ,fill opacity=1 ] (434,146.53) -- (429.5,138.53) -- (438.5,138.53) -- cycle ;
\draw (92,6) node [anchor=north west][inner sep=0.75pt]   [align=left] {$v_2$};
\draw (100,107) node [anchor=north west][inner sep=0.75pt]   [align=left] {$v_5$};
\draw (-4,94) node [anchor=north west][inner sep=0.75pt]   [align=left] {$v_4$};
\draw (182,94) node [anchor=north west][inner sep=0.75pt]   [align=left] {$v_6$};
\draw (92,193) node [anchor=north west][inner sep=0.75pt]   [align=left] {$v_7$};
\draw (-4,6) node [anchor=north west][inner sep=0.75pt]   [align=left] {$v_1$};
\draw (182,6) node [anchor=north west][inner sep=0.75pt]   [align=left] {$v_3$};
\draw (426,6) node [anchor=north west][inner sep=0.75pt]   [align=left] {$v_2$};
\draw (434,107) node [anchor=north west][inner sep=0.75pt]   [align=left] {$v_5$};
\draw (330,94) node [anchor=north west][inner sep=0.75pt]   [align=left] {$v_4$};
\draw (516,94) node [anchor=north west][inner sep=0.75pt]   [align=left] {$v_6$};
\draw (426,193) node [anchor=north west][inner sep=0.75pt]   [align=left] {$v_7$};
\draw (330,6) node [anchor=north west][inner sep=0.75pt]   [align=left] {$v_1$};
\draw (516,6) node [anchor=north west][inner sep=0.75pt]   [align=left] {$v_3$};
\draw (51,57) node [anchor=north west][inner sep=0.75pt]   [align=left] {$f_1$};
\draw (362,72) node [anchor=north west][inner sep=0.75pt]   [align=left] {$f_{1,2}$};
\draw (395,35) node [anchor=north west][inner sep=0.75pt]   [align=left] {$f_{1,1}$};
\draw (131,57) node [anchor=north west][inner sep=0.75pt]   [align=left] {$f_2$};
\draw (442,72) node [anchor=north west][inner sep=0.75pt]   [align=left] {$f_{2,2}$};
\draw (475,35) node [anchor=north west][inner sep=0.75pt]   [align=left] {$f_{2,1}$};
\draw (68,130) node [anchor=north west][inner sep=0.75pt]   [align=left] {$f_3$};
\draw (402,130) node [anchor=north west][inner sep=0.75pt]   [align=left] {$f_3$};
\draw (111,130) node [anchor=north west][inner sep=0.75pt]   [align=left] {$f_4$};
\draw (445,130) node [anchor=north west][inner sep=0.75pt]   [align=left] {$f_4$};
\end{tikzpicture}
   \caption{On the left, the lattice \(\mathcal{L}\) where some faces are not triangular. On the right, the triangular lattice \(\mathcal{L}'\) used to calculate \(A_{v_5}^g\), with ghost edges shown as dashed lines.}

\label{phase3}
\end{figure}

\end{example}

\begin{proposition}
Let \(f_1, f_2 \in F(\mathcal{L})\), \(v_1, v_2 \in V(\mathcal{L})\), and \(g_1, g_2 \in G\). Then:
\begin{enumerate}
    \item \(A_{v_1}^{g_1} A_{v_1}^{g_2} = A_{v_1}^{g_1 g_2}\).
    \item Commutativity:
    \begin{align*}
    A_{v_1} A_{v_2} = A_{v_2} A_{v_1}, \quad B_{f_1} B_{f_2} = B_{f_2} B_{f_1}, \quad A_{v_1} B_{f_1} = B_{f_1} A_{v_1}.
    \end{align*}
    \item Idempotence:
    \[
    A_{v_1}^2 = A_{v_1}, \quad B_{f_1}^2 = B_{f_1}.
    \]
\end{enumerate}
\end{proposition}

\begin{proof}
The equalities involving face operators are straightforward since the holonomy operators \(B_f\) are projectors and act nontrivially only on the colorings where the holonomy around the face \(f\) is trivial.

For the equations involving vertex operators, we consider \(\mathcal{L}'\) as the canonical triangular lattice of \(\mathcal{L}\). Using the fact that \(A_v^g = T^* \circ {A'}_v^{g} \circ T\), along with the properties of \(A_v^g\) in the triangular lattice (as detailed in \cite[Appendix B]{twisted}), we can derive the necessary proofs. For instance, for the first item:
\begin{align*}
A_{v_1}^{g_1} A_{v_1}^{g_2} & = T^* \circ {A'}_{v_1}^{g_1} \circ T \circ T^* \circ {A'}_{v_1}^{g_2} \circ T \\
& = T^* \circ {A'}_{v_1}^{g_1} {A'}_{v_1}^{g_2} \circ T \\
& = T^* \circ {A'}_{v_1}^{g_1 g_2} \circ T \\
& = A_{v_1}^{g_1 g_2}.
\end{align*}
The commutativity and idempotence properties follow similarly by using the corresponding properties of the operators in the triangular lattice and the intertwining property of \(T\) and \(T^*\).
\end{proof}

\begin{proposition}
Let \(\Sigma\) be a closed oriented surface. The dimension of the ground state space is a topological invariant; that is, it does not depend on the choice of lattice \(\mathcal{L}\).
\end{proposition}

\begin{proof}
We will prove that for any arbitrary lattice \(\mathcal{L}\), the dimension of the ground state space is equal to that of the ground state space associated with its canonical triangular lattice \(\mathcal{L}'\). This, together with \cite[Section 3]{twisted}, will conclude the proof.

Consider the sets:
\begin{align*}
\operatorname{Col}(\mathcal{L})^{0} & = \{\varphi \in \operatorname{Col}(\mathcal{L}) \mid B_f \ket{\varphi} = \ket{\varphi}, \ \forall f \in F(\mathcal{L}) \},\\
\operatorname{Col}(\mathcal{L}')^{0} & = \{\varphi \in \operatorname{Col}(\mathcal{L}') \mid B_f \ket{\varphi} = \ket{\varphi}, \ \forall f \in F(\mathcal{L}') \}.
\end{align*}

Note that the restriction of \(\pi\) and \(\iota\) gives bijections between \(\operatorname{Col}(\mathcal{L}')^{0}\) and \(\operatorname{Col}(\mathcal{L})^{0}\), since \(\pi \circ \iota = \operatorname{id}_{\operatorname{Col}(\mathcal{L})}\) and \(\iota \circ \pi\) is the identity on \(\operatorname{Col}(\mathcal{L}')^{0}\).

Now, take a vector \(\ket{\psi} \in V_{\text{gs}}(\mathcal{L})\) in the ground state space of \(\mathcal{H}_{\text{tot}}(\cL)\). Then
\[
\ket{\psi} = \sum_{\varphi \in \operatorname{Col}(\mathcal{L})^{0}} x_\varphi \ket{\varphi},
\]
with \(x_\varphi \in \mathbb{C}\). Consider the state \(\ket{\psi'} := T \ket{\psi}\). Since \(T\) maps \(\operatorname{Col}(\mathcal{L})^{0}\) to \(\operatorname{Col}(\mathcal{L}')^{0}\), it follows that \(B_f \ket{\psi'} = \ket{\psi'}\) for all \(f \in F(\mathcal{L}')\). Furthermore, \({A'}_v^g\) maps \(\operatorname{Col}(\mathcal{L}')^{0}\) to itself. Therefore, $A_{v}^{g}\ket{\psi'} \in \operatorname{Span} \{\Co(\cL')^{0}\} = \mathcal{H'}^0$. Thus,
\begin{align*}
    {A'}_{v}^{g}\ket{\psi'} & = T \circ T^*  \circ {A'}_{v}^{g}\ket{\psi'} \quad \quad \quad \quad \quad  \quad \text{(since } (T \circ T^* )|_{\mathcal{H'}^0} = \Id_{\mathcal{H'}^0})\\
    & = T \circ T^* \circ {A'}_{v}^{g} \circ T\ket{\psi}\\
    & =  T \circ A_{v}^{g} \ket{\psi} \\
    & = T\ket{\psi} \\
    & = \ket{\psi'}.
\end{align*}
Therefore, $\ket{\psi'} = T(\ket{\psi}) \in V_{gs}(\cL')$. Similarly, it can be shown that for any $\ket{\psi'} \in V_{gs}(\cL')$, the state $\ket{\psi} := T^*  (\ket{\psi'}) \in V_{gs}(\cL)$. In conclusion, the maps $T$ and $T^* $ induce a bijection between $V_{gs}(\cL)$ and $V_{gs}(\cL')$, proving that the dimension of the ground state space is independent of the choice of lattice \(\mathcal{L}\).

\end{proof}

\section{Ground states of the twisted model}\label{GS} 

In this section, we present an algorithm to construct a basis for the ground state space \(V_{gs}\), leveraging the fact that the operators \(A_v^g\) are monomial operators.

The constraint $B_f|\varphi \rangle=|\varphi \rangle$ is equivalent to the condition that $\varphi_{\partial f}= e$, where $\partial f$ is oriented counterclockwise, thought of as a cycle. Hence, the subspace fixed by all the \(B_f\) operators is spanned by the following set:
\begin{align}
\Co(\cL)^0 &= \left\{ \varphi \in \Co(\cL) \mid \varphi_{\partial f} = e, \ \forall f \in F \right\} \label{eq:col0_def1} \\
&= \left\{ \varphi \in \Co(\cL) \mid \varphi_\cC = e, \ \text{for any contractible cycle } \cC \right\}. \label{eq:col0_def2}
\end{align}

To describe a basis for \(V_{gs}\), we analyze how the vertex operators act on the set \(\Co(\cL)^0\). For this purpose, we recall the following definition.

\begin{definition}\cite{karpilovsky1985projective}
     A monomial space is a triple $\mathbf{W}=\left(W, X,\left(W_x\right)_{x \in X}\right)$ where,
     \begin{itemize}
         \item[$\imath$)] $\mathrm{W}$ is a finite dimensional complex vector space.
         \item[$\imath \imath$)] $\mathrm{X}$ is a finite set.
         \item[$\imath \imath \imath$)] $\left(W_x\right)_{x \in X}$ is a family of one dimensional subspaces of $W$, indexed by $X$ such that, $W=\bigoplus_{x \in X} W_x$.
     \end{itemize}
\end{definition}

Let \(H\) be a group. A \emph{monomial representation} of \(H\) on \(\mathbf{W}\) is a group homomorphism \(\Gamma: H \to \mathrm{GL}(W)\), such that for every \(h \in H\), \(\Gamma(h)\) permutes the subspaces \(W_x\); thus, \(\Gamma\) induces an action by permutation of \(H\) on \(X\). The subspace of \(H\)-invariant vectors is denoted by \(W^H = \{ w \in W \mid \Gamma(h) w = w, \ \forall h \in H \}\).

For each \(x \in X\), let \(\operatorname{Stab}_H(x)\) denote the stabilizer of \(x\) in \(H\), and let \(\mathcal{O}_H(x)\) denote the \(H\)-orbit of \(x\). We say that \(x \in X\) is \emph{regular} under the monomial action of \(H\) if \(\Gamma(h)\) acts as the identity on \(W_x\) for all \(h \in \operatorname{Stab}_H(x)\).

\begin{proposition}\label{prop:monomial_representation}
Let \(\mathbf{W} = (W, X, (W_x)_{x \in X})\) be a monomial representation of \(H\). Then:
\begin{enumerate}
    \item An element \(x \in X\) is regular if and only if \(\displaystyle \frac{1}{|H|} \sum_{h \in H} \Gamma(h) w_x \ne 0\), where \(w_x\) is a nonzero vector in \(W_x\).
    \item If \(x \in X\) is regular, then every element in its orbit \(\mathcal{O}_H(x)\) is regular.
    \item Let \(Y\) be a set of representatives for the regular \(H\)-orbits in \(X\). For each \(y \in Y\), choose a nonzero vector \(w_y \in W_y\), and let \(T_y\) be a left transversal for \(\operatorname{Stab}_H(y)\) in \(H\). Then the vectors \(u_y = \sum_{h \in T_y} \Gamma(h) w_y\) (for \(y \in Y\)) form a basis for \(W^H\).
    \item The dimension of \(W^H\) equals the number of regular \(H\)-orbits in \(X\).
\end{enumerate}
\end{proposition}

\begin{proof}
See \cite[Chapter 7, Lemma 9.1]{karpilovsky1985projective}.
\end{proof}

We now consider the Hilbert space \(\cH^0\) spanned by \(\Co(\cL)^0\), and set \(W = \cH^0\), \(X = \Co(\cL)^0\), and \(W_\varphi = \mathbb{C} \ket{\varphi}\) for each \(\varphi \in \Co(\cL)^0\). This forms a monomial space.

Let \(H = G^{|V|}\) be the group of functions from \(V\) to \(G\), and define the group homomorphism:
\[
\Gamma: H \to \mathrm{GL}(\cH^0), \quad (g_v)_{v \in V} \mapsto \prod_{v \in V} A_v^{g_v},
\]
where \(A_v^{g_v}\) acts on \(\cH^0\) as defined in Section~\ref{TM}. Then \(\cH^0\), together with the action \(\Gamma\), is a monomial representation of \(H\). The \(H\)-invariant subspace \(W^H\) is precisely the ground state space \(V_{gs}\).

To construct a basis for the ground state space, Proposition~\ref{prop:monomial_representation} indicates that we need to identify the regular elements of \(X\) under the action of \(H\). Before addressing the general case, we consider the torus as a specific example.

\subsection{Example: The Torus}

Let \(G\) be a finite group. Consider the lattice in Figure~\ref{toro}, which consists of two faces and a single vertex (star). This is the smallest $\Delta$-complex of a torus. The edges are labeled, and the orientations are determined by the ordering of the vertices.

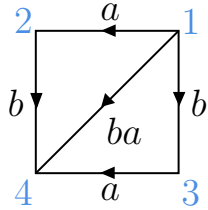
\begin{figure}[h!]
    \centering    

\tikzset{every picture/.style={line width=0.75pt}} 

\begin{tikzpicture}[x=0.45pt,y=0.45pt,yscale=-1,xscale=1]
 
\draw   (263,65.09) -- (383.03,65.09) -- (383.03,184.84) -- (263,184.84) -- cycle ;
\draw    (263,184.84) -- (383.03,65.09) ;
\draw  [color={rgb, 255:red, 0; green, 0; blue, 0 }  ,draw opacity=1 ][fill={rgb, 255:red, 0; green, 0; blue, 0 }  ,fill opacity=1 ][line width=0.75]  (383.5,126.58) -- (387.5,117.33) -- (379.5,117.33) -- cycle ;
\draw  [color={rgb, 255:red, 0; green, 0; blue, 0 }  ,draw opacity=1 ][fill={rgb, 255:red, 0; green, 0; blue, 0 }  ,fill opacity=1 ][line width=0.75]  (320.03,64.97) -- (329.29,68.94) -- (329.29,60.94) -- cycle ;
\draw  [color={rgb, 255:red, 0; green, 0; blue, 0 }  ,draw opacity=1 ][fill={rgb, 255:red, 0; green, 0; blue, 0 }  ,fill opacity=1 ][line width=0.75]  (263.5,128.93) -- (267.5,118.18) -- (259.5,118.18) -- cycle ;
\draw  [color={rgb, 255:red, 0; green, 0; blue, 0 }  ,draw opacity=1 ][fill={rgb, 255:red, 0; green, 0; blue, 0 }  ,fill opacity=1 ][line width=0.75]  (319.23,184.57) -- (328.49,188.54) -- (328.49,180.54) -- cycle ;
\draw  [fill={rgb, 255:red, 0; green, 0; blue, 0 }  ,fill opacity=1 ] (319.17,127.83) -- (321.65,118.99) -- (328.01,125.35) -- cycle ;

\draw (315,191.4) node [anchor=north west][inner sep=0.75pt]   [align=left] {{\large $a$}};
\draw (315,40) node [anchor=north west][inner sep=0.75pt]   [align=left] {{\large $a$}};
\draw (390.7,112.4) node [anchor=north west][inner sep=0.75pt]   [align=left] {{\large $b$}};
\draw (236.9,112.4) node [anchor=north west][inner sep=0.75pt]   [align=left] {{\large $b$}};
\draw (320,135.3) node [anchor=north west][inner sep=0.75pt]   [align=left] {{\large $ba$}};
\draw (242.3,188.6) node [anchor=north west][inner sep=0.75pt]  [font=\large] [align=left] {\textcolor[rgb]{0.29,0.56,0.89}{$4$}};
\draw (382.9,188.6) node [anchor=north west][inner sep=0.75pt]  [font=\large] [align=left] {\textcolor[rgb]{0.29,0.56,0.89}{$3$}};
\draw (242.3,43.6) node [anchor=north west][inner sep=0.75pt]  [font=\large] [align=left] {\textcolor[rgb]{0.29,0.56,0.89}{$2$}};
\draw (382.9,44.6) node [anchor=north west][inner sep=0.75pt]  [font=\large] [align=left] {\textcolor[rgb]{0.29,0.56,0.89}{$1$}};
\end{tikzpicture}
    \caption{$\Delta$-complex lattice for the torus.}  
    \label{toro}
\end{figure}

The subspace \(\cH^0\) is spanned by the states:
\[
\{ \ket{a,b} \mid a, b \in G, \ [a,b] = e \},
\]
where \([a,b] = a b a^{-1} b^{-1}\) is the commutator of \(a\) and \(b\). The only vertex in the lattice is associated with the operators \(A^x\) where $x\in G$. To determine whether the state \(\ket{a,b}\) is regular, we need to compute the scalar \(\alpha_{a,b}^x\) such that \(A^x \ket{a,b} = \alpha_{a,b}^x \ket{a,b}\) for each $x\in \operatorname{Stab}_G(x)$, and check whether \(\alpha_{a,b}^x = 1\).

The scalar \(\alpha_{a,b}^x\) is calculated using the 3-cocycle \(\alpha\) and the labels of the edges in the lattice. After performing the computation (for details see \cite[Section IV]{twisted} ), we find:
\begin{align*}
\alpha_{a,b}^x = \frac{\beta_x(a,b)}{\beta_x(b,a)},
\end{align*}
where \(\beta_x\) is defined as follows.

\begin{definition}\label{def:beta}
Let \(G\) be a group and \(\alpha: G^3 \to U(1)\) a normalized 3-cocycle. For each \(x \in G\), define the function \(\beta_x: G^2 \to U(1)\) by
\[
\beta_x(y,z) = \frac{\alpha(x,y,z) \alpha(y,z,z^{-1} y^{-1} x y z)}{\alpha(y,y^{-1} x y,z)}.
\]
\end{definition}
The  functions \(\beta_x\) are normalized  2-cocycles on \(C_G(g)=\{x\in G: gx=xg\}\)  and satisfy:
\[
\beta_x(y,z) \beta_x(yz,w) = \beta_x(y,zw) \beta_{y^{-1} x y}(z,w), \quad \forall y,z,w \in G.
\]

For a set \(\{g_1, \dots, g_n\} \subset G\), we define \(C_G(g_1, \dots, g_n) = \bigcap_{i=1}^n C_G(g_i)\).

The state \(\ket{a,b}\) is regular if and only if \(\alpha_{a,b}^x = 1\) for all \(x \in C_G(a,b)\). Thus, \(\ket{a,b}\) is regular if and only if the following map is trivial:
\[
\rho: C_G(a,b) \to \mathbb{C}^*, \quad x \mapsto \frac{\beta_x(a,b)}{\beta_x(b,a)}.
\]

\subsection{General Case: Surface of Genus \(n\)}

To determine the regular states for a surface of genus \(n\), we consider the simplest $\Delta$-complex of the surface, as illustrated in Figure~\ref{generon}. Similar to the torus case, the ordering of vertices cannot be arbitrary but can be suitably arranged.

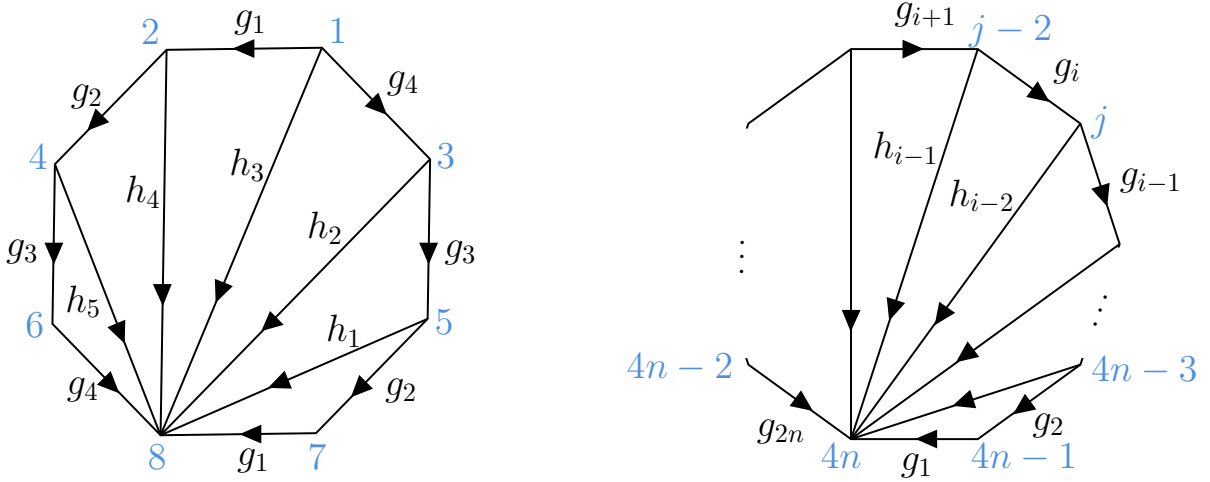
\begin{figure} 
    \centering
\tikzset{every picture/.style={line width=0.75pt}} 

\begin{tikzpicture}[x=0.75pt,y=0.75pt,yscale=-1,xscale=1]

\draw   (122.33,71.15) -- (200.71,70.1) -- (255.25,126.35) -- (254,206.93) -- (197.69,264.65) -- (119.31,265.69) -- (64.77,209.45) -- (66.02,128.87) -- cycle ;
\draw    (200.71,70.1) -- (119.31,265.69) ;
\draw    (255.25,126.35) -- (119.31,265.69) ;
\draw    (254,206.93) -- (119.31,265.69) ;
\draw    (122.33,71.15) -- (119.31,265.69) ;
\draw    (66.02,128.87) -- (119.31,265.69) ;
\draw  [color={rgb, 255:red, 0; green, 0; blue, 0 }  ,draw opacity=1 ][fill={rgb, 255:red, 0; green, 0; blue, 0 }  ,fill opacity=1 ][line width=0.75]  (65.5,178.25) -- (69.5,169) -- (61.5,169) -- cycle ;
\draw  [fill={rgb, 255:red, 0; green, 0; blue, 0 }  ,fill opacity=1 ] (97.61,243.64) -- (88.12,239.98) -- (93.95,234.15) -- cycle ;
\draw  [fill={rgb, 255:red, 0; green, 0; blue, 0 }  ,fill opacity=1 ] (232.61,103.14) -- (223.12,99.48) -- (228.95,93.65) -- cycle ;
\draw  [color={rgb, 255:red, 0; green, 0; blue, 0 }  ,draw opacity=1 ][fill={rgb, 255:red, 0; green, 0; blue, 0 }  ,fill opacity=1 ][line width=0.75]  (254.5,178.25) -- (258.5,169) -- (250.5,169) -- cycle ;
\draw  [color={rgb, 255:red, 0; green, 0; blue, 0 }  ,draw opacity=1 ][fill={rgb, 255:red, 0; green, 0; blue, 0 }  ,fill opacity=1 ][line width=0.75]  (160.36,265.14) -- (169.62,269.11) -- (169.62,261.11) -- cycle ;
\draw  [color={rgb, 255:red, 0; green, 0; blue, 0 }  ,draw opacity=1 ][fill={rgb, 255:red, 0; green, 0; blue, 0 }  ,fill opacity=1 ][line width=0.75]  (156.86,70.53) -- (166.12,74.5) -- (166.12,66.5) -- cycle ;
\draw  [fill={rgb, 255:red, 0; green, 0; blue, 0 }  ,fill opacity=1 ] (83.14,111.14) -- (86.7,101.65) -- (92.53,107.48) -- cycle ;
\draw  [fill={rgb, 255:red, 0; green, 0; blue, 0 }  ,fill opacity=1 ] (216.14,245.64) -- (219.7,236.15) -- (225.53,241.98) -- cycle ;
\draw  [fill={rgb, 255:red, 0; green, 0; blue, 0 }  ,fill opacity=1 ] (174.65,241.32) -- (181.7,233.98) -- (184.82,241.62) -- cycle ;
\draw  [fill={rgb, 255:red, 0; green, 0; blue, 0 }  ,fill opacity=1 ] (171,212.34) -- (174.67,203.86) -- (180.49,209.69) -- cycle ;
\draw  [fill={rgb, 255:red, 0; green, 0; blue, 0 }  ,fill opacity=1 ] (147.55,197.67) -- (147.21,187.5) -- (154.86,190.59) -- cycle ;
\draw  [color={rgb, 255:red, 0; green, 0; blue, 0 }  ,draw opacity=1 ][fill={rgb, 255:red, 0; green, 0; blue, 0 }  ,fill opacity=1 ][line width=0.75]  (120.43,199.25) -- (124.17,189.89) -- (116.18,189.89) -- cycle ;
\draw  [fill={rgb, 255:red, 0; green, 0; blue, 0 }  ,fill opacity=1 ] (100.74,218.32) -- (93.36,211.31) -- (100.98,208.15) -- cycle ;
\draw   (603,169.21) -- (583.23,230.05) -- (531.48,267.64) -- (467.52,267.64) -- (415.77,230.05) -- (396,169.21) -- (415.77,108.38) -- (467.52,70.78) -- (531.48,70.78) -- (583.23,108.38) -- cycle ;
\draw  [color={rgb, 255:red, 255; green, 255; blue, 255 }  ,draw opacity=1 ][fill={rgb, 255:red, 255; green, 255; blue, 255 }  ,fill opacity=1 ] (387.8,111.39) -- (559.82,111.39) -- (559.82,226.23) -- (387.8,226.23) -- cycle ;
\draw    (531.48,70.78) -- (467.52,267.64) ;
\draw    (467.52,70.78) -- (467.52,267.64) ;
\draw    (583.23,108.38) -- (467.52,267.64) ;
\draw    (583.23,230.05) -- (467.52,267.64) ;
\draw  [color={rgb, 255:red, 0; green, 0; blue, 0 }  ,draw opacity=1 ][fill={rgb, 255:red, 0; green, 0; blue, 0 }  ,fill opacity=1 ][line width=0.75]  (503,70.68) -- (494,66.8) -- (494,74.8) -- cycle ;
\draw  [fill={rgb, 255:red, 0; green, 0; blue, 0 }  ,fill opacity=1 ] (566.72,96.76) -- (556.77,94.63) -- (561.62,87.96) -- cycle ;
\draw  [fill={rgb, 255:red, 0; green, 0; blue, 0 }  ,fill opacity=1 ] (550.19,254.46) -- (555.29,245.66) -- (560.14,252.33) -- cycle ;
\draw  [color={rgb, 255:red, 0; green, 0; blue, 0 }  ,draw opacity=1 ][fill={rgb, 255:red, 0; green, 0; blue, 0 }  ,fill opacity=1 ][line width=0.75]  (501.18,267.51) -- (510.46,271.58) -- (510.46,263.58) -- cycle ;
\draw  [fill={rgb, 255:red, 0; green, 0; blue, 0 }  ,fill opacity=1 ] (519.83,250.45) -- (527.38,243.63) -- (529.96,251.47) -- cycle ;
\draw  [fill={rgb, 255:red, 0; green, 0; blue, 0 }  ,fill opacity=1 ] (447.35,252.73) -- (437.4,250.6) -- (442.25,243.93) -- cycle ;
\draw  [color={rgb, 255:red, 0; green, 0; blue, 0 }  ,draw opacity=1 ][fill={rgb, 255:red, 0; green, 0; blue, 0 }  ,fill opacity=1 ][line width=0.75]  (467.39,212.17) -- (471.61,202.92) -- (463.61,202.92) -- cycle ;
\draw  [fill={rgb, 255:red, 0; green, 0; blue, 0 }  ,fill opacity=1 ] (487.15,206.75) -- (486.1,196.64) -- (493.94,199.18) -- cycle ;
\draw  [fill={rgb, 255:red, 0; green, 0; blue, 0 }  ,fill opacity=1 ] (510.42,209.14) -- (512.37,199.16) -- (519.13,203.89) -- cycle ;
\draw  [color={rgb, 255:red, 255; green, 255; blue, 255 }  ,draw opacity=1 ][fill={rgb, 255:red, 255; green, 255; blue, 255 }  ,fill opacity=1 ] (551.82,171.46) -- (616.84,171.46) -- (616.84,226.23) -- (551.82,226.23) -- cycle ;
\draw    (603,169.21) -- (467.52,267.64) ;
\draw  [fill={rgb, 255:red, 0; green, 0; blue, 0 }  ,fill opacity=1 ] (521.43,228.12) -- (526.39,219.24) -- (531.35,225.84) -- cycle ;

\draw  [fill={rgb, 255:red, 0; green, 0; blue, 0 }  ,fill opacity=1 ] (596.35,149.09) -- (589.56,141.52) -- (597.4,138.97) -- cycle ;

\draw (156.5,271) node [anchor=north west][inner sep=0.75pt]   [align=left] {{\large $g_1$}};
\draw (231,235.5) node [anchor=north west][inner sep=0.75pt]   [align=left] {{\large $g_2$}};
\draw (261,166) node [anchor=north west][inner sep=0.75pt]   [align=left] {{\large $g_3$}};
\draw (232,80.5) node [anchor=north west][inner sep=0.75pt]   [align=left] {{\large $g_4$}};
\draw (155,50) node [anchor=north west][inner sep=0.75pt]   [align=left] {{\large $g_1$}};
\draw (72.5,87) node [anchor=north west][inner sep=0.75pt]   [align=left] {{\large $g_2$}};
\draw (40,165.5) node [anchor=north west][inner sep=0.75pt]   [align=left] {{\large $g_3$}};
\draw (70.5,235) node [anchor=north west][inner sep=0.75pt]   [align=left] {{\large $g_4$}};
\draw (201.4,203.5) node [anchor=north west][inner sep=0.75pt]   [align=left] {{\large $h_1$}};
\draw (191.8,151.5) node [anchor=north west][inner sep=0.75pt]   [align=left] {{\large $h_2$}};
\draw (152.6,119.5) node [anchor=north west][inner sep=0.75pt]   [align=left] {{\large $h_3$}};
\draw (100,133.1) node [anchor=north west][inner sep=0.75pt]   [align=left] {{\large $h_4$}};
\draw (69.9,188.45) node [anchor=north west][inner sep=0.75pt]   [align=left] {{\large $h_5$}};
\draw (111.3,268.6) node [anchor=north west][inner sep=0.75pt]  [font=\large] [align=left] {\textcolor[rgb]{0.29,0.56,0.89}{$8$}};
\draw (192.1,268.8) node [anchor=north west][inner sep=0.75pt]  [font=\large] [align=left] {\textcolor[rgb]{0.29,0.56,0.89}{$7$}};
\draw (256.1,200.8) node [anchor=north west][inner sep=0.75pt]  [font=\large] [align=left] {\textcolor[rgb]{0.29,0.56,0.89}{$5$}};
\draw (256.9,116.8) node [anchor=north west][inner sep=0.75pt]  [font=\large] [align=left] {\textcolor[rgb]{0.29,0.56,0.89}{$3$}};
\draw (202.5,53.6) node [anchor=north west][inner sep=0.75pt]  [font=\large] [align=left] {\textcolor[rgb]{0.29,0.56,0.89}{$1$}};
\draw (108.1,54.4) node [anchor=north west][inner sep=0.75pt]  [font=\large] [align=left] {\textcolor[rgb]{0.29,0.56,0.89}{$2$}};
\draw (51.3,117.6) node [anchor=north west][inner sep=0.75pt]  [font=\large] [align=left] {\textcolor[rgb]{0.29,0.56,0.89}{$4$}};
\draw (49.7,202.4) node [anchor=north west][inner sep=0.75pt]  [font=\large] [align=left] {\textcolor[rgb]{0.29,0.56,0.89}{$6$}};
\draw (407.68,186.5) node [anchor=north west][inner sep=0.75pt]  [rotate=-271.26] [align=left] {$\displaystyle \cdots $};
\draw (450.8,268.51) node [anchor=north west][inner sep=0.75pt]   [align=left] {{\large \textcolor[rgb]{0.29,0.56,0.89}{$4n$}}};
\draw (353,222.71) node [anchor=north west][inner sep=0.75pt]   [align=left] {{\large \textcolor[rgb]{0.29,0.56,0.89}{$4n-2$}}};
\draw (526.43,268.25) node [anchor=north west][inner sep=0.75pt]   [align=left] {{\large \textcolor[rgb]{0.29,0.56,0.89}{$4n-1$}}};
\draw (491.13,273.71) node [anchor=north west][inner sep=0.75pt]   [align=left] {{\large $g_1$}};
\draw (557.51,253.07) node [anchor=north west][inner sep=0.75pt]   [align=left] {{\large $g_2$}};
\draw (567.92,74.4) node [anchor=north west][inner sep=0.75pt]   [align=left] {{\large $g_i$}};
\draw (418.59,256.2) node [anchor=north west][inner sep=0.75pt]   [align=left] {{\large $g_{2n}$}};
\draw (585.08,211.2) node [anchor=north west][inner sep=0.75pt]  [rotate=-290.77] [align=left] {$\displaystyle \cdots $};
\draw (587.23,224.25) node [anchor=north west][inner sep=0.75pt]   [align=left] {{\large \textcolor[rgb]{0.29,0.56,0.89}{$4n-3$}}};
\draw (487.93,46) node [anchor=north west][inner sep=0.75pt]   [align=left] {{\large $g_{i+1}$}};
\draw (601.52,129.8) node [anchor=north west][inner sep=0.75pt]   [align=left] {{\large $g_{i-1}$}};
\draw (587.23,97.85) node [anchor=north west][inner sep=0.75pt]   [align=left] {{\large \textcolor[rgb]{0.29,0.56,0.89}{$j$}}};
\draw (525.63,52.05) node [anchor=north west][inner sep=0.75pt]   [align=left] {{\large \textcolor[rgb]{0.29,0.56,0.89}{$j-2$}}};
\draw (515.32,135) node [anchor=north west][inner sep=0.75pt]   [align=left] {{\large $h_{i-2}$}};
\draw (477,111) node [anchor=north west][inner sep=0.75pt]   [align=left] {{\large $h_{i-1}$}};
\end{tikzpicture}
    \caption{$\Delta$-complex for a surface of genus 2, and genus $n$.}
    \label{generon}
\end{figure}

\begin{proposition}
Let \(\Sigma\) be a surface of genus \(n\), and consider the $\Delta$-complex shown in Figure~\ref{generon}. A state \(\ket{g_1, \dots, g_{2n}} \in \cH^0\) is regular if and only if, for every \(x \in C_G(g_1, \dots, g_{2n})\), the following holds:
\[
\prod_{i=1}^{2n} \frac{\beta_x\left( g_i, g_{i+1} \cdots g_{2n} \right)}{\beta_x\left( g_i, g_{i-1} \cdots g_1 \right)} = 1.
\]
\end{proposition}

\begin{proof}
Since there is only one vertex in the lattice, the scalar obtained from the action of \(A^x\) on \(\ket{g_1, \dots, g_{2n}}\) can be expressed as the product of contributions from all edges involved. Each edge participates in two adjacent triangles, one oriented counterclockwise and the other clockwise. We analyze the contribution of each edge by considering its participation in these triangles.

For an edge labeled \(g_i\), we focus on the two triangles it is part of. In one triangle, the edge contributes a term involving \(\beta_x(g_i, g_{i-1} \cdots g_1)\) due to its clockwise orientation. In the adjacent triangle, it contributes \(\beta_x(g_i, g_{i+1} \cdots g_{2n})\) due to its counterclockwise orientation. By iterating this process for each edge in the $\Delta$-complex, we obtain the following product over the \(2n\) edges:
\[
\prod_{i=1}^{2n} \frac{\beta_x(g_i, g_{i+1} \cdots g_{2n})}{\beta_x(g_i, g_{i-1} \cdots g_1)}.
\]

This product must equal 1 for the state \(\ket{g_1, \dots, g_{2n}}\) to be regular. Thus, the condition in the proposition is derived.
\end{proof}

Recall that two 2-cocycles \(\alpha_1, \alpha_2 \in Z^2(G, U(1))\) are cohomologous if they differ by a coboundary. A 2-cocycle \(\beta: G^2 \to U(1)\) is called a \emph{coboundary} if there exists a function \(\gamma: G \to U(1)\) such that \(\beta(y, z) = \gamma(y) \gamma(z) \gamma(yz)^{-1}\). The group of cohomology classes is denoted by \(H^2(G, U(1))\).

\begin{proposition}\label{prop:cohomology_zero}
Let \(G\) be a finite group and \(\Sigma\) a surface of genus \(n\). Suppose that \(H^2\left( C_G(g_1, \dots, g_{2n}), U(1) \right) = 0\). Then every state \(\ket{g_1, \dots, g_{2n}} \in \cH^0\) is regular.
\end{proposition}

\begin{proof}
If \(H^2(C_G(g_1, \dots, g_{2n}), U(1)) = 0\), then every 2-cocycle of \(C_G(g_1, \dots, g_{2n})\) is a coboundary. This means that for each 3-cocycle \(\alpha\) and any \(x \in C_G(g_1, \dots, g_{2n})\), the corresponding 2-cocycle \(\beta_x\) can be expressed as:
\[
\beta_x(y, z) = \gamma_x(y) \gamma_x(z) \gamma_x(yz)^{-1}
\]
for some function \(\gamma_x: G \to U(1)\).

Substituting this into the product condition for regularity, we have:
\begin{align*}
\prod_{i=1}^{2n} \frac{\beta_x(g_i, g_{i+1} \cdots g_{2n})}{\beta_x(g_i, g_{i-1} \cdots g_1)}
&= \prod_{i=1}^{2n} \frac{\gamma_x(g_i) \gamma_x(g_{i+1} \cdots g_{2n}) \gamma_x(g_i g_{i+1} \cdots g_{2n})^{-1}}{\gamma_x(g_i) \gamma_x(g_{i-1} \cdots g_1) \gamma_x(g_i g_{i-1} \cdots g_1)^{-1}} \\
&= \prod_{i=1}^{2n} \frac{\gamma_x(g_{i+1} \cdots g_{2n}) \gamma_x(g_i \cdots g_{2n})^{-1}}{\gamma_x(g_{i-1} \cdots g_1) \gamma_x(g_i \cdots g_1)^{-1}}.
\end{align*}
This telescoping product simplifies to:
\[
\frac{\gamma_x(g_1 \cdots g_{2n})^{-1}}{\gamma_x(g_{2n} \cdots g_1)^{-1}}.
\]
Since \(\ket{g_1, \dots, g_{2n}} \in \Co(\cL)^0\), it follows that:
\[
g_1^{-1} \cdots g_{2n}^{-1} g_1 \cdots g_{2n} = e,
\]
which implies \(g_1 \cdots g_{2n} = g_{2n} \cdots g_1\). Thus:
\[
\frac{\gamma_x(g_1 \cdots g_{2n})^{-1}}{\gamma_x(g_{2n} \cdots g_1)^{-1}} = 1.
\]
Hence, \(A^x \ket{g_1, \dots, g_{2n}} = \ket{g_1, \dots, g_{2n}}\) for all \(x \in C_G(g_1, \dots, g_{2n})\), proving that \(\ket{g_1, \dots, g_{2n}}\) is regular.
\end{proof}

\begin{corollary}\label{cor:regular_states}
If \(H^2(C_G(g_1, \dots, g_{2n}), U(1)) = 0\) for every \(\ket{g_1, \dots, g_{2n}} \in \cH^0\), then:
\begin{enumerate}
    \item Every basis state in \(\cH^0\) is regular.
    \item The dimension of \(V_{gs}\) equals the number of orbits in \(\operatorname{Hom}(\pi_1(\Sigma), G)\) under the action of \(G\).
\end{enumerate}
\end{corollary}

\begin{proof}
The first part follows directly from Proposition~\ref{prop:cohomology_zero}. Since the result does not depend on the choice of the $3$-cocycle, the ground state dimension remains constant regardless of the specific $3$-cocycle used to define the model. Therefore, we can consider the trivial $3$-cocycle, in which case we recover Kitaev's model, where it was shown in \cite{shaw2020} that the dimension of \(V_{gs}\) is equal to the number of orbits in \(\operatorname{Hom}(\pi_1(\Sigma), G)\) under the action of \(G\).
\end{proof}

There are several groups that satisfy the hypotheses of Corollary~\ref{cor:regular_states}, such as the quaternion group, all dihedral groups \(D_n\) with odd \(n\), and cyclic groups. For the cyclic group \(\mathbb{Z}_m\), it is straightforward to verify that, for a surface of genus \(n\), the dimension of the ground state is \(m^{2n}\). More interesting examples are the following.

\subsection{Example: Dihedral Group \(D_3\)}
    Take $D_n = \{r,s\mid r^n=s^2=0, sr=r^{-1}s\}$, and consider a subgroup $G$ of $D_n$. If $n$ is odd then $H^2(G,U(1))=0$. Taking $D_3$ as an example, we find that all its colors will be regular. However, unlike in the case of a cyclic group, some orbits in $D_3$ will composed by more than a single element. For a surface of genus two, consider the color $\ket{rr^2rr^2}$
     \begin{center}
\begin{tikzpicture}[x=0.4pt,y=0.4pt,yscale=-1,xscale=1]

\draw    (77.4,185.91) -- (188.91,138.8) ;
\draw  [fill={rgb, 255:red, 0; green, 0; blue, 0 }  ,fill opacity=1 ] (112.65,171.12) -- (119.7,163.78) -- (122.82,171.42) -- cycle ;
\draw   (188.91,138.8) -- (142.91,185.45) -- (77.4,185.91) -- (30.75,139.91) -- (30.29,74.4) -- (76.29,27.75) -- (141.8,27.29) -- (188.45,73.29) -- cycle ;
\draw    (77.4,185.91) -- (188.45,73.29) ;
\draw  [fill={rgb, 255:red, 0; green, 0; blue, 0 }  ,fill opacity=1 ] (115.66,147.31) -- (119.33,137.82) -- (125.16,143.66) -- cycle ;
\draw    (77.4,185.91) -- (141.8,27.29) ;
\draw  [fill={rgb, 255:red, 0; green, 0; blue, 0 }  ,fill opacity=1 ] (100.55,128.67) -- (100.21,118.5) -- (107.86,121.59) -- cycle ;
\draw    (77.4,185.91) -- (76.29,27.75) ;
\draw  [fill={rgb, 255:red, 0; green, 0; blue, 0 }  ,fill opacity=1 ] (76.96,132.01) -- (72.83,122.71) -- (81.08,122.71) -- cycle ;
\draw    (77.4,185.91) -- (30.29,74.4) ;
\draw  [fill={rgb, 255:red, 0; green, 0; blue, 0 }  ,fill opacity=1 ] (62.74,151.32) -- (55.36,144.31) -- (62.98,141.15) -- cycle ;
\draw  [fill={rgb, 255:red, 0; green, 0; blue, 0 }  ,fill opacity=1 ] (112.31,185.69) -- (121.61,181.57) -- (121.61,189.82) -- cycle ;
\draw  [fill={rgb, 255:red, 0; green, 0; blue, 0 }  ,fill opacity=1 ] (102.31,27.54) -- (111.61,23.42) -- (111.61,31.67) -- cycle ;
\draw  [fill={rgb, 255:red, 0; green, 0; blue, 0 }  ,fill opacity=1 ] (30.62,118.19) -- (26.5,108.89) -- (34.75,108.89) -- cycle ;
\draw  [fill={rgb, 255:red, 0; green, 0; blue, 0 }  ,fill opacity=1 ] (188.62,118.19) -- (184.5,108.89) -- (192.75,108.89) -- cycle ;
\draw  [fill={rgb, 255:red, 0; green, 0; blue, 0 }  ,fill opacity=1 ] (157,171.16) -- (160.66,161.67) -- (166.49,167.5) -- cycle ;
\draw  [fill={rgb, 255:red, 0; green, 0; blue, 0 }  ,fill opacity=1 ] (48,56.5) -- (51.66,47) -- (57.49,52.84) -- cycle ;
\draw  [fill={rgb, 255:red, 0; green, 0; blue, 0 }  ,fill opacity=1 ] (55.24,164.5) -- (45.75,160.84) -- (51.58,155) -- cycle ;
\draw  [fill={rgb, 255:red, 0; green, 0; blue, 0 }  ,fill opacity=1 ] (170.91,55.83) -- (161.42,52.17) -- (167.25,46.34) -- cycle ;

\draw (111.64,192.76) node [anchor=north west][inner sep=0.75pt]    {$r$};
\draw (195.64,102.4) node [anchor=north west][inner sep=0.75pt]    {$r$};
\draw (104.97,3.13) node [anchor=north west][inner sep=0.75pt]    {$r$};
\draw (3.64,102.4) node [anchor=north west][inner sep=0.75pt]    {$r$};
\draw (168.31,155.73) node [anchor=north west][inner sep=0.75pt]    {$r^{2}$};
\draw (174.54,28.4) node [anchor=north west][inner sep=0.75pt]    {$r^{2}$};
\draw (12.64,155.73) node [anchor=north west][inner sep=0.75pt]    {$r^{2}$};
\draw (12.97,28.4) node [anchor=north west][inner sep=0.75pt]    {$r^{2}$};
\end{tikzpicture}
\end{center}
According to Proposition \ref{prop:monomial_representation},  to determine a basis for the ground state, we first identify the stabilizer of $\ket{rr^2rr^2}$, which is $\{ A^0, A^r,A^{ r^2} \}$. Next, we consider its left coset (which in this case is only one) $\{ A^s, A^{sr},A^{sr^2} \}$. Therefore, a suitable left transversal would be $\{ A^0, A^s \}$, hence the basis element of the ground state corresponding to the orbit of $\ket{rr^2rr^2}$ is $A^0\ket {rr^2rr^2}+ A^s\ket{rr^2rr^2}$, i.e.
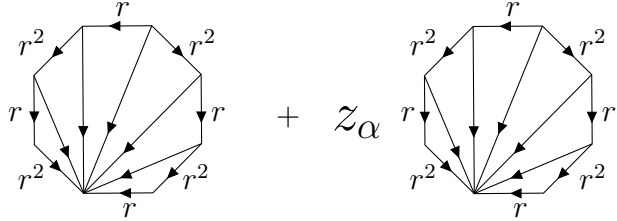
\begin{center}
\begin{tikzpicture}[x=0.4pt,y=0.4pt,yscale=-1,xscale=1]

\draw    (77.4,185.91) -- (188.91,138.8) ;
\draw  [fill={rgb, 255:red, 0; green, 0; blue, 0 }  ,fill opacity=1 ] (112.65,171.12) -- (119.7,163.78) -- (122.82,171.42) -- cycle ;
\draw   (188.91,138.8) -- (142.91,185.45) -- (77.4,185.91) -- (30.75,139.91) -- (30.29,74.4) -- (76.29,27.75) -- (141.8,27.29) -- (188.45,73.29) -- cycle ;
\draw    (77.4,185.91) -- (188.45,73.29) ;
\draw  [fill={rgb, 255:red, 0; green, 0; blue, 0 }  ,fill opacity=1 ] (115.66,147.31) -- (119.33,137.82) -- (125.16,143.66) -- cycle ;
\draw    (77.4,185.91) -- (141.8,27.29) ;
\draw  [fill={rgb, 255:red, 0; green, 0; blue, 0 }  ,fill opacity=1 ] (100.55,128.67) -- (100.21,118.5) -- (107.86,121.59) -- cycle ;
\draw    (77.4,185.91) -- (76.29,27.75) ;
\draw  [fill={rgb, 255:red, 0; green, 0; blue, 0 }  ,fill opacity=1 ] (76.96,132.01) -- (72.83,122.71) -- (81.08,122.71) -- cycle ;
\draw    (77.4,185.91) -- (30.29,74.4) ;
\draw  [fill={rgb, 255:red, 0; green, 0; blue, 0 }  ,fill opacity=1 ] (62.74,151.32) -- (55.36,144.31) -- (62.98,141.15) -- cycle ;
\draw  [fill={rgb, 255:red, 0; green, 0; blue, 0 }  ,fill opacity=1 ] (112.31,185.69) -- (121.61,181.57) -- (121.61,189.82) -- cycle ;
\draw  [fill={rgb, 255:red, 0; green, 0; blue, 0 }  ,fill opacity=1 ] (102.31,27.54) -- (111.61,23.42) -- (111.61,31.67) -- cycle ;
\draw  [fill={rgb, 255:red, 0; green, 0; blue, 0 }  ,fill opacity=1 ] (30.62,118.19) -- (26.5,108.89) -- (34.75,108.89) -- cycle ;
\draw  [fill={rgb, 255:red, 0; green, 0; blue, 0 }  ,fill opacity=1 ] (188.62,118.19) -- (184.5,108.89) -- (192.75,108.89) -- cycle ;
\draw  [fill={rgb, 255:red, 0; green, 0; blue, 0 }  ,fill opacity=1 ] (157,171.16) -- (160.66,161.67) -- (166.49,167.5) -- cycle ;
\draw  [fill={rgb, 255:red, 0; green, 0; blue, 0 }  ,fill opacity=1 ] (48,56.5) -- (51.66,47) -- (57.49,52.84) -- cycle ;
\draw  [fill={rgb, 255:red, 0; green, 0; blue, 0 }  ,fill opacity=1 ] (55.24,164.5) -- (45.75,160.84) -- (51.58,155) -- cycle ;
\draw  [fill={rgb, 255:red, 0; green, 0; blue, 0 }  ,fill opacity=1 ] (170.91,55.83) -- (161.42,52.17) -- (167.25,46.34) -- cycle ;
\draw    (446.76,186.18) -- (558.27,139.07) ;
\draw  [fill={rgb, 255:red, 0; green, 0; blue, 0 }  ,fill opacity=1 ] (482.01,171.38) -- (489.06,164.05) -- (492.18,171.69) -- cycle ;
\draw   (558.27,139.07) -- (512.27,185.72) -- (446.76,186.18) -- (400.11,140.18) -- (399.65,74.66) -- (445.65,28.01) -- (511.17,27.56) -- (557.81,73.56) -- cycle ;
\draw    (446.76,186.18) -- (557.81,73.56) ;
\draw  [fill={rgb, 255:red, 0; green, 0; blue, 0 }  ,fill opacity=1 ] (485.03,147.58) -- (488.69,138.09) -- (494.52,143.93) -- cycle ;
\draw    (446.76,186.18) -- (511.17,27.56) ;
\draw  [fill={rgb, 255:red, 0; green, 0; blue, 0 }  ,fill opacity=1 ] (469.91,128.94) -- (469.57,118.77) -- (477.22,121.86) -- cycle ;
\draw    (446.76,186.18) -- (445.65,28.01) ;
\draw  [fill={rgb, 255:red, 0; green, 0; blue, 0 }  ,fill opacity=1 ] (446.32,132.28) -- (442.19,122.98) -- (450.44,122.98) -- cycle ;
\draw    (446.76,186.18) -- (399.65,74.66) ;
\draw  [fill={rgb, 255:red, 0; green, 0; blue, 0 }  ,fill opacity=1 ] (432.1,151.59) -- (424.72,144.58) -- (432.34,141.42) -- cycle ;
\draw  [fill={rgb, 255:red, 0; green, 0; blue, 0 }  ,fill opacity=1 ] (481.67,185.96) -- (490.97,181.83) -- (490.97,190.08) -- cycle ;
\draw  [fill={rgb, 255:red, 0; green, 0; blue, 0 }  ,fill opacity=1 ] (471.67,27.81) -- (480.97,23.68) -- (480.97,31.93) -- cycle ;
\draw  [fill={rgb, 255:red, 0; green, 0; blue, 0 }  ,fill opacity=1 ] (399.98,118.46) -- (395.86,109.16) -- (404.11,109.16) -- cycle ;
\draw  [fill={rgb, 255:red, 0; green, 0; blue, 0 }  ,fill opacity=1 ] (557.98,118.46) -- (553.86,109.16) -- (562.11,109.16) -- cycle ;
\draw  [fill={rgb, 255:red, 0; green, 0; blue, 0 }  ,fill opacity=1 ] (526.36,171.43) -- (530.02,161.94) -- (535.85,167.77) -- cycle ;
\draw  [fill={rgb, 255:red, 0; green, 0; blue, 0 }  ,fill opacity=1 ] (417.36,56.76) -- (421.02,47.27) -- (426.85,53.1) -- cycle ;
\draw  [fill={rgb, 255:red, 0; green, 0; blue, 0 }  ,fill opacity=1 ] (424.6,164.76) -- (415.11,161.1) -- (420.95,155.27) -- cycle ;
\draw  [fill={rgb, 255:red, 0; green, 0; blue, 0 }  ,fill opacity=1 ] (540.27,56.1) -- (530.78,52.44) -- (536.61,46.6) -- cycle ;

\draw (111.64,192.76) node [anchor=north west][inner sep=0.75pt]    {$r$};
\draw (195.64,102.4) node [anchor=north west][inner sep=0.75pt]    {$r$};
\draw (104.97,3.13) node [anchor=north west][inner sep=0.75pt]    {$r$};
\draw (3.64,102.4) node [anchor=north west][inner sep=0.75pt]    {$r$};
\draw (168.31,155.73) node [anchor=north west][inner sep=0.75pt]    {$r^{2}$};
\draw (174.54,28.4) node [anchor=north west][inner sep=0.75pt]    {$r^{2}$};
\draw (12.64,155.73) node [anchor=north west][inner sep=0.75pt]    {$r^{2}$};
\draw (12.97,28.4) node [anchor=north west][inner sep=0.75pt]    {$r^{2}$};
\draw (481,192.76) node [anchor=north west][inner sep=0.75pt]    {$r$};
\draw (565,102.4) node [anchor=north west][inner sep=0.75pt]    {$r$};
\draw (474.33,3.13) node [anchor=north west][inner sep=0.75pt]    {$r$};
\draw (375,102.4) node [anchor=north west][inner sep=0.75pt]    {$r$};
\draw (543.67,155.73) node [anchor=north west][inner sep=0.75pt]    {$r^{2}$};
\draw (543.67,28.4) node [anchor=north west][inner sep=0.75pt]    {$r^{2}$};
\draw (383,155.73) node [anchor=north west][inner sep=0.75pt]    {$r^{2}$};
\draw (383.33,28.4) node [anchor=north west][inner sep=0.75pt]    {$r^{2}$};
\draw (310,101.4) node [anchor=north west][inner sep=0.75pt]    {\LARGE$z_\alpha$};
\draw (257,101.4) node [anchor=north west][inner sep=0.75pt]    {$+$};
\end{tikzpicture}
\end{center}
    Where $z_\alpha \in \C$,  its value will depend on the 3-cocycle used to define the model, but regardless of this, the ground state dimension for $D_3$ and the surface of genus two will always be $116$.

\subsection{Example: Group with Nontrivial Second Cohomology}
   Let us consider an example of a group where the dimension of the ground state depends on the 3-cocycle. Let $G= \Z_4^3$ and $\Sigma $ be the surface of genus one. Firstly, note that if we take the trivial 3-cocycle, as the group is abelian, then the ground state will have dimension $|\Z_4^3|^2=4^6$. However, for the 3-cocycle $\alpha: \Z_4^3 \times \Z_4^3 \times \Z_4^3 \to U(1) $ defined by $$\alpha((a_1,a_2,a_3),(b_1 ,b_2,b_3),(c_1,c_2,c_3)) = \operatorname{exp} \left( \frac{2 \pi i a_1b_2c_3}{4}\right),$$
    We will have colors that are not regular, resulting in a smaller dimension for the ground state. For instance, the state $\ket{(1,2,3)(3,3,2)}$,
    
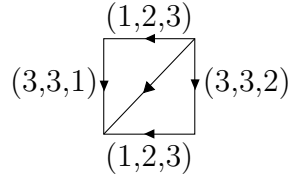
\begin{center}
\begin{tikzpicture}[x=0.45pt,y=0.45pt,yscale=-1,xscale=1]

\draw   (330.49,130.1) -- (407.02,130.1) -- (407.02,210.41) -- (330.49,210.41) -- cycle ;
\draw    (330.49,210.41) -- (407.02,130.1) ;
\draw  [color={rgb, 255:red, 0; green, 0; blue, 0 }  ,draw opacity=1 ][fill={rgb, 255:red, 0; green, 0; blue, 0 }  ,fill opacity=1 ][line width=0.75]  (407.32,171.34) -- (409.87,165.14) -- (404.77,165.14) -- cycle ;
\draw  [color={rgb, 255:red, 0; green, 0; blue, 0 }  ,draw opacity=1 ][fill={rgb, 255:red, 0; green, 0; blue, 0 }  ,fill opacity=1 ][line width=0.75]  (366.87,130.02) -- (372.76,132.68) -- (372.76,127.32) -- cycle ;
\draw  [color={rgb, 255:red, 0; green, 0; blue, 0 }  ,draw opacity=1 ][fill={rgb, 255:red, 0; green, 0; blue, 0 }  ,fill opacity=1 ][line width=0.75]  (330.81,172.41) -- (333.36,166.21) -- (328.26,166.21) -- cycle ;
\draw  [color={rgb, 255:red, 0; green, 0; blue, 0 }  ,draw opacity=1 ][fill={rgb, 255:red, 0; green, 0; blue, 0 }  ,fill opacity=1 ][line width=0.75]  (366.34,210.23) -- (372.25,212.89) -- (372.25,207.52) -- cycle ;
\draw  [fill={rgb, 255:red, 0; green, 0; blue, 0 }  ,fill opacity=1 ] (364.17,175.33) -- (366.65,166.49) -- (373.01,172.85) -- cycle ;

\draw (329.6,215.18) node [anchor=north west][inner sep=0.75pt]   [align=left] {(1,2,3)};
\draw (329.6,97.18) node [anchor=north west][inner sep=0.75pt]   [align=left] {(1,2,3)};
\draw (411.6,152.18) node [anchor=north west][inner sep=0.75pt]   [align=left] {(3,3,2)};
\draw (249.6,152.18) node [anchor=north west][inner sep=0.75pt]   [align=left] {(3,3,2)};
\end{tikzpicture}
    \end{center}
is not regular because $A^{(0,0,1)}\ket{(1,2,3)(3,3,2)} = i \ket{(1,2,3)(3,3 ,2)}$. In fact, for this 3-cocycle, the dimension of the ground state is $400$, which is smaller than $4^6$.

\section{Local topological quantum order}
\label{sec:local quantum smooth}
In this section, we study the axiomatization of Local Topological Order (LTO), originally proposed in \cite{LTO} for lattices where all faces are squares. Here, we adapt this axiomatization to arbitrary lattices.

The LTO axioms are based on regions of the surface. In the original case (square lattices), these regions were rectangular. For arbitrary lattices, the regions will be convex regions, bounded by a cycle $\mathcal{C}$ of the lattice $\mathcal{L}$. 

\begin{definition} \hfill
\begin{enumerate}
    \item Let $\cC$ be a cycle of $\cL$. A \emph{smooth-rough function} on $\cC$ is a function $\mathcal{F} : \{l \mid l \in \cC \} \to \{ \text{smooth, rough} \}$. We say that an edge $l \in \cC$ is \emph{smooth} if $\mathcal{F}(l) = \text{smooth}$, and \emph{rough} if $\mathcal{F}(l) = \text{rough}$.
    
    \item A region $\Lambda$ is a convex set bounded by a cycle $\partial \Lambda$, which is equipped with a smooth-rough function. For simplicity, we will also identify $\Lambda$ with the set of edges it contains. Figure \ref{ejemploregion} illustrates two examples of such regions, each consisting of five edges.
\end{enumerate}   
\end{definition}

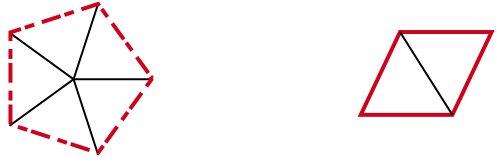
\begin{figure}[h!] 
    \centering
\tikzset{every picture/.style={line width=0.75pt}}
\begin{tikzpicture}[x=0.55pt,y=0.55pt,yscale=-1,xscale=1]

\draw  [color={rgb, 255:red, 208; green, 2; blue, 27 }  ,draw opacity=1 ][line width=1.5]  (362,40) -- (425,40) -- (398,97) -- (335,97) -- cycle ;
\draw    (154.34,20.67) -- (137.5,72.5) ;
\draw    (137.5,72.5) -- (192,72.5) ;
\draw    (137.5,72.5) -- (154.34,124.33) ;
\draw    (137.5,72.5) -- (93.41,104.53) ;
\draw    (93.41,40.47) -- (137.5,72.5) ;
\draw    (362,40) -- (398,97) ;
\draw  [color={rgb, 255:red, 208; green, 2; blue, 27 }  ,draw opacity=1 ][dash pattern={on 3.75pt off 3pt on 7.5pt off 1.5pt}][line width=1.5]  (191.59,71.93) -- (154.34,123.19) -- (94.08,103.61) -- (94.08,40.25) -- (154.34,20.67) -- cycle ;

\end{tikzpicture}

    \caption{On the left, there is a region bounded by a cycle whose edges are represented by dashed lines, indicating that the edges of the cycle are rough and do not belong to the region. On the right, there is a region bounded by a cycle with solid lines for edges, signifying that the edges of the cycle are smooth and are included in the region.}
    \label{ejemploregion}
\end{figure}

\begin{definition} 
Let $\cL$ be a lattice and $\Lambda$ be a region. Let $I \subset \partial \Lambda$ be a path. $I$ is a \emph{smooth interval} if every edge of $I$ is smooth; and $I$ is a \emph{rough interval} if every edge of $I$ is rough.
\end{definition}

\begin{definition}[Nets of algebras]
    Let $\cL$ be a lattice. A \emph{net of algebras} on $\cL$ in the ambient $\mathrm{C}^*$-algebra $\mathfrak{A}$ is an assignment of a $\mathrm{C}^*$-subalgebra $\mathfrak{A}(\Lambda) \subset \mathfrak{A}$ to each region $\Lambda \subset \cL$ such that
    \begin{itemize}
        \item[1)] $\mathfrak{A}(\emptyset)=\mathbb{C} 1_{\mathfrak{A}}$.
        \item[2)] If $\Lambda \subset \Delta$, then $\mathfrak{A}(\Lambda) \subset \mathfrak{A}(\Delta)$.
        \item[3)] If $\Lambda \cap \Delta=\emptyset$, then $[\mathfrak{A}(\Lambda), \mathfrak{A}(\Delta)]=0$.
        \item[4)] $\bigcup_{\Lambda} \mathfrak{A}(\Lambda)$ is norm dense in $\mathfrak{A}$.       
    \end{itemize}
\end{definition}

\begin{definition}[Net of projections]
     A net of projections is an assignment of an orthogonal projector $p_{\Lambda} \in \mathfrak{A}(\Lambda)$ such that $p_{\Delta} p_{\Lambda} p_{\Delta}=p_{\Delta}$ if $\Lambda \subseteq \Delta$.
\end{definition}

\begin{definition}
Let $\Lambda$ and $\Delta$ be regions, and let $I = \partial \Lambda \cap \partial \Delta$. Assume that $I \neq \partial \Delta$ and that $I$ is either a smooth or rough interval.
\begin{itemize}
    \item[$\imath$)] $\Lambda$ is said to be \emph{completely surrounded} by $\Delta$, denoted $\Lambda \ll \Delta$, if $I=\emptyset$ and for every vertex $v \in \Lambda$ , the closed star $\overline{s(v)}$ is contained in $\Delta$, i.e., $\overline{s(v)} \subset \Delta$.

    \item[$\imath \imath$)]If $I \neq \emptyset$,
\begin{itemize}
    \item If $I$ is smooth: $\Lambda$ is said to be \emph{surrounded} by $\Delta$, denoted $\Lambda \Subset \Delta$, if for every vertex $v \in \Lambda \setminus I$, $\overline{s(v)} \subset \Delta$.
    
    \item If $I$ is rough: $\Lambda$ is said to be \emph{surrounded} by $\Delta$, denoted $\Lambda \Subset \Delta$, if for every vertex $v \in \Lambda$ such that $\overline{s(v)} \cap I = \emptyset$, it holds that $\overline{s(v)} \subset \Delta$.
\end{itemize}
\end{itemize}
\end{definition}

In Figure \ref{examplesurrounded}, examples of both surrounded and completely surrounded regions are illustrated.

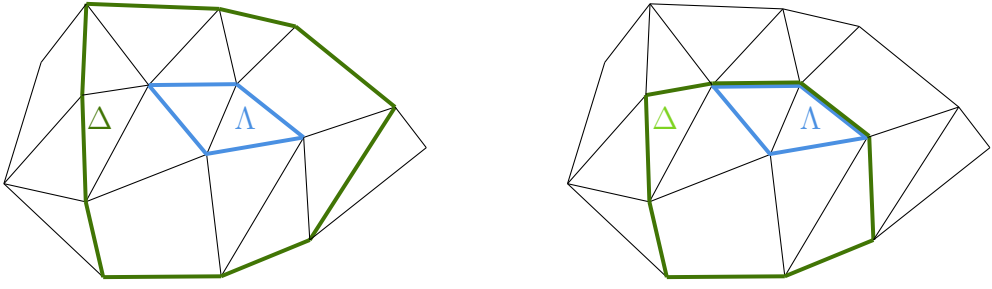
\begin{figure}[h!]
    \centering 
\begin{tikzpicture}[x=0.65pt,y=0.65pt,yscale=-1,xscale=1]

\draw [color={rgb, 255:red, 65; green, 117; blue, 5 }  ,draw opacity=1 ][line width=1.5]    (384.98,65.24) -- (423.65,58.38) ;
\draw [color={rgb, 255:red, 65; green, 117; blue, 5 }  ,draw opacity=1 ][line width=1.5]    (517.39,149.58) -- (514.95,89.05) ;
\draw    (137.83,170.74) -- (129.37,99.75) ;
\draw    (59.03,127.38) -- (129.37,99.75) ;
\draw    (59.03,127.38) -- (95.65,59.38) ;
\draw    (59.03,127.38) -- (11.39,116.8) ;
\draw    (69.17,171.22) -- (11.39,116.8) ;
\draw    (11.39,116.8) -- (56.98,65.24) ;
\draw    (11.39,116.8) -- (33.06,46.14) ;
\draw    (137.83,170.74) -- (185.84,89.85) ;
\draw    (129.37,99.75) -- (146.85,58.83) ;
\draw    (146.85,58.83) -- (136.71,14.99) ;
\draw    (95.65,59.38) -- (136.71,14.99) ;
\draw    (146.85,58.83) -- (181.16,25.24) ;
\draw [color={rgb, 255:red, 65; green, 117; blue, 5 } ,draw opacity=1 ][line width=1.5]    (136.71,14.99) -- (181.16,25.24) ;
\draw    (59.39,12.05) -- (95.65,59.38) ;
\draw [color={rgb, 255:red, 65; green, 117; blue, 5 }  ,draw opacity=1 ][line width=1.5]    (59.39,12.05) -- (136.71,14.99) ;
\draw [color={rgb, 255:red, 65; green, 117; blue, 5 }  ,draw opacity=1 ][line width=1.5]    (59.39,12.05) -- (56.98,65.24) ;
\draw    (185.84,89.85) -- (239.06,72.15) ;
\draw [color={rgb, 255:red, 65; green, 117; blue, 5 }  ,draw opacity=1 ][line width=1.5]    (181.16,25.24) -- (239.06,72.15) ;
\draw [color={rgb, 255:red, 65; green, 117; blue, 5 }  ,draw opacity=1 ][line width=1.5]    (189.39,149.58) -- (239.06,72.15) ;
\draw    (239.06,72.15) -- (257.2,95.82) ;
\draw    (189.39,149.58) -- (257.2,95.82) ;
\draw [color={rgb, 255:red, 74; green, 144; blue, 226 }  ,draw opacity=1 ][line width=1.5]    (128.62,99.61) -- (185.84,89.85) ;
\draw [color={rgb, 255:red, 74; green, 144; blue, 226 }  ,draw opacity=1 ][line width=1.5]    (146.85,58.83) -- (95.65,59.38) ;
\draw [color={rgb, 255:red, 74; green, 144; blue, 226 }  ,draw opacity=1 ][line width=1.5]    (146.85,58.83) -- (185.84,89.85) ;
\draw [color={rgb, 255:red, 65; green, 117; blue, 5 }  ,draw opacity=1 ][line width=1.5]    (59.03,127.38) -- (69.17,171.22) ;
\draw [color={rgb, 255:red, 65; green, 117; blue, 5 }  ,draw opacity=1 ][line width=1.5]    (59.03,127.38) -- (56.98,65.24) ;
\draw [color={rgb, 255:red, 65; green, 117; blue, 5 }  ,draw opacity=1 ][line width=1.5]    (137.83,170.74) -- (189.39,149.58) ;
\draw    (33.06,46.14) -- (59.39,12.05) ;
\draw [color={rgb, 255:red, 74; green, 144; blue, 226 }  ,draw opacity=1 ][line width=1.5]    (129.37,99.75) -- (95.65,59.38) ;
\draw [color={rgb, 255:red, 65; green, 117; blue, 5 } ,draw opacity=1 ][line width=1.5]    (137.83,170.74) -- (69.17,171.22) ;
\draw    (56.98,65.24) -- (95.65,59.38) ;
\draw    (189.39,149.58) -- (185.84,89.85) ;
\draw    (465.83,170.74) -- (457.37,99.75) ;
\draw    (387.03,127.38) -- (457.37,99.75) ;
\draw    (387.03,127.38) -- (423.65,59.38) ;
\draw    (387.03,127.38) -- (339.39,116.8) ;
\draw    (397.17,171.22) -- (339.39,116.8) ;
\draw    (339.39,116.8) -- (384.98,65.24) ;
\draw    (339.39,116.8) -- (361.06,46.14) ;
\draw    (465.83,170.74) -- (513.84,89.85) ;
\draw    (457.37,99.75) -- (474.85,58.83) ;
\draw    (474.85,58.83) -- (464.71,14.99) ;
\draw    (423.65,59.38) -- (464.71,14.99) ;
\draw    (474.85,58.83) -- (509.16,25.24) ;
\draw    (464.71,14.99) -- (509.16,25.24) ;
\draw    (387.39,12.05) -- (423.65,59.38) ;
\draw    (387.39,12.05) -- (464.71,14.99) ;
\draw    (387.39,12.05) -- (384.98,65.24) ;
\draw    (513.84,89.85) -- (567.06,72.15) ;
\draw    (509.16,25.24) -- (567.06,72.15) ;
\draw    (517.39,149.58) -- (567.06,72.15) ;
\draw    (567.06,72.15) -- (585.2,95.82) ;
\draw    (517.39,149.58) -- (585.2,95.82) ;
\draw [color={rgb, 255:red, 74; green, 144; blue, 226 }  ,draw opacity=1 ][line width=1.5]    (456.62,99.61) -- (513.84,89.85) ;
\draw [color={rgb, 255:red, 74; green, 144; blue, 226 }  ,draw opacity=1 ][line width=1.5]    (474.85,59.83) -- (423.65,60.38) ;
\draw [color={rgb, 255:red, 74; green, 144; blue, 226 }  ,draw opacity=1 ][line width=1.5]    (474.85,59.83) -- (512.84,89.85) ;
\draw [color={rgb, 255:red, 65; green, 117; blue, 5 }  ,draw opacity=1 ][line width=1.5]    (387.03,127.38) -- (397.17,171.22) ;
\draw [color={rgb, 255:red, 65; green, 117; blue, 5 }  ,draw opacity=1 ][line width=1.5]    (387.03,127.38) -- (384.98,65.24) ;
\draw [color={rgb, 255:red, 65; green, 117; blue, 5 }  ,draw opacity=1 ][line width=1.5]    (465.83,170.74) -- (517.39,149.58) ;
\draw    (361.06,46.14) -- (387.39,12.05) ;
\draw [color={rgb, 255:red, 74; green, 144; blue, 226 }  ,draw opacity=1 ][line width=1.5]    (457.37,99.75) -- (423.65,59.38) ;
\draw [color={rgb, 255:red, 65; green, 117; blue, 5 }  ,draw opacity=1 ][line width=1.5]    (465.83,170.74) -- (397.17,171.22) ;
\draw [color={rgb, 255:red, 65; green, 117; blue, 5 }  ,draw opacity=1 ][line width=1.5]    (423.65,58.38) -- (474.85,57.83) ;
\draw [color={rgb, 255:red, 65; green, 117; blue, 5 } ,draw opacity=1 ][line width=1.5]    (474.85,57.83) -- (494.52,73.14) -- (514.95,89.05) ;
\draw (473.17,70.71) node [anchor=north west][inner sep=0.75pt]  [color={rgb, 255:red, 74; green, 144; blue, 226 }  ,opacity=1 ]  {$\Lambda $};

\draw (144.17,70.71) node [anchor=north west][inner sep=0.75pt]  [color={rgb, 255:red, 74; green, 144; blue, 226 }  ,opacity=1 ]  {$\Lambda $};

\draw (387.17,70.71) node [anchor=north west][inner sep=0.75pt]  [color={rgb, 255:red, 65; green, 117; blue, 5 } ,opacity=1 ]  {$\Delta $};

\draw (58.17,70.71) node [anchor=north west][inner sep=0.75pt]  [color={rgb, 255:red, 65; green, 117; blue, 5 } ,opacity=1 ]  {$\Delta $};
\end{tikzpicture}
    \caption{On the left, there is an example where two regions, $\Lambda$ (in blue) and $\Delta$ (in green), have non-intersecting boundaries, with $\Lambda$ being completely surrounded by $\Delta$. On the right, there is an example where the boundaries intersect, and $\Lambda$ (in blue) is surrounded by $\Delta$ (in green).}
    \label{examplesurrounded}
\end{figure}
\begin{definition}[Boundary algebra]
    Given $\Lambda \Subset \Delta$, we define
\begin{align*}
    \mathfrak{B}\left(\Lambda \Subset \Delta\right):=\{x p_{\Delta} \mid & x \in p_{\Lambda} \mathfrak{A}(\Lambda) p_{\Lambda} \text { and } x p_{\Delta'}=p_{\Delta'} x \\
    & \text { whenever } \Lambda \Subset \Delta' \text { with } \partial \Lambda \cap \partial \Delta'=\partial \Lambda \cap \partial \Delta \}
\end{align*}
Observe that $\mathfrak{B}\left(\Lambda \Subset \Delta\right)$ is a unital $*$-algebra with identity $p_{\Delta}$.
\end{definition}  

\begin{definition}[Local topological order] \label{LTO}
We say $(\mathfrak{A}, p)$ is locally topologically ordered if it satisfies the following four axioms for sufficiently large $\Lambda$:
\begin{itemize}
    \item[(LTO1)]  Whenever $\Lambda \ll \Delta, p_{\Delta} \mathfrak{A}(\Lambda) p_{\Delta}=\mathbb{C} p_{\Delta}$.
    \item[(LTO2)] Whenever $\Lambda \Subset \Delta, p_{\Delta} \mathfrak{A}(\Lambda) p_{\Delta}=\mathfrak{B}\left(\Lambda \Subset \Delta\right) p_{\Delta}$.
    \item[(LTO3)] Whenever $\Lambda_1 \subset \Lambda_2 \Subset \Delta$ with $\partial \Lambda_1 \cap \partial \Delta=\partial \Lambda_2 \cap \partial \Delta, $ then $\mathfrak{B}\left(\Lambda_1 \Subset \Delta\right)=\mathfrak{B}\left(\Lambda_2 \Subset \Delta\right)$.
    \item[(LTO4)]  Whenever $\Lambda \Subset \Delta_1 \subset \Delta_2$ with $\partial 
    {\Lambda} \cap \partial 
    {\Delta_1}={\partial} \Lambda \cap {\partial \Delta_2}$, if $x \in \mathfrak{B}\left(\Lambda \Subset \Delta_1\right)$ satisfaies $x p_{\Delta_2}=0$, then $x=0$.
\end{itemize}
\end{definition}

\section{Local topological order for smooth boundaries} \label{SECCIONLTO}

In this section, we will prove that the Twisted Quantum Double model for finite groups satisfies the Local Topological Order (LTO) axioms for regions with a smooth boundary. This approach simplifies the proofs, since the rough boundary case requires additional considerations. Nevertheless, in Section \ref{LTOROUGH}, we will address the rough boundary case and explain the extra conditions that must be imposed. These additional requirements arise from the behavior of the twisted model near the boundary.

Therefore, throughout Section \ref{SECCIONLTO}, the regions $\Lambda$ we will consider are those with smooth boundaries, that is, $\partial \Lambda \subset \Lambda$.
\begin{theorem} \label{LTOTEO}
    The Twisted Quantum Double model for a finite group $G$, and for regions with smooth boundaries, satisfies the LTO axioms of Definition \ref{LTO}, where for $\Lambda \subset \cL$ we have
    \begin{align*}
        \mathfrak{A}(\Lambda):=\bigotimes_{\ell \in \Lambda} M_{|G|}(\mathbb{C}) && \text{and} && p_{\Lambda}:=\prod_{ {s(v)} \subset  \Lambda} A_v \prod_{f \subset \Lambda} B_f.
    \end{align*}
\end{theorem}
\begin{remark}
    Note that although the action of \( A^g_v \) only affects the edges in the star \( s(v) \), \( A^g_v \) also introduces a phase determined by the function \( \alpha_{v}^g \). To compute \( \alpha_{v}^g \), we need not only the edges in \( s(v) \), but also all the edges belonging to the closed star \( \overline{s(v)} \). Thus, the vertex operator \( A^g_v \) can be defined in a region only if \( \overline{s(v)} \) is completely contained within that region. However, since we are dealing with regions that have a smooth boundary, if \( s(v) \subset \Lambda \), then \( \overline{s(v)} \subset \Lambda \).
\end{remark}

To prove Theorem \ref{LTOTEO}, we will rely on a series of observations and preliminary results. Initially, let us consider a finite group $G$.  Consider the following operators over $\C[G]$: $L_g$ the left translation by $g$, and $P_h$ the projection onto $|h\rangle$.
\begin{align} \label{leftyproyection}
    L_g|k\rangle:=|g k\rangle, && P_h|k\rangle:=\delta_{h,k}|k\rangle.
\end{align}
The set \(\left\{L_g P_h \mid g, h \in G\right\}\) forms a basis for the matrix algebra of linear operators on the group algebra. This can be verified by noting that this set is orthogonal with respect to the Hilbert–Schmidt inner product, which, for any two operators \(A\) and \(B\), is defined as 
\begin{align}
\langle A,B \rangle_{\text{HS}} := \operatorname{Tr}(A^{\dagger}B) = \sum_{g \in G} \langle g|A^\dagger B|g \rangle.
\end{align}
Instead of left translation operators, we can also use the right translation operators $R_g$ given by $R_g|k\rangle=|k g\rangle$. These operators satisfy the relations $
L_g P_h=P_{g h} L_g$, and $R_g P_h=P_{h g} R_g$.

When we are in a region $\Lambda$, we indicate the location of operators by either drawing pictures with sites labeled by these operators, or we write
$P_l^g$, $L_l^g$, $R_l^g$ for the operators $P_g, L_g, R_g$ respectively
acting on the edge $l$.

\begin{definition}
 Let $\Lambda \subset \cL$ be a region,  by abuse of notation we identify $\Lambda$ with the set of edges contained in $\Lambda$. Given a coloring $c :\Lambda \to G$, we define the vector $\ket{c}=\bigotimes_{l \in \Lambda}\ket{c(l)}$ and the operators \begin{align*}
     L_c:= \bigotimes_{l \in \Lambda} L^{c(l)}_l && \text{and} && P_c:= \bigotimes_{l \in \Lambda} P^{c(l)}_l
 \end{align*}
Note that any $x \in \mathfrak{A}(\Lambda)$ can be written as a linear combination of operators of the form $L_{c_1} P_{c_2}$.  Furthermore, $\ket{c}$ is the only vector for which $P_{c}\ket{c} = \ket{c}$. We say $c:\Lambda \to G $ is flat if $ B_f\ket{c} =\ket{c}$ for all
faces $f \subset \Lambda$, we call $\ket{c}$ a flat vector of $\Lambda$, and $P_c$ a flat operator of $\Lambda$.
\end{definition}

\begin{definition}
    Let $\Lambda \subset \Delta$ be regions. For a coloring $c: \Lambda \rightarrow G$, , we define the extension of $L_{c}$ to the region $\Delta$ as $$L_{c}^\Delta :=  \left( \bigotimes_{l \in \Lambda} L^{c(l)}_l \right) \otimes \left( \bigotimes_{l \in \Delta \setminus \Lambda} \operatorname{Id}_l \right)$$
\end{definition}

\begin{proposition} \label{proigualdad}
    Let $\Lambda \subset \Delta$, $c_1, c_2, c_3: \Lambda \rightarrow G$, and $\cC$ a cycle in $\Delta$ with edges $\{l_1, \dots, l_m\}$. If the operators $L_{c_1}^\Delta$ and $L_{c_2}^\Delta$ coincide on $(m-1)$ edges of $\cC$, and both $p_{\Delta} L_{c_1} P_{c_3} p_{\Delta} \neq 0$ and $p_{\Delta} L_{c_2} P_{c_3} p_{\Delta} \neq 0$, then $L_{c_1}^\Delta$ and $L_{c_2}^\Delta$ coincide on all of $\cC$.
\end{proposition}
\begin{proof}
    Suppose that $L_{c_1}^\Delta$ and $L_{c_2}^\Delta$ coincide on the edges $\{l_1, \dots, l_{m-1}\}$. If $l_m \not\in \Lambda$, then the proposition holds trivially.  Now, let us explore the case when $l_m \in \Lambda$.
    
    For the basis element $\ket{g}=\bigotimes_{l \in \Delta}\ket{g_l}$, we can express $\ket{g} = \ket{g}_\Lambda \ket{g}_{\bar{\Lambda}}$, where $\bar{\Lambda}$ represents the complement of $\Lambda$ in $\Delta$. Since $p_{\Delta} L_{c_1} P_{c_3} p_{\Delta} \neq 0$, it implies that $\ket{c_3} = \ket{c_3}_\Lambda$ must be a flat vector. Consequently, there are vectors of the form $\ket{c_3} \ket{g}_{\bar{\Lambda}}$ such that $B_f\ket{c_3} \ket{g}_{\bar{\Lambda}} = \ket{c_3} \ket{g}_{\bar{\Lambda}}$ for every face operator $B_f$ with $f \in \Delta$. This implies that the holonomy of $\ket{c_3} \ket{g}_{\bar{\Lambda}}$ on the cycle $\mathcal{C}$ is the identity. If the elements of $\ket{c_3} \ket{g}_{\bar{\Lambda}}$ on the cycle $\mathcal{C}$ are denoted as $g_1, \dots, g_m$, then $g_1^{\epsilon(1)} \cdots g_m^{\epsilon(m)} = e$, where $\epsilon(i) \in \{ 1,-1 \}$ depends on the orientation of edge $l_i$. Since $p_{\Delta} L_{c_1} P_{c_3} p_{\Delta} \neq 0$, then $L_{c_1}\ket{c_3} \ket{g}_{\bar{\Lambda}} = L_{c_1}^\Delta \ket{c_3} \ket{g}_{\bar{\Lambda}}$ must also have a trivial holonomy, which implies, $$(L_{l_1}^{c_1(l_1)}g_1)^{\epsilon(1)} \cdots (L_{l_m}^{c_1(l_m)}g_m)^{\epsilon(m)}=e.$$ Now, if an edge $l_i \in \mathcal{C}$ is not in $\Lambda$, we have $L_{l_i}^{c_1(l_i)} = \operatorname{Id}$, leading to $$(L_{l_1}^{c_1(l_1)}g_1)^{\epsilon(1)} \cdots  \underbrace { (g_a)^{\epsilon(a)} \cdots (g_{b})^{\epsilon({b})} }_{ \text{edges in }\Delta \setminus \Lambda }  \cdots (L_{l_m}^{c_1(l_m)}g_m)^{\epsilon(m)}  = \operatorname{Id}.$$
    Then,
    \begin{align*}
        (L_{l_{m}}^{c_1(l_{m})}g_{m})^{\epsilon({m})} & = (L_{l_{m-1}}^{c_1(l_{m-1})}g_{m-1})^{-\epsilon({m-1})} \cdots g_b^{-\epsilon(b)} \cdots g_{a}^{-\epsilon({a})}  \cdots (L_{l_1}^{c_1(l_1)}g_1)^{-\epsilon(1)}_.
    \end{align*}    
    Since $g_1^{\epsilon(1)} \cdots g_m^{\epsilon(m)} = e$, it follows that $g_{b+1}^{\epsilon(b+1)} \cdots g_{m}^{\epsilon(m)}g_{1}^{\epsilon(1)}\cdots g_{a-1}^{\epsilon(a-1)} = g_b^{-\epsilon(b)} \cdots g_{a}^{-\epsilon({a})}$. Therefore,
    \begin{align*}
        (L_{l_{m}}^{c_1(l_{m})}g_{m})^{\epsilon({m})} & = (L_{l_{m-1}}^{c_1(l_{m-1})}g_{m-1})^{-\epsilon({m-1})} \cdots g_{b+1}^{\epsilon(b+1)} \cdots g_{a-1}^{\epsilon(a-1)}  \cdots (L_{l_1}^{c_1(l_1)}g_1)^{-\epsilon(1)}_.
    \end{align*}    
      This demonstrates that $L_{l_m}^{c_1(l_m)}$ only depends on the fixed state $\ket{c_3}$ and on the other $(m-1)$ left operators in $L_{c_1} ^{\Delta}$. A similar argument can be applied to $L_{l_m}^{c_2(l_m)}$. Since $L_{l_i}^{c_1(l_i)} = L_{l_i}^{ c_2(l_i)}$ for $1 \leq i < m$, we obtain that
    \begin{align*}
         (L_{l_{m}}^{c_1(l_{m})}g_{m})^{\epsilon({m})} & = (L_{l_{m-1}}^{c_1(l_{m-1})}g_{m-1})^{-\epsilon({m-1})} \cdots g_{b+1}^{\epsilon(b+1)} \cdots g_{a-1}^{\epsilon(a-1)}  \cdots (L_{l_1}^{c_1(l_1)}g_1)^{-\epsilon(1)}\\
         & = (L_{l_{m-1}}^{c_2(l_{m-1})}g_{m-1})^{-\epsilon({m-1})} \cdots g_{b+1}^{\epsilon(b+1)} \cdots g_{a-1}^{\epsilon(a-1)}  \cdots (L_{l_1}^{c_2(l_1)}g_1)^{-\epsilon(1)}\\
         & = (L_{l_{m}}^{c_2(l_{m})}g_{m})^{\epsilon({m})}
    \end{align*}
    Thus, we conclude that $L_{l_{m}}^{c_1(l_{m})} =L_{l_{m}}^{c_2(l_{m})}.$
\end{proof} 

\begin{theorem} \label{teoLTO1}
    The Twisted Quantum Double model for a finite group $G$ satisfies the LTO1, that is,  if $\Lambda \ll \Delta$, then $p_{\Delta} \mathfrak{A}(\Lambda)p_{\Delta} = \C p_{\Delta}$.    
\end{theorem}
\begin{remark}
    The proofs we present to show that the Twisted model satisfies the LTO axioms follow the spirit of those in \cite{chuah2024boundary}, where it was established that Kitaev's model fulfills these axioms. However, our approach involves two additional considerations that lead to new arguments: first, the presence of the $3$-cocycle $\alpha$, and second, the generality of the lattice structure, which is no longer necessarily square.
\end{remark}
\begin{proof}
Let $\Lambda \ll \Delta$ and $x \in \mathfrak{A}(\Lambda)$. Suppose $x=L_{c_1} P_{c_2}$ for some $c_1, c_2: \Lambda \rightarrow G$.\\

Step 1: Let us see that if $p_{\Delta} x p_{\Delta} \neq 0$, then $p_{\Delta} x p_{\Delta}=p_{\Delta} a_{c_1} P_{c_2} p_{\Delta}$ for some $a_{c_1} \in \C$. To establish this, we aim to identify star operators \( A_{v_i}^{g_i} \), where \( s(v_i) \subset \Lambda \), such that \(
A_{v_n}^{g_n} \cdots A_{v_1}^{g_1} L_{c_1} P_{c_2} = a_{c_1}  P_{c_2}\). First, for every edge $l \in \Lambda$ that does not have any of its stars contained in $\Lambda$ (particularly the edges of the smooth boundary $\partial \Lambda$), we have that $L_{l}^{c_1(l)} = \operatorname{Id}$. This is because there exists a cycle in $\Delta$ where $l$ is the only edge in $\Lambda$. By Proposition \ref{proigualdad} taking the operator $L_{c_1}$ and the identity, we conclude that $L_{l}^{c_1(l)} =  \operatorname{Id}$. To choose the $A_{v_i}^{g_i}$'s, we will recursively define some sets; $\Lambda_j, E_j$ are a set of edges and $F_j$ is a set of stars.
    \begin{align*}
    \Lambda_1 & =  \{ l \in \Lambda\mid \text{there is a star } s \subset \Lambda  \text{ and } l \in s \}, \\
    E_1 & = \{ l \in \Lambda_1\mid \text{there is a star } s \not\subset \Lambda_1  \text{ and } l \in s\},\\
     F_1 & = \{ s \subset \Lambda_1\mid \text{there is an edge } l \in s \text{ such that } l \in E_1 \}.
     \end{align*}
Consider the recursive formula. For each $j$, we define
\begin{align*}
    \Lambda_j  & = \{ l \in \Lambda_{j-1} \setminus E_{j-1}\mid \text{there is a star } s \subset \Lambda_{j-1} \setminus E_{j-1}  \text{ and } l \in s \},\\
    E_j & = \{ l \in \Lambda_j\mid \text{there is a star } s \not\subset \Lambda_j  \text{ and } l \in s \},\\   
  F_j & = \{ s\subset \Lambda_j\mid \text{there is an edge } l \in s \text{ such that } l \in E_j \}.
\end{align*}
Note that by construction, if $\Lambda_j \neq \emptyset$ then there is a edge in $\Lambda_j$ that has at least one of its stars contained in $\Lambda_j$, and so $F_j \neq \emptyset $. The sets $F_j$'s are disjoint sets of stars, whose union results in all the stars contained in $\Lambda$. We begin by listing the stars in $F_1 
= \{s_1, \dots,s_{m_1} \}$. Let $v_1$ the vertex such that $s(v_1) =s_1$, and we choose an edge $l_1 \in s_1 \cap E_1$,   since $(L_{c_1}P_{c_2})|_{l_1} =  L_{l_1}^{h_1}P_{l_1}^{k_1}$ for some $h_1,k_1 \in G$, then due the relations
\begin{align*}
    R_g L_h P_k=L_{h k g k^{-1}} P_k && \text {and} && L_g L_h P_k=L_{g h} P_k \text {, }
\end{align*}
we define $A_{v_1}^{g_1}$ as follows, if  $l_1$ points away from $v_1$, we take $g_1 = h_1^{-1}$; if $l_1$ points to 
$v_1$, we take $g_1= k_1^{-1}h_1k_1$. In both cases,  $A_{v_1}^{g_1}L_{c_1}P_{c_2} = 
 a_1 L_{c'_1}P_{c_2}$ where $(L_{c'_1}P_{c_2})|_{l_1} =  P_{l_1}^{k_1}$, and $a_1 = \alpha_{v_1}^{g_1}(L_{c_1}\ket{c_2})$, (the scalar resulting from applying $A_{v_1}^{g_1}$ to $L_{c_1}\ket{c_2}$). We now proceed with the star $s_2 \in F_1$.  Let $v_2$ the vertex such that $s(v_2) =s_2$, we also chose an edge  $l_2 \in s_2 \cap E_1$. In the same way as we defined $A_{v_1}^{g_1}$, we can define an operator $A_{v_2}^{g_2}$ such that $A_{v_2}^{g_2}a_1L_{c'_1}P_{c_2} = 
 a_2 L_{c''_1}P_{c_2}$, where $(L_{c''_1}P_{c_2})|_{l_2} =  P_{l_2}^{k_2}$. Thus, $$A_{v_2}^{g_2}A_{v_1}^{g_1}L_{c_1}P_{c_2} = a_2 L_{c''_1}P_{c_2},$$ where $a_2= a_1 \alpha_{v_2}^{g_2}(L_{c'_1}\ket{c_2})$. Note that $a_2$ depends on both $A_{v_1}^{g_1}$ and $A_{v_2}^{g_2}$. We continue in this manner until we have covered all the stars in $F_1$. At this point, we can assert that we have found the $A_{v_i}^{g_i}$'s that match $L_{c_1}P_{c_2}$ with $P_{ c_2}$ on the outermost edges of $\Lambda$. Next, we repeat the same process with $F_2$. That is, for each $s_i \in F_2$, we choose an edge $l_i \in s_i \cap E_2$ and define the operator $A_{v_i}^{g_i}$ accordingly. Once we finish with all the stars in $F_2$, we proceed systematically with the remaining $F_j$ for $j > 2$. We continue this procedure until we encounter $F_j = \emptyset$, where we stop.
 In conclusion, we obtain an operator $A_{v_i}^{g_i}$ for each star contained in $\Lambda$. Hence, if $\Lambda$ contains $n$ stars, we will have $$A_{v_1}^{g_1} \cdots A_{v_n}^{g_n} L_{c_1} P_{c_2}= a_{c_1}L_{c_3}P_{c_2},$$ where $a_{c_1}  \in \C$ depends on the operators  $A_{v_i}^{g_i}$. 
 The operator $L_{c_3}P_{c_2}$ is such that for every edge $l_i$ that was chosen for the star $s_i$, we have $(L_{c_3}P_{c_2})|_{l_i } = P_{l_i}^{k_i}$, meaning $L_{l_i}^{c_3(l_i)} = \operatorname{Id}$. Now, considering other edges in $\Lambda$ not in $\{ l_1, \dots,l_n \}$, let $l \not\in \{ l_1, \dots,l_n \}$. Due to how the edges ${l_i}$ were chosen from each of the two vertices incident to $l$, we have a path made by edges of the set $\{ l_1, \dots,l_n \}$ where both paths end at the boundary of $\Lambda$. As $\Lambda \ll \Delta$, we can join these paths using edges in $\Delta \setminus \Lambda$, resulting in a cycle. By Proposition \ref{proigualdad} taking the operator $L_{c_3}$ and the identity, we conclude that $L_{l}^{c_3(l)} = \operatorname{Id}$. Therefore, $L_{c_3} = \operatorname{Id}$, and hence $A_{v_1}^{g_1} \cdots A_{v_n}^{g_n} L_{c_1} P_{c_2}= a_{c_1}P_{ c_2}$. Since $p_{\Delta}$ absorbs every $A_v^{g}$, then
\begin{align*}
p_{\Delta} L_{c_1} P_{c_2} p_{\Delta} =p_{\Delta} A_{v_1}^{g_1} \cdots A_{v_n}^{g_n} L_{c_1} P_{c_2} p_{\Delta} 
=p_{\Delta} a_{c_1} L_{c_3}P_{c_2} p_{\Delta}.
\end{align*} 
Furthermore, since $p_{\Delta} L_{c_1}P_{c_2} p_{\Delta} \neq 0$, then $p_{\Delta}  P_{c_2} p_{\Delta} \neq 0$. Thus, by definition, $P_{c_2}$ is flat.\\

Step 2: Let us prove that whenever $d_1,d_2: \Lambda \to G$ are flat, $p_{\Delta} P_{d_1} p_{\Delta} = p_{\Delta}  P_{d_2} p_{\Delta}$. We begin by selecting a spanning tree, $T$, of $\Lambda$. Since $T$ is a tree, there are vertices incident to only one edge in $T$.  Let us denote one of these vertices as $v_0$, and the corresponding edge in $T$ as $l_1$. Denote $v_1$ as the other vertex that is an endpoint of $l_1$, (note that $s(v_1) \subset \Delta$, but not necessarily $s(v_1) \subset \Lambda$). We consider $d_1(l_1)$ and $d_2(l_1)$ and use the relations
\begin{align*}
    R_gP_hR_{g^{-1}} = P_{hg} && \text{and} && L_gP_hL_{g^{-1}} = P_{gh},
\end{align*}
to define $A_{v_1}^{g_1}$ as follows, if $l_1$ points away from $v_1$, then we take $g_1 = d_2(l_1)d_1(l_1)^{-1}$, if $l_1$ points to $v_1$, then we take $g_1= d_2(l_1)^{-1}d_1(l_1)$. In either case, $A_{v_1}^{g_1}P_{d_1}(A_{v _1}^{g_1})^{-1} = 
 P_{d'_1}$ with $d'_1(l_1) = d_2(l_1)$. Moving to the next edge $l_2$ in $T$, adjacent to $l_1$, take $v_2$ the vertex that is an endpoint of $l_2$, but is not an endpoint of $l_1$. Similarly to $A_{v_1}^{g_1}$, we define an operator $A_{v_2}^{g_2}$ such that $A_{v_2}^{g_2}P_{d'_1}(A_{v_2}^{g_2})^{-1} = 
 P_{d''_1}$ where $d''_1(l_1) = d_2(l_1)$ y $d''_1(l_2) = d_2(l_2)$. Thus $$A_{v_2}^{g_2}A_{v_1}^{g_1}P_{d_1}(A_{v_1}^{g_1})^{-1}(A_{v_2}^{g_2})^{-1} = P_{d''_1}.$$
Continuing this process for all edges in $T$ that are adjacent to $l_1$. Then, we do the same for the edges in $T$ adjacent to $l_2$, excluding, of course, $l_1$. We continue this way until all the edges in $T$ have been considered. Thus, if $T$ contains $m$ edges, we find operators $A_{v_i}^{g_i}$, such that
 $$A_{v_m}^{g_m} \cdots A_{v_1}^{g_1}P_{d_1}(A_{v_1}^{g_1})^{-1} \cdots(A_{v_m}^{g_m})^{-1} = P_{d_3}.$$ The operator $P_{d_3}$ satisfies $d_3(l_i) = d_2(l_i)$ for all $l_i \in T $. Now, for edges $l \in \Lambda \setminus T$, by adding $l$ to $T$, we form a cycle in $\Lambda$. We know that for all edges of this cycle except $l$, $P_{d_3}$ is equal to $P_{d_2}$. Furthermore, since both are flat, it means that $P_{d_2}^{d_2(l)}$ and $P_{d_3}^{d_3(l)}$ are totally determined by the other edges of the cycle. Thus, both must coincide, and hence $P_{d_3} = P_{d_2}$. Since $p_{\Delta}$ absorbs every $A_v^{g}$, then
\begin{align*}
p_{\Delta} P_{d_1} p_{\Delta} =p_{\Delta} A_{v_m}^{g_m} \cdots A_{v_1}^{g_1}P_{d_1}(A_{v_1}^{g_1})^{-1} \cdots(A_{v_m}^{g_m})^{-1}p_{\Delta} =p_{\Delta} P_{d_2} p_{\Delta}.
\end{align*}

Step 3: Finally, we prove that $p_{\Delta}x
p_{\Delta}  = ap_{\Delta}$ for some $a \in \C$. Considering all flat functions $d:\Lambda \to G$, we have $$\sum_{d \text{ flat}} P_d = \prod_{f \subset \Lambda} B_f,$$ so $p_{\Delta}( \sum_{d \text{ flat}} P_d )p_{\Delta}=p_{\Delta}$. Thus, by steps 1 and 2, we can conclude that $$p_{\Delta}x
p_{\Delta} = a_{c_1} p_{\Delta}  P_{c_2} p_{\Delta} = \frac{a_{c_1}}{\text{number of flat functions on } \Lambda }p_{\Delta}.$$
\end{proof}

\begin{corollary}
    The Twisted Quantum Double model is a quantum error correcting code (QECC).
\end{corollary}
\begin{proof}
    The LTO1 condition is equivalent to the necessary and sufficient conditions given in \cite{knill-laflamme} for a code to be a QECC.
\end{proof}

 To prove LTO2, we will use some additional algebras. For a smooth interval  $I$ of $\partial \Lambda$, we will define the algebra $\mathfrak{D}(I)$. This algebra will incorporate certain operators that can be considered as partial vertex and face operators.
 
 \begin{definition}[Operators $C_{v}^g$] \label{defstarparcial}
     Let $\Lambda$ be a region and $g \in G$. Consider a vertex $v$ such that $ \overline{s(v)} \not\subset \Lambda$, yet $s(v) \cap \Lambda \neq \emptyset$. The operator $C_{v}^g$ acts on the edges in $s(v) \cap \Lambda$. As shown on the right side of Figure \ref{algebraD}, we illustrate an example of edges affected by the action of $C_{v}^g$. The action of $C_{v}^g$ on an edge of $s(v) \cap \Lambda$ is a left multiplication by $g$ if the edge is pointed away from $v$, and a right multiplication by $g^{-1}$ if it points to $v$.
     
     Similarly to $A_v^g$, the operator $C_{v}^g$ introduces a phase. This phase is given by the restriction of the function $\alpha_v^g$ \ref{starfase}, to the region $\Lambda$. To properly define this restriction, we must first recall that, for an arbitrary lattice $\mathcal{L}$, the function $\alpha_v^g$ was originally defined on its canonical triangular lattice $\mathcal{L}'$, and each face $f' \in F(\mathcal{L}')$ containing the vertex $v$ contributed to $\alpha_v^g$. When restricting to $\Lambda$, we only consider the faces $f' \subset F(\mathcal{L}')$ for which there exists a face $f \in F(\mathcal{L})$, such that $f' \subset f \subset \Lambda$. Thus, for a given coloring \(\varphi\) of \(\Lambda\), we have that
\[
\alpha_v^g|_\Lambda(\varphi) := \prod_{\substack{f' \in \nabla_v \\ f' \subset f \subset \Lambda}} \zeta(f')^{\epsilon(f')}.
\]
Note that the only difference from the function \ref{starfase} lies in the faces over which the product is taken. In Figure \ref{algebraD}, the faces that contribute to the calculation of $\alpha_{v}^g|_\Lambda$ are highlighted.
 \end{definition}
 Let $\Lambda$ be a region, and  $I$ a smooth  interval of $\partial \Lambda$. On the left of Figure \ref{algebraD}, a portion of the region $\Lambda$ is depicted, highlighting the smooth interval $I$ of $\partial \Lambda$. We will now consider two sets of operators that generate the algebra $\mathfrak{D}(I)$. Firstly, for each $g \in G$ and every edge $l \in I$, we take the projection operator $P_l^g$. An example of one such operator can be seen on the right of Figure \ref{algebraD}. Secondly, for each $g \in G$ and vertex $v$ incident to two edges of $I$, we consider the operator $C_{v}^g$. Thus the  algebra $\mathfrak{D}(I)$ is given by:
\begin{align}
    \mathfrak{D}(I) = \C^* \{ P_l^{g_1}, C_{v}^{g_2} \mid g_1,g_2 \in G, l \in I, v \text{ incident to two edges of } I \}
\end{align}
 
 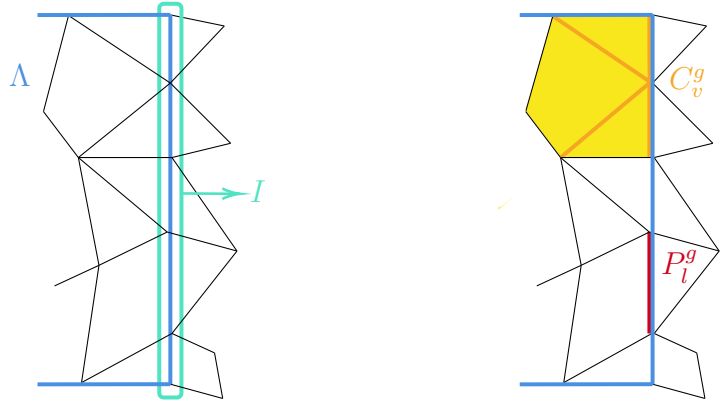
\begin{figure}[h!]
     \centering
\begin{tikzpicture}[x=0.6pt,y=0.6pt,yscale=-1,xscale=1]

\draw  [color={rgb, 255:red, 255; green, 255; blue, 255 }  ,draw opacity=1 ][fill={rgb, 255:red, 248; green, 231; blue, 28 }  ,fill opacity=0.29 ] (311.2,134.52) -- (345.82,11.07) -- (407.91,52.64) -- cycle ;
\draw    (41,11) -- (25.2,71.6) ;
\draw    (47.2,100.6) -- (60.2,168.6) ;
\draw    (49.2,242.6) -- (60.2,168.6) ;
\draw    (25.2,71.6) -- (47.2,100.6) ;
\draw    (60.2,168.6) -- (32.2,181.6) ;
\draw    (105.2,10.6) -- (139.2,17.6) ;
\draw [color={rgb, 255:red, 74; green, 144; blue, 226 }  ,draw opacity=1 ][line width=1.5]    (105.2,10.6) -- (105.2,243.6) ;
\draw    (408.7,211.65) -- (353.2,242.6) ;
\draw    (353.2,242.6) -- (364.2,168.6) ;
\draw    (364.2,168.6) -- (336.2,181.6) ;
\draw    (409.2,53.6) -- (443.2,17.6) ;
\draw    (409.2,53.6) -- (447.2,91.6) ;
\draw    (410.2,100.6) -- (451.2,159.6) ;
\draw    (410.2,211.6) -- (451.2,159.6) ;
\draw    (410.2,211.6) -- (437.05,223.85) ;
\draw    (409.2,243.6) -- (442.2,252.6) ;
\draw    (442.2,252.6) -- (437.05,223.85) ;
\draw    (409.2,10.6) -- (443.2,17.6) ;
\draw    (410.2,100.6) -- (447.2,91.6) ;
\draw [color={rgb, 255:red, 74; green, 144; blue, 226 }  ,draw opacity=1 ][line width=1.5]    (325.45,243.7) -- (409.2,243.6) ;
\draw [color={rgb, 255:red, 208; green, 2; blue, 27 }  ,draw opacity=1 ][line width=1.5]    (407.2,147.6) -- (407.2,211.6) ;
\draw    (407.2,147.6) -- (364.2,168.6) ;
\draw    (451.2,159.6) -- (407.2,147.6) ;
\draw  [color={rgb, 255:red, 80; green, 227; blue, 194 }  ,draw opacity=1 ][line width=1.5]  (98.2,6.4) .. controls (98.2,4.85) and (99.45,3.6) .. (101,3.6) -- (109.4,3.6) .. controls (110.95,3.6) and (112.2,4.85) .. (112.2,6.4) -- (112.2,249.2) .. controls (112.2,250.75) and (110.95,252) .. (109.4,252) -- (101,252) .. controls (99.45,252) and (98.2,250.75) .. (98.2,249.2) -- cycle ;
\draw    (105.2,53.6) -- (139.2,17.6) ;
\draw    (41,11) -- (105.2,53.6) ;
\draw    (105.2,53.6) -- (47.2,100.6) ;
\draw    (105.2,53.6) -- (143.2,91.6) ;
\draw    (106.2,100.6) -- (143.2,91.6) ;
\draw    (47.2,100.6) -- (106.2,100.6) ;
\draw    (103.2,147.6) -- (47.2,100.6) ;
\draw    (103.2,147.6) -- (60.2,168.6) ;
\draw    (147.2,159.6) -- (103.2,147.6) ;
\draw    (106.2,100.6) -- (147.2,159.6) ;
\draw [color={rgb, 255:red, 80; green, 227; blue, 194 }  ,draw opacity=1 ][line width=1]   (112.5,123.25) -- (146.9,123.49) ;
\draw [shift={(148.9,123.5)}, rotate = 180.39] [color={rgb, 255:red, 80; green, 227; blue, 194 }  ,draw opacity=1 ][line width=1]    (10.93,-3.29) .. controls (6.95,-1.4) and (3.31,-0.3) .. (0,0) .. controls (3.31,0.3) and (6.95,1.4) .. (10.93,3.29)   ;
\draw    (106.2,211.6) -- (147.2,159.6) ;
\draw    (106.2,211.6) -- (49.2,242.6) ;
\draw    (106.2,211.6) -- (133.05,223.85) ;
\draw    (138.2,252.6) -- (133.05,223.85) ;
\draw    (105.2,243.6) -- (138.2,252.6) ;
\draw [color={rgb, 255:red, 74; green, 144; blue, 226 }  ,draw opacity=1 ][line width=1.5]    (21.45,243.7) -- (105.2,243.6) ;
\draw [color={rgb, 255:red, 74; green, 144; blue, 226 }  ,draw opacity=1 ][line width=1.5]    (21.45,10.7) -- (105.2,10.6) ;
\draw  [color={rgb, 255:red, 255; green, 255; blue, 255 }  ,draw opacity=1 ][fill={rgb, 255:red, 248; green, 231; blue, 28 }  ,fill opacity=0.29 ] (406.33,52.68) -- (344.13,10.08) -- (406.99,11.08) -- cycle ;
\draw  [color={rgb, 255:red, 255; green, 255; blue, 255 }  ,draw opacity=1 ][fill={rgb, 255:red, 248; green, 231; blue, 28 }  ,fill opacity=0.29 ] (351.18,100.63) -- (407.18,53.63) -- (407.2,100.6) -- cycle ;
\draw [color={rgb, 255:red, 245; green, 166; blue, 35 }  ,draw opacity=1 ][line width=1.5]    (407.2,53.6) -- (351.2,100.6) ;
\draw [color={rgb, 255:red, 245; green, 166; blue, 35 }  ,draw opacity=1 ][line width=1.5]    (407.2,53.6) -- (407.2,100.6) ;
\draw [color={rgb, 255:red, 245; green, 166; blue, 35 }  ,draw opacity=1 ][line width=1.5]    (407.2,10.6) -- (407.2,53.6) ;
\draw    (351.2,100.6) -- (410.2,100.6) ;
\draw [color={rgb, 255:red, 245; green, 166; blue, 35 }  ,draw opacity=1 ][line width=1.5]    (346.3,11.35) -- (407.33,52.68) ;
\draw    (346.3,11.35) -- (329.2,71.6) ;
\draw [color={rgb, 255:red, 74; green, 144; blue, 226 }  ,draw opacity=1 ][line width=1.5]    (325.45,10.7) -- (409.2,10.6) ;
\draw [color={rgb, 255:red, 74; green, 144; blue, 226 }  ,draw opacity=1 ][line width=1.5]    (409.2,10.6) -- (409.2,243.6) ;
\draw  [color={rgb, 255:red, 255; green, 255; blue, 255 }  ,draw opacity=1 ][fill={rgb, 255:red, 255; green, 255; blue, 255 }  ,fill opacity=1 ] (327.84,69.88) -- (352.16,101.93) -- (312.97,131.68) -- (288.64,99.63) -- cycle ;
\draw    (351.2,100.6) -- (364.2,168.6) ;
\draw    (407.18,147.63) -- (351.18,100.63) ;
\draw    (329.2,71.6) -- (351.2,100.6) ;
\draw (2,40) node [anchor=north west][inner sep=0.75pt]  [color={rgb, 255:red, 74; green, 144; blue, 226 }  ,opacity=1 ]  {$\Lambda $};
\draw (153.75,113.6) node [anchor=north west][inner sep=0.75pt]  [color={rgb, 255:red, 80; green, 227; blue, 194}  ,opacity=1 ]  {$I$};
\draw (418,42.4) node [anchor=north west][inner sep=0.75pt]  [color={rgb, 255:red, 245; green, 166; blue, 35 }  ,opacity=1 ]  {$C_{v}^{g}$};
\draw (413,157.4) node [anchor=north west][inner sep=0.75pt]  [color={rgb, 255:red, 208; green, 2; blue, 27 }  ,opacity=1 ]  {$P_{l}^{g}$};
\end{tikzpicture}
     \caption{On the left, we see a smooth interval \(I\) of \(\Lambda\). On the right, the edges on which the operators \(C_{v}^g\) and \(P_{l}^g\) act are indicated.}
     \label{algebraD}
 \end{figure}
 
We now have everything we need to prove that the Twisted Quantum Double model satisfies the LTO2 condition.
\begin{theorem} \label{teoLTO2}
    The Twisted Quantum Double model for a finite group $G$, and for regions with smooth boundaries, satisfies the LTO2, that is,  if $\Lambda \Subset \Delta$, then $p_{\Delta} \mathfrak{A}(\Lambda) p_{\Delta}=\mathfrak{B}\left(\Lambda \Subset \Delta\right) p_{\Delta}$.
\end{theorem}
\begin{proof}
 Let \(I = \partial \Lambda \cap \partial \Delta\). We will show that \(p_{\Delta} \mathfrak{A}(\Lambda) p_{\Delta} = \mathfrak{D}(I) p_{\Delta}\). To do this, we will demonstrate the following inclusions:
\begin{align} \label{contenencia1}
    \mathfrak{D}(I) p_{\Delta} \subseteq \mathfrak{B}\left(\Lambda \Subset \Delta\right) \subseteq p_{\Delta} \mathfrak{A}\left(\Lambda\right) p_{\Delta}.
\end{align}
To see that \(\mathfrak{D}(I) p_{\Delta} \subseteq \mathfrak{B}\left(\Lambda \Subset \Delta\right)\), take \(x \in \mathfrak{D}(I)\) and consider \(xp_{\Lambda} \in \mathfrak{D}(I) p_{\Lambda}\). By the definition of the algebra \(\mathfrak{D}(I)\), every operator in this algebra commutes with the face and star operators. Therefore, \(xp_{\Lambda} = xp_{\Lambda}p_{\Lambda} = p_{\Lambda} xp_{\Lambda} \in p_{\Lambda} \mathfrak{A}(\Lambda) p_{\Lambda}\). Moreover, for any region \(\Delta'\) such that \(\Lambda \Subset \Delta'\) with \(\partial \Lambda \cap \partial \Delta' = \partial \Lambda \cap \partial \Delta\), it holds that \(xp_{\Lambda}p_{\Delta'} = p_{\Delta'} xp_{\Lambda}\). Hence, \(xp_{\Lambda}p_{\Delta} = xp_{\Delta} \in \mathfrak{B}\left(\Lambda \Subset \Delta\right)\). Now, let's show that \(\mathfrak{B}\left(\Lambda \Subset \Delta\right) \subseteq p_{\Delta} \mathfrak{A}\left(\Lambda\right) p_{\Delta}\). Let \(xp_{\Delta} \in \mathfrak{B}\left(\Lambda \Subset \Delta\right)\). By definition, \(x \in p_{\Lambda} \mathfrak{A}(\Lambda) p_{\Lambda}\) and \(xp_{\Delta} = p_{\Delta} x\). Then, \(xp_{\Delta} = p_{\Delta} xp_{\Delta} = p_{\Delta} p_{\Lambda} xp_{\Lambda} p_{\Delta}\), where clearly \(p_{\Lambda} xp_{\Lambda} \in \mathfrak{A}(\Lambda)\). Thus, \(xp_{\Delta} \in p_{\Delta} \mathfrak{A}(\Lambda) p_{\Delta}\).

The next two steps aim to show that \(p_{\Delta} \mathfrak{A}(\Lambda) p_{\Delta} \subseteq \mathfrak{D}(I) p_{\Delta}\). These steps are analogous to the first two steps in the proof of Theorem \ref{teoLTO1}. Let \(x \in \mathfrak{A}(\Lambda)\), and suppose \(x = L_{c_1} P_{c_2}\) for some \(c_1, c_2: \Lambda \rightarrow G\).\\

Step 1: Let us see that if $p_{\Delta} x p_{\Delta} \neq 0$, then
$$
p_{\Delta} x p_{\Delta}=b_{c_1}\prod_{i=1}^n C_{u_i}^{g_i} p_{\Delta} P_{c_2} p_{\Delta}
$$
for some $b_{c_1} \in \C$, and operators $C_{u_i}^{g_i}$. To establish this, we aim to identify star operators \( A_{v_i}^{g_i} \), where \( s(v_i) \subset \Lambda \), such that \(
A_{v_n}^{g_n} \cdots A_{v_1}^{g_1} L_{c_1} P_{c_2} = a_{c_1} L_{c_3} P_{c_2}\), with \( L_{l}^{c_3(l)} = \Id \) for all edges in \( \Lambda \) that are not incident to any vertex in \( I \). First, for each edge \( l \in \partial \Lambda \setminus I \), we have \(L_{l}^{c_1(l)} = \operatorname{Id}\). Next, in order to select the appropriate \( A_{v_i}^{g_i} \)'s, we define the following sets:
    \begin{align*}
    \Lambda'_1 & =  \{ l \in \Lambda\mid \text{there is a star } s \subset \Lambda  \text{ and } l \in s \}, \\
    E'_1 & = \{ l \in \Lambda_1\mid \text{there is a star } s \not\subset \Lambda_1, \overline{s} \subset \Delta, \text{ and } l \in s\},\\
     F'_1 & = \{ s \subset \Lambda_1\mid \text{there is an edge } l \in s \text{ such that } l \in E'_1 \}.
     \end{align*}
Consider the recursive formula. For each $j$, we define
\begin{align*}
    \Lambda'_j  & = \{ l \in \Lambda_{j-1} \setminus E_{j-1}\mid \text{there is a star } s \subset \Lambda_{j-1} \setminus E'_{j-1}  \text{ and } l \in s \},\\
    E'_j & = \{ l \in \Lambda_j\mid \text{there is a star } s \not\subset \Lambda_j, \overline{s} \subset \Delta,  \text{ and } l \in s \},\\   
  F'_j & = \{ s\subset \Lambda_j\mid \text{there is an edge } l \in s \text{ such that } l \in E'_j \}.
\end{align*}
Note that the only difference between these sets and those defined in the proof of Theorem \ref{teoLTO1} lies in the definition of the sets $E'_j$. As before, the edges in $E'_j$ satisfy the condition that one of their stars is not contained in $\Lambda'_j$. However, an additional condition is now imposed: the star-closed of the corresponding star must be contained in $\Delta$. Previously, this requirement was unnecessary since $\Lambda$ was completely surrounded by $\Delta$. This extra condition prevents $E'_j$ from containing edges incident to any vertex in $I$.

To choose the operators $A_{v_i}^{g_i}$, we proceed as in Theorem \ref{teoLTO1}, but using the sets $E'_j$ and $F'_j$. Specifically, for the set $F'_1 = \{s_1, \dots, s_{m_1} \}$, we select vertices $v_i$ such that $s(v_i) = s_i$, and define the operators $A_{v_i}^{g_i}$, where each one depends on the choice of an edge $l_i \in s_i \cap E'_1$. We continue this process sequentially for each $F'_j$ until we reach $F_j = \emptyset$, at which point we stop. In this way, we obtain operators $A_{v_i}^{g_i}$ such that
$$A_{v_1}^{g_1} \cdots A_{v_m}^{g_m} L_{c_1} P_{c_2} = a_{c_1} L_{c_3} P_{c_2},$$
where $a_{c_1} \in \mathbb{C}$. As was shown in Theorem \ref{teoLTO1}, it can be demonstrated that $L_{l}^{c_3(l)} = \operatorname{Id}$ for any edge $l$ that is not incident to a vertex in $I$. Thus, $L_{c_3}$ is supported on $I$ and on the edges incident to it.

Now, take the edges \(\{ l_1, \dots, l_{n+1} \}\) of \(I\), ordered in such a way that if we travel around \(\partial \Lambda\) in the counterclockwise, we meet these edges in that order. Let \(u_i\) be the vertex connecting the edges \(l_i\) and \(l_{i+1}\). Since \((L_{c_3}P_{c_2})|_{l_1} = L_{l_1}^{h_1}P_{l_1}^{k_1}\) for some \(h_1, k_1 \in G\), we define \(C_{u_1}^{g_1}\) as follows: if \(l_1\) points away from \(u_1\), we take \(g_1 = h_1\); if \(l_1\) points towards \(u_1\), we take \(g_1 = k_1^{-1}h_1^{-1}k_1\). In both cases, \(C_{u_1}^{g_1}(P_{c_2}) = b_1 L_{c'_3}P_{c_2}\), where \((L_{c'_3}P_{c_2})|_{l_1} = (L_{c_3}P_{c_2})|_{l_1}\), and \(b_1 = \alpha_{u_1}^{g_1}|_\Lambda(\ket{c_2})\). Next, we consider vertex \(u_2\). As we did for $u_1$, we define an operator \(C_{u_2}^{g_2}\) such that \(C_{u_2}^{g_2}(b_1L_{c'_3}P_{c_2}) = b_2 L_{c''_3}P_{c_2}\), where \((L_{c''_3}P_{c_2})|_{l_2} = (L_{c_3}P_{c_2})|_{l_2}\). Thus,
$$
C_{u_2}^{g_2}C_{u_1}^{g_1}(P_{c_2}) = b_2 L_{c''_3}P_{c_2},
$$
where \(b_2 = b_1 \alpha_{u_2}^{g_2}|_\Lambda(L_{c'_3} \ket{c_2})\). We continue in this manner until we have covered all the vertices \(u_i\) for \(1 \le i \le n\). In this way, we obtain
$$
C_{u_1}^{g_1} \cdots C_{u_n}^{g_n}(P_{c_2}) = b_{c_3} L_{c_4} P_{c_2},
$$
where \(b_{c_3} \in \mathbb{C}\) depends on the operators \(C_{u_i}^{g_i}\). Using Proposition \ref{proigualdad}, it can be shown that \(L_{c_4} = L_{c_3}\), and hence \(C_{u_1}^{g_1} \cdots C_{u_n}^{g_n}(P_{c_2}) = b_{c_3} L_{c_3} P_{c_2}\). Then
\begin{align*}
p_{\Delta} L_{c_1} P_{c_2} p_{\Delta} & =p_{\Delta} A_{v_1}^{g_1} \cdots A_{v_m}^{g_m} L_{c_1} P_{c_2} p_{\Delta} \\
& =p_{\Delta} a_{c_1} L_{c_3}P_{c_2} p_{\Delta}\\
& =  p_{\Delta} b_{c_1}\prod_{i=1}^n C_{u_i}^{g_i}  P_{c_2} p_{\Delta}  \quad \quad \quad \quad \text{ where } b_{c_1} = \frac{a_{c_1}}{b_{c_3}}\\
& = b_{c_1}\prod_{i=1}^n C_{u_i}^{g_i} p_{\Delta}  P_{c_2} p_{\Delta}.
\end{align*} 
Furthermore, since $p_{\Delta} L_{c_1}P_{c_2} p_{\Delta} \neq 0$, then $p_{\Delta} P_{c_2} p_{\Delta} \neq 0$. Thus, by definition, $P_{c_2}$ is flat.\\

Step 2: Let us prove that whenever $d_1, d_2: \Lambda \to G$ are flat, and $d_1(l) = d_2(l)$ for all $l \in I$, then $p_{\Delta} P_{d_1} p_{\Delta} = p_{\Delta} P_{d_2} p_{\Delta}$. To do this, we must take a tree $T$ of $\Lambda$, such that $I \subset T$. Now consider the tree $T' = T \setminus I$. We will use this tree and a procedure similar to the one we used in the proof of Theorem \ref{teoLTO1}. In that proof, to define the operators $A_{v_i}^{g_i}$, we first took an edge from $T'$ that was incident to a single vertex in $T'$. However, in this case, the edges we will use to define the first operators $A_{v_i}^{g_i}$ will be those in $T'$ that are incident to $I$. Once we have defined the vertex operator corresponding to each of these, we will continue as described in Theorem \ref{teoLTO1}. This extra consideration is necessary to ensure that, when defining the operator $A_{v_i}^{g_i}$, we have $\overline{s(v_i)} \subset \Delta$. In this way, we can prove that if $T'$ contains $m$ edges, there exist operators $A_{v_i}^{g_i}$, such that
\[
p_{\Delta} P_{d_1} p_{\Delta} = p_{\Delta} A_{v_m}^{g_m} \cdots A_{v_1}^{g_1} P_{d_1} (A_{v_1}^{g_1})^{-1} \cdots (A_{v_m}^{g_m})^{-1} p_{\Delta} = p_{\Delta} P_{d_2} p_{\Delta}.
\]

Step 3: Finally, we prove that $p_{\Delta} x p_{\Delta} \in \mathfrak{D}(I) p_{\Delta}$. Let $Flat_{c_2} = \{ c: \Lambda \to G \mid c \text{ is flat, and } c(l) = c_2(l) \text{ for all } l \in I \}$. Note that
\[
\sum_{c \in Flat_{c_2}} P_c = \prod_{f \subset F(\Lambda)} B_f \prod_{l \in J} P_{l}^{c_2(l)},
\]
so $p_{\Delta} \left( \sum_{c \in Flat_{c_2}} P_c \right) p_{\Delta} = \left( \prod_{l \in I} P_{l}^{c_2(l)} \right) p_{\Delta}$. Thus, by steps 1 and 2, we can conclude that
\[
p_{\Delta} x p_{\Delta} = b_{c_1} \prod_{i=1}^n C_{u_i}^{g_i} p_{\Delta} P_{c_2} p_{\Delta} = \frac{b_{c_1}}{|Flat_{c_2}|} \prod_{i=1}^n C_{u_i}^{g_i} \prod_{l \in J} P_{l}^{c_2(l)} p_{\Delta}.\]
Thus, $p_{\Delta} x p_{\Delta} \in \mathfrak{D}(I) p_{\Delta}$, that is, $p_{\Delta} \mathfrak{A}(\Lambda) p_{\Delta} \subseteq \mathfrak{D}(I) p_{\Delta}$, which, together with the inclusions \ref{contenencia1}, concludes that $p_{\Delta} \mathfrak{A}(\Lambda) p_{\Delta} = \mathfrak{D}(I) p_{\Delta}$.
\end{proof}

\begin{theorem} \label{teoLTO3}
   The Twisted Quantum Double model for a finite group $G$, and for regions with smooth boundaries, satisfies LTO3. That is, whenever $\Lambda_1 \subset \Lambda_2 \Subset \Delta$ with $\partial \Lambda_1 \cap \partial \Delta = \partial \Lambda_2 \cap \partial \Delta$, we have $\mathfrak{B}\left(\Lambda_1 \Subset \Delta\right) = \mathfrak{B}\left(\Lambda_2 \Subset \Delta\right)$.
\end{theorem}
\begin{proof}
    Let $I = \partial \Lambda_1 \cap \partial \Delta = \partial \Lambda_2 \cap \partial \Delta$. Note that the definition of the algebra $\mathfrak{D}(I)$ does not depend on the region $\Lambda_1$ or $\Lambda_2$. Thus, by the proof of Theorem \ref{teoLTO2}, we have
    \begin{align*}
        \mathfrak{B}\left(\Lambda_1 \Subset \Delta\right) = \mathfrak{D}(I) p_{\Delta} = \mathfrak{B}\left(\Lambda_2 \Subset \Delta\right).
    \end{align*}
\end{proof}

\begin{theorem}\label{teoLTO4}
    The Twisted Quantum Double model for a finite group $G$, and for regions with smooth boundaries, satisfies the LTO4, that is, whenever $\Lambda \Subset \Delta_1 \subset \Delta_2$ with $\partial 
    {\Lambda} \cap \partial 
    {\Delta_1}={\partial} \Lambda \cap {\partial \Delta_2}$, if $x \in \mathfrak{B}\left(\Lambda \Subset \Delta_1\right)$ with $x p_{\Delta_2}=0$, then $x=0$.
\end{theorem}
\begin{proof}
    Let \( c: \Delta_2 \to G \) be a flat coloring, and let \(\{ l_1, \dots, l_{n+1} \}\) denote the set of edges of \( I \), while \(\{ v_1, \dots, v_{n} \}\) denotes the set of internal vertices of \( I \), i.e., those vertices for which we have defined the operators \( C_{v_i} \). Given the colors assigned to the edges of \( I \), there exists a unique \((n+1)\)-tuple \((h_1, \dots, h_{n+1}) \in G^{n+1}\), with \( h_i = c(l_i) \), such that \( \prod_{i=1}^{n+1} P_{l_i}^{h_i} \ket{c} = \ket{c} \).
    
    Let \( x \in \mathfrak{D}(I) \) such that \( x p_{\Delta_2} = 0 \). Then \( 0 = x p_{\Delta_2} \ket{c} = p_{\Delta_2} x \ket{c} \), and we have \( x \ket{c} = (a_1 x_1 + \cdots + a_m x_m) \ket{c} \), where each \( x_k \) is of the form \( x_k = \prod_{j=1}^{n} C_{v_j}^{g_{k,j}} \prod_{i=1}^{n+1} P_{l_i}^{h_i} \) for some \( n \)-tuple \((g_{k,1}, \dots, g_{k,n}) \in G^n \). Note that \( \{ x_k \ket{c} \mid 1 \leq k \leq m \} \) forms a linearly independent set of flat vectors, since any pair of them must differ by at least one edge of \( I \). Thus,
    \begin{align*}
        0 = p_{\Delta_2} x \ket{c} = \sum_{k=1}^m a_k p_{\Delta_2} x_k \ket{c} = \sum_{k=1}^m a_k p_{\Delta_2} \left( \prod_{j=1}^{n} C_{v_j}^{g_{k,j}} \ket{c} \right)_.
    \end{align*}
    Now, considering \( p_{\Delta_2} x_k \ket{c} \), we have
    \begin{align*}
        0 = p_{\Delta_2} x_k \ket{c} = \prod_{\overline{s(v)} \subset \Delta_2} A_v (x_k \ket{c}) = \prod_{\overline{s(v)} \subset \Delta_2} \frac{1}{|G|} \sum_{g \in G} A_v^g (x_k \ket{c}).
    \end{align*}
    This results in a sum of flat vectors. If \( N \) denotes the number of vertices \( v \) such that \( \overline{s(v)} \subset \Delta_2 \), then the vectors in the sum \( p_{\Delta_2} x_k \ket{c} \) have the form \( \frac{1}{|G|^N} \prod_{i=1}^N A_{v_i}^{g_i} (x_k \ket{c}) \) for some \( N \)-tuple \((g_1, \dots, g_N) \in G^N \). Consequently, the vectors in \( p_{\Delta_2} x_k \ket{c} \) form a linearly independent set, since distinct \( N \)-tuples yield vectors differing on at least one edge of \( \Delta_2 \). Moreover, since the operators \( A_{v_i}^{g_i} \) do not affect the edges in \( I \), if \( l \neq k \), then the vectors in \( p_{\Delta_2} x_l \ket{c} \) and \( p_{\Delta_2} x_k \ket{c} \) are also linearly independent.

    Introducing the notation \( (g_N) := (g_1, \dots, g_N) \in G^N \) and \( A^{(g_N)} := \prod_{i=1}^N A_{v_i}^{g_i} \), we then have
    \begin{align*}
        0 = p_{\Delta_2} x \ket{c} = \sum_{k=1}^m a_k p_{\Delta_2} x_k \ket{c} = \sum_{k=1}^m \sum_{(g_N) \in G^N} \frac{a_k}{|G|^N} A^{(g_N)} (x_k \ket{c}).
    \end{align*}
    Since this is a sum of linearly independent vectors, it must hold that \( a_k = 0 \) for all \( 1 \leq k \leq m \). As this holds for any flat vector, we conclude that \( x = 0 \).
\end{proof}



\section{Local topological order for rough boundaries} \label{LTOROUGH}
In this section, we will examine the LTO axioms for the Twisted Quantum Double model in the context of rough intervals. However, in this case, certain conditions must be imposed on the boundary of the region. The first necessary change to satisfy the LTO axioms for rough boundaries is redefining the projectors $p_{\Lambda}$, which are now defined as follows:
\begin{align} \label{plambda2}
    p_{\Lambda} := \prod_{\overline{s(v)} \subset \Lambda} A_v \prod_{f \subset \Lambda} B_f.
\end{align}
The key difference here is that, due to the possible presence of rough edges, defining the vertex operator in the region $\Lambda$ requires that $\overline{s(v)} \subset \Lambda$, which is a more restrictive condition than $s(v) \subset \Lambda$.

\begin{remark}
    The condition $\overline{s(v)} \subset \Lambda$ can be simplified back to $s(v) \subset \Lambda$, even in the case of a rough boundary, in scenarios where $\alpha$ is the trivial \(3\)-cocycle, where the Twisted Quantum Double model reduces to Kitaev's model, or when the lattice is a triangular lattice.
\end{remark}

The definition of completely surrounded does not change in the smooth or rough case, so the proof for LTO1 is entirely analogous. However, the definition of surrounded does change depending on whether the interval is smooth or rough. For a rough boundary $I$ of $\partial \Lambda$, we will define the algebra $\mathfrak{C}(I)$.

For the algebra $\mathfrak{C}(I)$, we will introduce additional operators that can be regarded as partial face operators. To achieve this, it is necessary to express the face operators using the projectors $P_g$. However, because the definition of the face operator depends on the lattice orientation, we first need to consider the following sign function.
\begin{definition} \label{signof}
Let $(f,l) \in F \times E $, 
where  $l \in \partial f$. Define the sign function as follows
\begin{align}
    \operatorname{Sng}(f,l) := \begin{cases}
        \phantom{-}1, & \text{if } l \text{ is oriented in the counterclockwise direction along } \partial f, \\
        -1, & \text{if } l \text{ is oriented in the clockwise direction along } \partial f.
    \end{cases}
\end{align}
\end{definition}
Thus, for a face \(f\) where \(\partial f\) consists of the edges \(l_1, l_2, \ldots, l_n\), the operator \(B_f\) can be written as 
\begin{align} \label{faceoperator2}
    B_f = \sum_{g'_1 \cdots g'_n=0} \left( \bigotimes_{i=1 }^n P_{l_i}^{g_i}\right) \otimes \left( \bigotimes_{l \in E\setminus \partial f}\operatorname{Id}_l \right),
\end{align}
where $g'_i = g_i^{\operatorname{Sng}(f,l_i)}$. In Figure \ref{facedraw2}, we see how \(B_f\) looks on a specific face, using equation \ref{faceoperator2}.
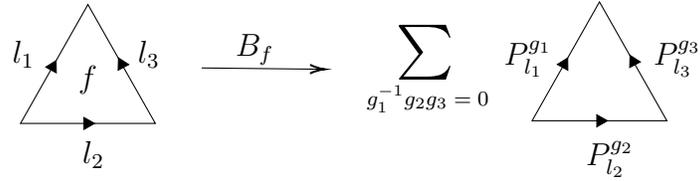
\begin{figure}[h!]
    \centering
\begin{tikzpicture}[x=0.5pt,y=0.5pt,yscale=-1,xscale=1]

\draw   (54,4) -- (105,94) -- (3,94) -- cycle ;
\draw  [fill={rgb, 255:red, 0; green, 0; blue, 0 }  ,fill opacity=1 ] (29.92,46.49) -- (30.03,55.67) -- (22.13,51.35) -- cycle ;
\draw  [fill={rgb, 255:red, 0; green, 0; blue, 0 }  ,fill opacity=1 ] (77.14,44.46) -- (84.85,49.44) -- (76.88,53.64) -- cycle ;
\draw  [fill={rgb, 255:red, 0; green, 0; blue, 0 }  ,fill opacity=1 ] (59,93.99) -- (51.01,98.51) -- (50.99,89.51) -- cycle ;
\draw    (140,52) -- (231,52.98) ;
\draw [shift={(233,53)}, rotate = 180.62] [color={rgb, 255:red, 0; green, 0; blue, 0 }  ][line width=0.75]    (10.93,-3.29) .. controls (6.95,-1.4) and (3.31,-0.3) .. (0,0) .. controls (3.31,0.3) and (6.95,1.4) .. (10.93,3.29)   ;
\draw   (441,2) -- (492,92) -- (390,92) -- cycle ;
\draw  [fill={rgb, 255:red, 0; green, 0; blue, 0 }  ,fill opacity=1 ] (416.92,44.49) -- (417.03,53.67) -- (409.13,49.35) -- cycle ;
\draw  [fill={rgb, 255:red, 0; green, 0; blue, 0 }  ,fill opacity=1 ] (464.14,42.46) -- (471.85,47.44) -- (463.88,51.64) -- cycle ;
\draw  [fill={rgb, 255:red, 0; green, 0; blue, 0 }  ,fill opacity=1 ] (447,91.99) -- (439.01,96.51) -- (438.99,87.51) -- cycle ;

\draw (164,24.4) node [anchor=north west][inner sep=0.75pt]    {$B_f$};
\draw (-5,29.4) node [anchor=north west][inner sep=0.75pt]    {$l_{1}$};
\draw (48,105.4) node [anchor=north west][inner sep=0.75pt]    {$l_{2}$};
\draw (90,29.4) node [anchor=north west][inner sep=0.75pt]    {$l_{3}$};
\draw (45,48.4) node [anchor=north west][inner sep=0.75pt]    {$f$};
\draw (367,29.4) node [anchor=north west][inner sep=0.75pt]    {$P_{l_{1}}^{g_{1}}$};
\draw (429,105.4) node [anchor=north west][inner sep=0.75pt]    {$P_{l_{2}}^{g_{2}}$};
\draw (480,29.4) node [anchor=north west][inner sep=0.75pt]    {$P_{l_{3}}^{g_{3}}$};
\draw (285,17.4) node [anchor=north west][inner sep=0.75pt]  [font=\LARGE]  {$\sum$};

\draw (263,63.4) node [anchor=north west][inner sep=0.75pt]  [font=\tiny]  {$g_{1}^{-1} g_{2} g_{3} =0$};
\end{tikzpicture}
    \caption{The action of the operator \(B_f\), using equation \ref{faceoperator2}.}
    \label{facedraw2}
\end{figure}
\begin{definition}[Operators $D_{f}^g$] \label{deffaceparcial}
    Let \(\Lambda\) be a region, \(g \in G\), and let \(f\) be a face such that \(f \not\subset \Lambda\) and \(\partial f \cap \Lambda = \{ l_1, \ldots, l_n \} \neq \emptyset\). The operator \(D_{f}^g: \cH_{\Lambda} \to \cH_{\Lambda}\) is given by
\begin{align} \label{partialfaceoperator}
    D_{f}^g = \sum_{gg'_1\cdots g'_n=0} \left( \bigotimes_{i=1}^n P_{l_i}^{g_i} \right) \otimes \left( \bigotimes_{l \in \Lambda \setminus \partial f} \operatorname{Id}_l \right),
\end{align}
where \(g'_i = g_i^{\operatorname{Sng}(f, l_i)}\). In Figure \ref{facedraw3}, we illustrate operator \(D_{f}^g\) on a specific face \(f\).
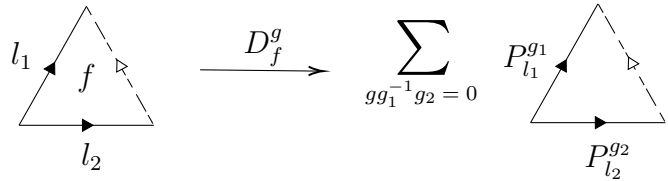
\begin{figure}[h!]
    \centering
\begin{tikzpicture}
[x=0.5pt,y=0.5pt,yscale=-1,xscale=1]

\draw  [fill={rgb, 255:red, 0; green, 0; blue, 0 }  ,fill opacity=1 ] (29.92,46.49) -- (30.03,55.67) -- (22.13,51.35) -- cycle ;
\draw  [fill={rgb, 255:red, 0; green, 0; blue, 0 }  ,fill opacity=1 ] (59,93.99) -- (51.01,98.51) -- (50.99,89.51) -- cycle ;
\draw    (140,52) -- (231,52.98) ;
\draw [shift={(233,53)}, rotate = 180.62] [color={rgb, 255:red, 0; green, 0; blue, 0 }  ][line width=0.75]    (10.93,-3.29) .. controls (6.95,-1.4) and (3.31,-0.3) .. (0,0) .. controls (3.31,0.3) and (6.95,1.4) .. (10.93,3.29)   ;
\draw  [fill={rgb, 255:red, 0; green, 0; blue, 0 }  ,fill opacity=1 ] (416.92,44.49) -- (417.03,53.67) -- (409.13,49.35) -- cycle ;
\draw  [fill={rgb, 255:red, 0; green, 0; blue, 0 }  ,fill opacity=1 ] (447,91.99) -- (439.01,96.51) -- (438.99,87.51) -- cycle ;
\draw    (54,4) -- (3,94) ;
\draw    (105,94) -- (3,94) ;
\draw  [dash pattern={on 3.75pt off 3pt on 7.5pt off 1.5pt}]  (54,4) -- (105,94) ;
\draw  [fill={rgb, 255:red, 255; green, 255; blue, 255 }  ,fill opacity=1 ] (77.14,44.46) -- (84.85,49.44) -- (76.88,53.64) -- cycle ;
\draw    (441,2) -- (390,92) ;
\draw    (492,92) -- (390,92) ;
\draw  [dash pattern={on 3.75pt off 3pt on 7.5pt off 1.5pt}]  (491.5,91) -- (440.5,1) ;
\draw  [fill={rgb, 255:red, 255; green, 255; blue, 255 }  ,fill opacity=1 ] (464.14,42.46) -- (471.85,47.44) -- (463.88,51.64) -- cycle ;

\draw (171,19.4) node [anchor=north west][inner sep=0.75pt]    {$D_{f}^g$};
\draw (-5,29.4) node [anchor=north west][inner sep=0.75pt]    {$l_{1}$};
\draw (48,105.4) node [anchor=north west][inner sep=0.75pt]    {$l_{2}$};
\draw (45,45.4) node [anchor=north west][inner sep=0.75pt]    {$f$};
\draw (367,29.4) node [anchor=north west][inner sep=0.75pt]    {$P_{l_{1}}^{g_{1}}$};
\draw (429,105.4) node [anchor=north west][inner sep=0.75pt]    {$P_{l_{2}}^{g_{2}}$};
\draw (285,12.4) node [anchor=north west][inner sep=0.75pt]  [font=\LARGE]  {$\sum$};

\draw (263,58.4) node [anchor=north west][inner sep=0.75pt]  [font=\tiny]  {$gg_{1}^{-1} g_{2}=0$};
\end{tikzpicture}
    \caption{The action of the operator \(D_{f}^g\). For the face \(f\), the edge represented by the dashed line is the one that does not belong to the region \(\Lambda\).}
    \label{facedraw3}
\end{figure}
\end{definition}
Now, we will consider two sets of operators that generate \(\mathfrak{C}(I)\). First, for each \(g \in G\) and each face \(f\) such that \(\partial f \cap I \neq \emptyset\), we consider the operator \(D_{f}^g\). On the left of Figure \ref{algebraC}, we can see an example of such operator.

For the second set of operators,  we consider the operators \(C_{v}^g\), for each \(g \in G\) and each vertex \(v \) that satisfies the following condition:
\begin{equation}\label{condicionv}
    \overline{s(v)} \cap I \neq \emptyset, \text{ and } v \text{ is either incident to two edges of } I \text{ or is in the interior of } \Lambda.
\end{equation} 
The algebra \(\mathfrak{C}(I)\) is defined as follows:
\begin{align}
    \mathfrak{C}(I) = \C^* \{ D_{f}^{g_1}, C_{v}^{g_2} \mid g_1,g_2 \in G,f \text{ such that } \partial f \cap I \neq \emptyset, v \text{ satisfied } \ref{condicionv} \}
\end{align}
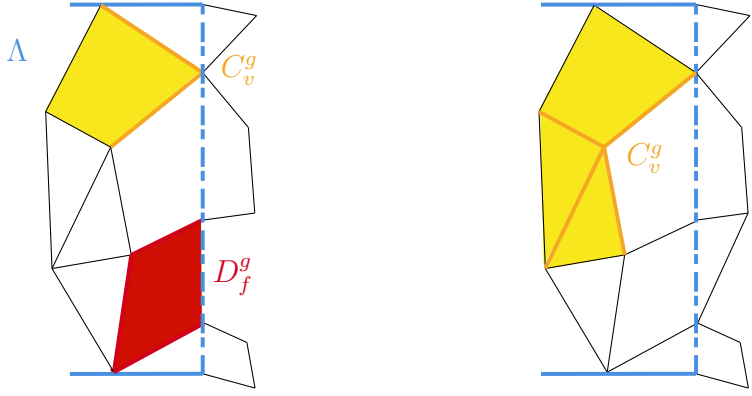
\begin{figure}[h!]
\centering
\begin{tikzpicture}[x=0.6pt,y=0.6pt,yscale=-1,xscale=1]

\draw  [color={rgb, 255:red, 255; green, 255; blue, 255 }  ,draw opacity=1 ][fill={rgb, 255:red, 248; green, 231; blue, 28 }  ,fill opacity=0.29 ] (441.1,62.16) -- (378.07,113.93) -- (336.54,91.22) -- (359.74,17.67) -- cycle ;
\draw  [color={rgb, 255:red, 255; green, 255; blue, 255 }  ,draw opacity=1 ][fill={rgb, 255:red, 255; green, 255; blue, 255 }  ,fill opacity=1 ] (436.95,66.74) -- (374.25,25.89) -- (379,18.6) -- (441.7,59.44) -- cycle ;
\draw  [color={rgb, 255:red, 255; green, 255; blue, 255 }  ,draw opacity=1 ][fill={rgb, 255:red, 248; green, 231; blue, 28 }  ,fill opacity=0.29 ] (339.82,189.83) -- (337.63,92.41) -- (379.99,115.94) -- cycle ;
\draw  [color={rgb, 255:red, 255; green, 255; blue, 255 }  ,draw opacity=1 ][fill={rgb, 255:red, 248; green, 231; blue, 28 }  ,fill opacity=0.29 ] (377.37,115.09) -- (390.54,182.74) -- (342.01,191.21) -- cycle ;
\draw  [color={rgb, 255:red, 255; green, 255; blue, 255 }  ,draw opacity=1 ][fill={rgb, 255:red, 255; green, 255; blue, 255 }  ,fill opacity=1 ] (372,24.4) -- (334.68,94.62) -- (320.21,86.93) -- (357.53,16.71) -- cycle ;
\draw  [color={rgb, 255:red, 255; green, 255; blue, 255 }  ,draw opacity=1 ][fill={rgb, 255:red, 248; green, 231; blue, 28 }  ,fill opacity=0.29 ] (124.08,66.27) -- (58.3,121.45) -- (23.05,98.29) -- (47.6,16.01) -- cycle ;
\draw  [color={rgb, 255:red, 255; green, 255; blue, 255 }  ,draw opacity=1 ][fill={rgb, 255:red, 255; green, 255; blue, 255 }  ,fill opacity=1 ] (62.5,21.57) -- (25.18,91.8) -- (10.71,84.11) -- (48.03,13.88) -- cycle ;
\draw  [color={rgb, 255:red, 255; green, 255; blue, 255 }  ,draw opacity=1 ][fill={rgb, 255:red, 255; green, 255; blue, 255 }  ,fill opacity=1 ] (21.54,89.63) -- (67.11,114.41) -- (61.84,124.1) -- (16.27,99.32) -- cycle ;
\draw  [color={rgb, 255:red, 255; green, 255; blue, 255 }  ,draw opacity=1 ][fill={rgb, 255:red, 208; green, 15; blue, 2 }  ,fill opacity=0.19 ] (124.42,159.7) -- (125.08,226.51) -- (67.48,256.4) -- (44.66,201.08) -- cycle ;
\draw  [color={rgb, 255:red, 255; green, 255; blue, 255 }  ,draw opacity=1 ][fill={rgb, 255:red, 255; green, 255; blue, 255 }  ,fill opacity=1 ] (80.12,184.71) -- (69.12,253.84) -- (29.62,247.56) -- (40.61,178.43) -- cycle ;
\draw    (435.65,161) -- (391.2,182) ;
\draw    (337.2,91.6) -- (341.2,190.6) ;
\draw    (391.2,182) -- (341.2,190.6) ;
\draw    (125.2,24) -- (159.2,31) ;
\draw    (158.2,266) -- (153.05,237.25) ;
\draw    (125.2,257) -- (158.2,266) ;
\draw    (26.2,91.6) -- (30.2,190.6) ;
\draw [color={rgb, 255:red, 74; green, 144; blue, 226 }  ,draw opacity=1 ][line width=1.5]    (41.45,24.1) -- (125.2,24) ;
\draw [color={rgb, 255:red, 74; green, 144; blue, 226 }  ,draw opacity=1 ][line width=1.5]    (41.45,257.1) -- (125.2,257) ;
\draw    (61,24.4) -- (26.2,91.6) ;
\draw    (67.2,114) -- (80.2,182) ;
\draw    (80.2,182) -- (30.2,190.6) ;
\draw [color={rgb, 255:red, 208; green, 2; blue, 27 }  ,draw opacity=1 ][line width=1.5]    (69.2,256) -- (80.2,182) ;
\draw    (69.2,256) -- (30.2,190.6) ;
\draw    (158,155.6) -- (154,101.6) ;
\draw    (125.2,67) -- (154,101.6) ;
\draw [color={rgb, 255:red, 245; green, 166; blue, 35 }  ,draw opacity=1 ][line width=1.5]    (61,24.4) -- (125.2,67) ;
\draw [color={rgb, 255:red, 245; green, 166; blue, 35 }  ,draw opacity=1 ][line width=1.5]    (125.2,67) -- (67.2,114) ;
\draw    (26.2,91.6) -- (67.2,114) ;
\draw    (67.2,114) -- (30.2,190.6) ;
\draw    (125.2,67) -- (159.2,31) ;
\draw [color={rgb, 255:red, 208; green, 2; blue, 27 }  ,draw opacity=1 ][line width=1.5]    (124.64,160.13) -- (80.2,182) ;
\draw    (158,155.6) -- (124.64,160.13) ;
\draw    (126.2,225) -- (153.05,237.25) ;
\draw [color={rgb, 255:red, 208; green, 2; blue, 27 }  ,draw opacity=1 ][line width=1.5]    (126.2,225) -- (69.2,256) ;
\draw    (372,24.4) -- (436.2,67) ;
\draw    (437.2,225) -- (380.2,256) ;
\draw [color={rgb, 255:red, 74; green, 144; blue, 226 }  ,draw opacity=1 ][line width=1.5]    (338.45,24.1) -- (436.2,24) ;
\draw    (372,24.4) -- (337.2,91.6) ;
\draw    (436.2,67) -- (470.2,31) ;
\draw    (436.2,24) -- (470.2,31) ;
\draw    (380.2,256) -- (391.2,182) ;
\draw    (380.2,256) -- (341.2,190.6) ;
\draw    (469.2,266) -- (464.05,237.25) ;
\draw    (436.2,257) -- (469.2,266) ;
\draw [color={rgb, 255:red, 74; green, 144; blue, 226 }  ,draw opacity=1 ][line width=1.5]    (338.45,257.1) -- (436.2,257) ;
\draw    (437.2,225) -- (469,155.6) ;
\draw    (437.2,225) -- (441.56,226.99) -- (464.05,237.25) ;
\draw    (469,155.6) -- (435.64,160.13) ;
\draw [color={rgb, 255:red, 74; green, 144; blue, 226 }  ,draw opacity=1 ][line width=1.5]  [dash pattern={on 3.75pt off 3pt on 7.5pt off 1.5pt}]  (436.2,24) -- (436.2,257) ;
\draw [color={rgb, 255:red, 74; green, 144; blue, 226 }  ,draw opacity=1 ][line width=1.5]  [dash pattern={on 3.75pt off 3pt on 7.5pt off 1.5pt}]  (125.2,24) -- (125.2,257) ;
\draw    (436.2,67) -- (465,101.6) ;
\draw    (469,155.6) -- (465,101.6) ;
\draw [color={rgb, 255:red, 245; green, 166; blue, 35 }  ,draw opacity=1 ][line width=1.5]    (378.2,114) -- (436.2,67) ;
\draw [color={rgb, 255:red, 245; green, 166; blue, 35 }  ,draw opacity=1 ][line width=1.5]    (337.2,91.6) -- (378.2,114) ;
\draw [color={rgb, 255:red, 245; green, 166; blue, 35 }  ,draw opacity=1 ][line width=1.5]    (378.2,114) -- (391.2,182) ;
\draw [color={rgb, 255:red, 245; green, 166; blue, 35 }  ,draw opacity=1 ][line width=1.5]    (378.2,114) -- (341.2,190.6) ;
\draw (390.75,110) node [anchor=north west][inner sep=0.75pt]  [color={rgb, 255:red, 245; green, 166; blue, 35 }  ,opacity=1 ]  {$C_{v}^{g}$};
\draw (135,54.4) node [anchor=north west][inner sep=0.75pt]  [color={rgb, 255:red, 245; green, 166; blue, 35 }  ,opacity=1 ]  {$C_{v}^{g}$};
\draw (130,182.4) node [anchor=north west][inner sep=0.75pt]  [color={rgb, 255:red, 245; green, 166; blue, 35 }  ,opacity=1 ]  {$\textcolor[rgb]{0.82,0.01,0.11}{D_{f}^{g}}$};
\draw (0,44.4) node [anchor=north west][inner sep=0.75pt]  [color={rgb, 255:red, 74; green, 144; blue, 226 }  ,opacity=1 ]  {$\Lambda $};
\end{tikzpicture}
    \caption{Examples of the operator \(C_{v}^g\) and \(D_{f}^g\) for a rough interval.}
    \label{algebraC}
\end{figure}
Thus, in $\mathfrak{C}(I)$, we identify two types of operators \(C_{v}^g\): those where $v \in I$ (as illustrated on the left side of Figure \ref{algebraC}, which shows an example of such an operator), and those where \(C_{v}^g\) are located in the interior of $\Lambda$ (an example of these interior operators is shown on the right side of Figure \ref{algebraC}). The inclusion of these interior operators in the algebra \(\mathfrak{C}(I)\) introduces additional complexities for rough intervals. The main goal of defining the algebra \(\mathfrak{C}(I)\) is to ensure that $\mathfrak{B}\left(\Lambda \Subset \Delta\right) = \mathfrak{C}(I) p_{\Delta}$ holds. However, for this equality to be valid, the operators \(C_{v}^g\) must commute with every vertex operator in $\Delta$. To ensure this commutativity, an additional condition must be imposed on the region $\Lambda$, which results in a restriction on the types of regions that can be considered.

\begin{definition} \label{definicionlarge}
Let $\Lambda$ be a region and $I$ be a rough boundary of $\partial \Lambda$. We say that $\Lambda$ is \emph{sufficiently large} for $I$ if, for each vertex $v$ such that \(C_{v}^g \in \mathfrak{C}(I)\), if $v \in \partial f$ and $f \cap I = \emptyset$, then $f \subset \Lambda$.
\end{definition}

In Figure \ref{suffiectl}, on the left, we see a region $\Lambda$ that is not sufficiently large with respect to the interval $I$. Then, we see a region $\Delta$ such that $\Lambda \Subset \Delta$, but in this case, we cannot guarantee that for all $g, h \in G$, the operator \(C_v^g \in \mathfrak{C}(I)\) commutes with the operator \(A_u^h\). For this reason, we will not work with this type of region. Instead, on the right of Figure \ref{suffiectl}, we see a region $\Lambda'$ that is sufficiently large with respect to the interval $I$.
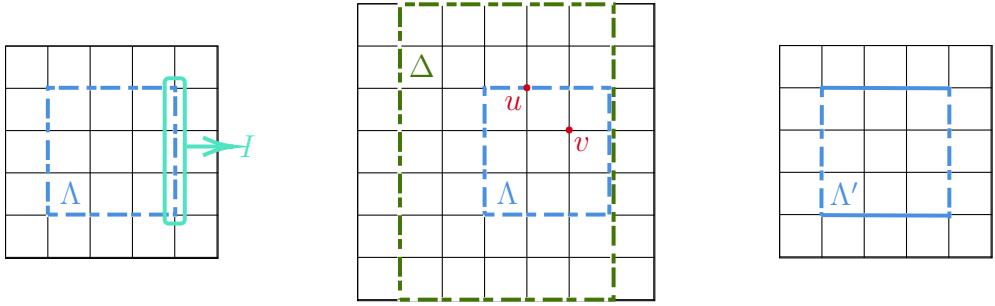
\begin{figure}[h!]
    \centering
\begin{tikzpicture}[x=0.9pt,y=0.9pt,yscale=-1,xscale=1]

\draw  [draw opacity=0] (393,50) -- (493.5,50) -- (493.5,150.5) -- (393,150.5) -- cycle ; \draw   (393,50) -- (393,150.5)(413,50) -- (413,150.5)(433,50) -- (433,150.5)(453,50) -- (453,150.5)(473,50) -- (473,150.5)(493,50) -- (493,150.5) ; \draw   (393,50) -- (493.5,50)(393,70) -- (493.5,70)(393,90) -- (493.5,90)(393,110) -- (493.5,110)(393,130) -- (493.5,130)(393,150) -- (493.5,150) ; \draw    ;
\draw  [color={rgb, 255:red, 255; green, 255; blue, 255 }  ,draw opacity=1 ][line width=1.5]  (413.25,70.25) -- (473.25,70.25) -- (473.25,130.25) -- (413.25,130.25) -- cycle ;
\draw [color={rgb, 255:red, 74; green, 144; blue, 226 }  ,draw opacity=1 ][line width=1.5]  [dash pattern={on 3.75pt off 3pt on 7.5pt off 1.5pt}]  (473.25,69.5) -- (473.25,131.75) ;
\draw  [draw opacity=0] (193.5,30.25) -- (334,30.25) -- (334,170.75) -- (193.5,170.75) -- cycle ; \draw   (193.5,30.25) -- (193.5,170.75)(213.5,30.25) -- (213.5,170.75)(233.5,30.25) -- (233.5,170.75)(253.5,30.25) -- (253.5,170.75)(273.5,30.25) -- (273.5,170.75)(293.5,30.25) -- (293.5,170.75)(313.5,30.25) -- (313.5,170.75)(333.5,30.25) -- (333.5,170.75) ; \draw   (193.5,30.25) -- (334,30.25)(193.5,50.25) -- (334,50.25)(193.5,70.25) -- (334,70.25)(193.5,90.25) -- (334,90.25)(193.5,110.25) -- (334,110.25)(193.5,130.25) -- (334,130.25)(193.5,150.25) -- (334,150.25)(193.5,170.25) -- (334,170.25) ; \draw    ;
\draw  [color={rgb, 255:red, 255; green, 255; blue, 255 }  ,draw opacity=1 ][line width=1.5]  (253.75,70.25) -- (313.75,70.25) -- (313.75,130.25) -- (253.75,130.25) -- cycle ;
\draw  [color={rgb, 255:red, 255; green, 255; blue, 255 }  ,draw opacity=1 ][line width=1.5]  (213.25,30.25) -- (313.75,30.25) -- (313.75,170) -- (213.25,170) -- cycle ;
\draw  [color={rgb, 255:red, 65; green, 117; blue, 5 }  ,draw opacity=1 ][dash pattern={on 3.75pt off 3pt on 7.5pt off 1.5pt}][line width=1.5]  (213.5,30.5) -- (314.5,30.5) -- (314.5,170.25) -- (213.5,170.25) -- cycle ;
\draw  [color={rgb, 255:red, 74; green, 144; blue, 226 }  ,draw opacity=1 ][dash pattern={on 3.75pt off 3pt on 7.5pt off 1.5pt}][line width=1.5]  (253.5,70) -- (312.5,70) -- (312.5,130) -- (253.5,130) -- cycle ;
\draw  [draw opacity=0] (27,50.25) -- (127.5,50.25) -- (127.5,150.75) -- (27,150.75) -- cycle ; \draw   (27,50.25) -- (27,150.75)(47,50.25) -- (47,150.75)(67,50.25) -- (67,150.75)(87,50.25) -- (87,150.75)(107,50.25) -- (107,150.75)(127,50.25) -- (127,150.75) ; \draw   (27,50.25) -- (127.5,50.25)(27,70.25) -- (127.5,70.25)(27,90.25) -- (127.5,90.25)(27,110.25) -- (127.5,110.25)(27,130.25) -- (127.5,130.25)(27,150.25) -- (127.5,150.25) ; \draw    ;
\draw  [color={rgb, 255:red, 255; green, 255; blue, 255 }  ,draw opacity=1 ][line width=1.5]  (47,70.25) -- (107,70.25) -- (107,130.25) -- (47,130.25) -- cycle ;
\draw  [color={rgb, 255:red, 74; green, 144; blue, 226 }  ,draw opacity=1 ][dash pattern={on 3.75pt off 3pt on 7.5pt off 1.5pt}][line width=1.5]  (47,70) -- (107.25,70) -- (107.25,130) -- (47,130) -- cycle ;
\draw  [color={rgb, 255:red, 80; green, 227; blue, 194 }  ,draw opacity=1 ][line width=1.5]  (102.25,67.4) .. controls (102.25,66.35) and (103.1,65.5) .. (104.15,65.5) -- (109.85,65.5) .. controls (110.9,65.5) and (111.75,66.35) .. (111.75,67.4) -- (111.75,132.35) .. controls (111.75,133.4) and (110.9,134.25) .. (109.85,134.25) -- (104.15,134.25) .. controls (103.1,134.25) and (102.25,133.4) .. (102.25,132.35) -- cycle ;
\draw [color={rgb, 255:red, 80; green, 227; blue, 194 }  ,draw opacity=1 ][line width=1.5]    (112.5,97.77) -- (131.6,97.8) ;
\draw [shift={(134.6,97.8)}, rotate = 180.09] [color={rgb, 255:red, 80; green, 227; blue, 194 }  ,draw opacity=1 ][line width=1.5]    (14.21,-4.28) .. controls (9.04,-1.82) and (4.3,-0.39) .. (0,0) .. controls (4.3,0.39) and (9.04,1.82) .. (14.21,4.28)   ;
\draw [color={rgb, 255:red, 74; green, 144; blue, 226 }  ,draw opacity=1 ][line width=1.5]    (412.25,70) -- (473.25,70.25) ;
\draw [color={rgb, 255:red, 74; green, 144; blue, 226 }  ,draw opacity=1 ][line width=1.5]    (412.75,130.25) -- (473.75,130.5) ;
\draw [color={rgb, 255:red, 74; green, 144; blue, 226 }  ,draw opacity=1 ][line width=1.5]  [dash pattern={on 3.75pt off 3pt on 7.5pt off 1.5pt}]  (413,69.75) -- (412.75,130.5) ;
\draw  [color={rgb, 255:red, 208; green, 2; blue, 27 }  ,draw opacity=1 ][fill={rgb, 255:red, 208; green, 2; blue, 27 }  ,fill opacity=1 ] (274.37,68.83) .. controls (275,69.31) and (275.12,70.21) .. (274.64,70.84) .. controls (274.16,71.47) and (273.26,71.59) .. (272.63,71.11) .. controls (272,70.63) and (271.88,69.73) .. (272.36,69.1) .. controls (272.84,68.47) and (273.74,68.35) .. (274.37,68.83) -- cycle ;
\draw  [color={rgb, 255:red, 208; green, 2; blue, 27 }  ,draw opacity=1 ][fill={rgb, 255:red, 208; green, 2; blue, 27 }  ,fill opacity=1 ] (294.37,88.83) .. controls (295,89.31) and (295.12,90.21) .. (294.64,90.84) .. controls (294.16,91.47) and (293.26,91.59) .. (292.63,91.11) .. controls (292,90.63) and (291.88,89.73) .. (292.36,89.1) .. controls (292.84,88.47) and (293.74,88.35) .. (294.37,88.83) -- cycle ;

\draw (136.45,91.4) node [anchor=north west][inner sep=0.75pt]  [color={rgb, 255:red, 80; green, 227; blue, 194 }  ,opacity=1 ]  {$I$};
\draw (261.75,73) node [anchor=north west][inner sep=0.75pt]    {$\textcolor[rgb]{0.82,0.01,0.11}{u}$};
\draw (294.5,92.15) node [anchor=north west][inner sep=0.75pt]    {$\textcolor[rgb]{0.82,0.01,0.11}{v}$};
\draw (51.5,113.4) node [anchor=north west][inner sep=0.75pt]    {$\textcolor[rgb]{0.29,0.56,0.89}{\Lambda }$};
\draw (258,113.4) node [anchor=north west][inner sep=0.75pt]    {$\textcolor[rgb]{0.29,0.56,0.89}{\Lambda }$};
\draw (415.5,113.4) node [anchor=north west][inner sep=0.75pt]    {$\textcolor[rgb]{0.29,0.56,0.89}{\Lambda' }$};
\draw (216.5,53.4) node [anchor=north west][inner sep=0.75pt]    {$\textcolor[rgb]{0.25,0.46,0.02}{\Delta }$};
\end{tikzpicture}
   \caption{Example of a region $\Lambda$ that is not sufficiently large for $I$ and a region $\Lambda'$ that is sufficiently large for $I$.}
    \label{suffiectl}
\end{figure}

Taking into account the previous discussion, we can prove that the Twisted Quantum Double model satisfies the LTO2, LTO3, and LTO4 axioms for rough intervals. For brevity, proofs are omitted as they closely follow the reasoning in Theorems \ref{teoLTO2}, \ref{teoLTO3}, and \ref{teoLTO4}. We now state the following result.

\begin{theorem}
    The Twisted Quantum Double model for a finite group $G$ satisfies the following:
\begin{itemize}
    \item[$\imath)$] If $\Lambda \Subset \Delta$, with $I = \partial \Lambda \cap \partial \Delta$ being a rough interval and $\Lambda$ sufficiently large for $I$, then $p_{\Delta} \mathfrak{A}(\Lambda) p_{\Delta} = \mathfrak{B}\left(\Lambda \Subset \Delta\right) p_{\Delta}$.
    \item[$\imath \imath)$] Whenever $\Lambda_1 \subset \Lambda_2 \Subset \Delta$ with $I = \partial \Lambda_1 \cap \partial \Delta = \partial \Lambda_2 \cap \partial \Delta$ a rough interval and $\Lambda_1, \Lambda_2$ sufficiently large for $I$, then $\mathfrak{B}\left(\Lambda_1 \Subset \Delta\right) = \mathfrak{B}\left(\Lambda_2 \Subset \Delta\right)$.
    \item[$\imath \imath \imath )$] Whenever $\Lambda \Subset \Delta_1 \subset \Delta_2$ with $I = \partial \Lambda \cap \partial \Delta_1 = \partial \Lambda \cap \partial \Delta_2$ a rough interval and $\Lambda$ sufficiently large for $I$, then if $x \in \mathfrak{B}\left(\Lambda \Subset \Delta_1\right)$ with $x p_{\Delta_2} = 0$, then $x = 0$.
\end{itemize}
\end{theorem}

\vspace{0.5cm}
\textbf{Acknowledgement.}  S.C. is partially supported by NSF CCF-2006667 and DMS-2304990, the ORNL-led Quantum Science Center, and ARO MURI. C.G. was partially supported by grant INV-2023-162-2830 from the School of Science at Universidad de los Andes.
D.R. was partially supported by the grant INV-2024-191-3125 from the School of Science of Universidad de los Andes.

\end{document}